\pdfoutput=1
\RequirePackage{ifpdf}
\ifpdf 
\documentclass[pdftex]{sigma}
\else
\documentclass{sigma}
\fi

\numberwithin{equation}{section}

\newtheorem{Theorem}{Theorem}[section]
\newtheorem*{Theorem*}{Theorem}
\newtheorem*{Claim*}{Claim}
\newtheorem{Corollary}[Theorem]{Corollary}
\newtheorem{Lemma}[Theorem]{Lemma}
\newtheorem{Proposition}[Theorem]{Proposition}
 { \theoremstyle{definition}
\newtheorem{Definition}[Theorem]{Definition}

\newtheorem{Example}[Theorem]{Example}
\newtheorem{Remark}[Theorem]{Remark} }

\newtheorem{thmx}{Theorem} 
\newtheorem{lemmaA}{Lemma A\ignorespaces}
\newtheorem{lemmaB}{Lemma B\ignorespaces}

\def\C{\mathbb{C}}
\def\R{\mathbb{R}}
\def\CP{\mathbb{CP}}
\def\Z{\mathbb{Z}}

\def\H{\mathbb{H}}

\def\S{\mathcal{S}}

\def\G{\mathfrak{G}}

\def\CP{\mathbb{C}P}

\def\Q{\mathbb{Q}}

\def\F{\mathcal{F}}

\def\S{\mathcal{S}}

\def\PP2{\mathcal{P}^2}
\def\G{\mathrm{G}}
\def\G1{\mathrm{G}^+_1}
\def\NR1{\mathrm{N}^+_1}

\begin{document}
\allowdisplaybreaks

\renewcommand{\thefootnote}{}

\newcommand{\arXivNumber}{2307.04763}

\renewcommand{\PaperNumber}{101}

\FirstPageHeading

\ShortArticleName{On the Total CR Twist of Transversal Curves in the 3-Sphere}

\ArticleName{On the Total CR Twist of Transversal Curves\\ in the 3-Sphere\footnote{This paper is a~contribution to the Special Issue on Symmetry, Invariants, and their Applications in honor of Peter J.~Olver. The~full collection is available at \href{https://www.emis.de/journals/SIGMA/Olver.html}{https://www.emis.de/journals/SIGMA/Olver.html}}}

\Author{Emilio MUSSO~$^{\rm a}$ and Lorenzo NICOLODI~$^{\rm b}$}

\AuthorNameForHeading{E.~Musso and L.~Nicolodi}

\Address{$^{\rm a)}$~Dipartimento di Scienze Matematiche, Politecnico di Torino,\\
\hphantom{$^{\rm a)}$}~Corso Duca degli Abruzzi 24, I-10129 Torino, Italy}
\EmailD{\href{mailto:emilio.musso@polito.it}{emilio.musso@polito.it}}

\Address{$^{\rm b)}$~Dipartimento di Scienze Matematiche, Fisiche e Informatiche, Universit\`a di Parma,\\
\hphantom{$^{\rm b)}$}~Parco Area delle Scienze 53/A, I-43124 Parma, Italy}
\EmailD{\href{mailto:lorenzo.nicolodi@unipr.it}{lorenzo.nicolodi@unipr.it}}

\ArticleDates{Received July 11, 2023, in final form November 26, 2023; Published online December 21, 2023}

\Abstract{We investigate the total CR twist functional on transversal curves in the standard CR 3-sphere $\mathrm S^3 \subset \mathbb C^2$. The question of the integration by quadratures of the critical curves and the problem of existence and properties of closed critical curves are addressed. A~procedure for the explicit integration of general critical curves is provided and a characterization of closed curves within a specific class of general critical curves is given. Experimental evidence of the existence of infinite countably many closed critical curves is provided.}

\Keywords{CR 3-sphere; transversal curves; CR invariants; total CR twist; Griffiths' formalism; Lax formulation of E-L equations; integration by quadratures; closed critical curves}

\Classification{53C50; 53C42; 53A10}

\begin{flushright}
\begin{minipage}{60mm}
\it Dedicated to Peter Olver\\ on the occasion of his 70th birthday
 \end{minipage}
\end{flushright}

\renewcommand{\thefootnote}{\arabic{footnote}}
\setcounter{footnote}{0}

\section{Introduction}

The present paper finds its inspiration and theoretical framework in
the subjects of moving frames, differential invariants, and invariant variational problems,
three of the many research topics to which Peter Olver has made
lasting contributions.
Among the many publications of
Peter Olver dedicated to these subjects, we like to mention
\cite{FelsOlver1, FelsOlver2,Olver-book1,Olver-book2}
as the ones that most influenced
our research activity.

More specifically, in this paper we further develop some of the themes
considered in~\cite{M,MNS-Kharkiv,MS}
concerning the Cauchy--Riemann (CR) geometry of transversal and Legendrian curves in the 3-sphere.
In three dimensions, a CR structure on a manifold is defined by an oriented contact distribution
equipped with a complex structure.
While the automorphism group of a contact manifold is infinite dimensional,
that of a CR threefold is finite dimensional and of dimension less or equal than eight~\cite{Cartan1932-2,Cartan1932,ChMo1974}.
The maximally symmetric CR threefold is the 3-sphere ${\mathrm S}^3$, realized as a real hyperquadric of $\mathbb{CP}^2$
acted upon transitively by the Lie group $G \cong \mathrm{SU}(2,1)$.
This homogeneous model allows the application of
differential-geometric techniques to the study of transversal and Legendrian curves in ${\mathrm S}^3$.
Since the seminal work of Bennequin~\cite{Benn1983}, the study of the topological properties of transversal and Legendrian
knots in 3-dimensional contact manifolds has been an important area of research
(see, for instance,~\cite{Eliash1993,Et3,Etn1999, Et2,EtHo,FuTa1997} and the literature therein).
Another reason of interest for 3-dimensional contact geometry
comes from its applications to neuroscience.
In fact, as shown by Hoffman~\cite{Ho}, the visual cortex can be modeled as a bundle equipped with a contact structure.
For more details, the interested reader is referred to the monograph
\cite[Section~5]{Pe}.
Recently, the CR geometry of Legendrian and transversal curves in ${\mathrm S}^3$ has also found
interesting applications in the framework of integrable system~\cite{CI}.

Let us begin by recalling some results from the CR geometry
of transversal curves in $\mathrm{S}^3$.
According to~\cite{MNS-Kharkiv}, away from CR inflection points,
a curve transversal to the contact distribution of $\mathrm{S}^3$
can be parametrized by a natural \textit{pseudoconformal} parameter~$s$ and in this parametrization it is
uniquely determined, up to CR automorphisms,
by two local CR invariants: the CR \textit{bending}~$\kappa$ and the CR \textit{twist} $\tau$.
This was achieved by developing the method of moving frames and by constructing a canonical
frame field along \textit{generic}\footnote{I.e., with no CR inflection points.} transversal curves.
Moreover, for closed transversal curves, we defined three discrete global invariants, namely,
the wave number, the CR spin, and the CR turning number.
Next, we investigated the total strain functional, defined
by integrating the strain element~${\rm d}s$.
We proved that the corresponding critical curves have
constant bending and twist, and hence
arise as orbits of 1-parameter groups of CR automorphisms. Finally,
closed critical curves
are shown to be transversal positive torus knots with maximal Bennequin number.

In the present paper, we consider
the CR invariant variational problem for generic transversal curves in $\mathrm{S}^3$ defined by the
\textit{total CR twist functional},
\[
\mathcal{W}(\gamma) = \int_\gamma \tau \,{\rm d}s.
\]
Our purpose is to address both the question of the explicit integration of critical curves and
the problem of existence and properties of closed critical curves of $\mathcal{W}$.

We now give a brief outline of the content and results of this paper.
In Section~\ref{s:preliminaries}, we shortly describe the standard CR
structure of the 3-sphere $\mathrm{S}^3$,
viewed as a homogeneous space of the
group $G$, and collect some preliminary material. We then recall the basic facts
about the CR geometry of transversal curves in $\mathrm{S}^3$ as
developed in~\cite{MNS-Kharkiv} (see the description above).
Moreover, besides the already mentioned discrete global invariants for a closed transversal curve,
we introduce a fourth global invariant, the \textit{trace} of the curve with respect to a spacelike line.

In Section~\ref{s:total-twist},
we apply the method of moving frames
and the Griffiths approach to the calculus of variations~\cite{Gr,Hsu, GM} to compute the Euler--Lagrange equations
of the total CR twist functional.
We construct the momentum space of the corresponding variational problem
and find a Lax pair formulation for the Euler--Lagrange equations satisfied by the critical curves.
This is the content of Theorem~\ref{thmA}, the first main result of the paper, whose proof occupies
the whole Section~\ref{s:total-twist}.
As a consequence of Theorem~\ref{thmA},
to each critical curve we associate a \textit{momentum} operator, which is
a fixed element of the $G$-module $\mathfrak{h}$
of traceless selfadjoint endomorphisms of~$\mathbb{C}^{2,1}$.
From the conservation of the momentum along a critical curve,
we derive two conservation laws, involving two
real parameters $c_1$ and $c_2$.
The pair $\mathbf{c}=(c_1,c_2)$ is referred to as the \textit{modulus} of the critical curve.

In Section~\ref{s:twist},
we introduce the phase type of the modulus of a critical curve. We then define the phase curve
of a
given modulus and the associated notion of signature of a critical curve with that given modulus.
For a generic modulus $\mathbf{c}$, the phase type of $\mathbf{c}$ refers to the properties of the roots
of the quintic polynomial in principal form given by
\[{\rm P}_{\mathbf{c}}(x)=x^5+\frac{3}{2}c_2x^2+27c_1x-\frac{27}{2}c_1^2 .\]
The phase curve of
the modulus $\mathbf{c}$ is the real algebraic curve
defined by the equation ${y^2={\rm P}_{\mathbf{c}}(x)}$.
The signature of a critical curve $\gamma$ with modulus $\mathbf{c}$ and nonconstant twist
provides a parametrization of the connected components of the
phase curve of $\mathbf{c}$ by the twist of $\gamma$. Importantly, the periodicity of the twist of $\gamma$
amounts to the compactness of the image of the signature of $\gamma$.
This will play a role in Sections~\ref{s:integrability}
and~\ref{s:heuristic}, where the closedness question for critical curves is addressed.
Using the Klein formulae for the icosahedral solutions of the quintic~\cite{K,Na,Tr},
the roots of~${\rm P}_{\mathbf{c}}$
can be evaluated in terms of hypergeometric functions. As a byproduct, we show that the twist
and the bending of a
critical curve can be obtained by inverting incomplete hyperelliptic integrals of the first kind.
We further specialize our analysis by introducing the orbit type of the modulus $\mathbf{c}$ of a
critical curve $\gamma$. The orbit type of $\mathbf{c}$ refers to the spectral properties of the momentum associated to $\gamma$.
Depending on the phase type, the number of connected components of the phase curves,
and the orbit type, the critical curves are then divided into twelve classes.
The critical curves of only three of these classes have periodic twist.

In Section~\ref{s:integrability},
we show that a \textit{general} critical curve (cf.\ Definition~\ref{def:general})
can be integrated by quadratures using the momentum of the curve. This is the content
of Theorem~\ref{thmss:integrability1}, the second main result of the paper.
Theorem~\ref{thmss:integrability1} is then specialized
to one of the twelve classes of critical curves, the class characterized by the compactness of the
connected component of the phase
curve and by the existence of three distinct real
{eigenvalues} of the momentum.
Theorem~\ref{Theorem1ss:integrability2}, the third main result,
shows that the critical curves of this specific class can be explicitly written by inverting hyperelliptic integrals
of the first and third kind.
We then examine the closure conditions and prove that a critical curve in this class is closed if and only if certain
complete hyperelliptic integrals depending on the {modulus} of the curve are rational.
Finally, the relations between these rational numbers and the global CR invariants mentioned above
are discussed.

In the last section, Section~\ref{s:heuristic}, we develop
convincing heuristic and numerical arguments to support the claim that there exist
infinite countably many distinct congruence classes of closed critical curves.
These curves are uniquely determined by the four discrete geometric invariants: the wave number,
the CR spin, the CR turning number, and the trace
with respect to the spacelike $\lambda_1$-eigenspace of the momentum.
Using numerical tools, we construct and illustrate explicit examples of approximately closed critical curves.

\section{Preliminaries}\label{s:preliminaries}

\subsection{The standard CR structure on the 3-sphere}

Let $\C^{2,1}$ denote $\C^3$ with the indefinite Hermitian scalar product of signature $(2,1)$ given by
\begin{equation}\label{hp}
\langle \mathbf{z},\mathbf{w} \rangle
= {^t \overline{\mathbf{z}}} \mathbf{h} \mathbf{w},
\qquad {\mathbf h}= (h_{ij})=
\begin{pmatrix}
\hphantom{-}0& 0& {\rm i}\\
\hphantom{-}0& 1&0\\
-{\rm i}&0&0
\end{pmatrix}.
 \end{equation}
Following common terminology in pseudo-Riemannian geometry, a nonzero vector ${\bf z}\in\C^{2,1}$
is \emph{spacelike}, \emph{timelike} or \emph{lightlike}, depending on whether $\langle {\bf z}, {\bf z}\rangle$
is positive, negative or zero.
By $\mathcal{N}$ we denote the \textit{nullcone}, i.e., the set of all lightlike vectors.

Let $\mathcal{S}= \mathbb{P}(\mathcal{N})$ be the real hypersurface in $\CP^2$ defined by
\begin{equation*}
\mathcal S =\big\{[{\bf z}] \in \CP^2 \mid \langle \mathbf{z},\mathbf{z} \rangle
= {\rm i}(\overline{z}_1z_3-\overline{z}_3z_1)+\overline{z}_2z_2=0\big\}.
 \end{equation*}
The restriction of the affine chart
\begin{equation*}
 s\colon \ \C^2 \ni (z_1,z_2) \longmapsto
 \left[{^t\!\bigg(}\frac{1+z_1}{2},{\rm i}\frac{z_2}{\sqrt{2}},{\rm i}\frac{1-z_1}{2}\bigg)\right]\in \mathcal S\subset \CP^2
 \end{equation*}
to the unit sphere ${\mathrm S}^3$ of $\C^2$ defines a smooth diffeomorphism
between ${\mathrm S}^3$ and $\mathcal{S}$.
For each $p=[{\bf z}]\in \mathcal{S}$, the differential $(1,0)$-form
\begin{equation*}
 \tilde\zeta\big|_p =
 - \frac{{\rm i} \langle \bf{z},d{\bf z}\rangle}{\overline{{\bf z}} \, {^t\!\bf z}} \bigg|_p \in \Omega^{1,0}\big(\CP^2\big)\big|_p
 \end{equation*}
is well defined. In addition, the null space of the imaginary part of $\tilde\zeta|_p$ is
${{T}(\mathcal{S})}|_p$, namely the tangent space of $\mathcal{S}$ at $p$.
Thus, the restriction of $\tilde\zeta$ to ${{T}(\mathcal{S})}$ is a real-valued 1-form $\zeta\in\Omega^1(\mathcal{S})$.
Since the pullback of $\zeta$ by the diffeomorphism $s\colon \mathrm{S}^3\to\mathcal{S}$ is the standard contact
form ${\rm i}\overline{{\bf z}}\cdot d{{\bf z}}|_{\mathrm{S}^3}$ of $\mathrm{S}^3$, then $\zeta$ is a contact
form whose contact distribution $\mathcal{D}$ is, by construction, a complex subbundle
of ${{T}\big(\CP^2\big)\big|_{\mathcal{S}}}$. Therefore, $\mathcal{D}$ inherits
from ${{T}\big(\CP^2\big)\big|_{\mathcal{S}}}$ a complex structure $J$. This defines a CR structure on $\mathcal{S}$.

Let $\mathbf{e}_1$, $\mathbf{e}_2$, $\mathbf{e}_3$ denote the standard basis of $\mathbb C^3$.
Consider ${P}_{0}=[\mathbf{e}_1]
\in {\mathcal S}$ and ${P}_{\infty}=[\mathbf{e}_3]
\in {\mathcal S}$ as the origin
and the point at infinity of $\S$. Then, $\dot{\S}:= \S\setminus\{P_\infty\}$ can be identified with
Euclidean 3-space with its standard contact structure
${\rm d}z - y{\rm d}x + x{\rm d}y$
by means of the {\em Heisenberg projection}\footnote{This map is the analogue of the stereographic projection in M\"obius (conformal) geometry.}
\[
 \pi_H\colon \ \dot{\S}\ni [{\bf z}]\longmapsto {^t\!\left(\operatorname{Re}(z_2/z_1),\operatorname{Im}(z_2/z_1),\operatorname{Re}(z_3/z_1)\right)}\in \R^3.
\]
The inverse of the Heisenberg projection is the {\em Heisenberg chart}
\[
 j_H \colon \ \R^3 \ni {^t\!(}x, y, z) \longmapsto \bigg[{^t\!\big(}1,x+{\rm i}y,z+\frac{{\rm i}}{2}\big(x^2+y^2\big)\big)\bigg]\in\dot{\S}.
\]
The Heisenberg chart can be lifted to a map
 whose image is a 3-dimensional closed subgroup~$\H^3$ of ${G}$,
 which is isomorphic to the 3-dimensional {\em Heisenberg group}~\cite{MNS-Kharkiv}.

Let $G$ be the special pseudo-unitary group of~\eqref{hp},
i.e., the 8-dimensional Lie group
of unimodular complex $3 \times 3$ matrices
preserving~\eqref{hp},
\begin{equation*}
G = \big\{ A \in \mathrm{SL}(3,\mathbb C) \mid {^t\!\bar{A}}\mathbf h A = \mathbf h\big\} \cong \mathrm{SU}(2,1),
\end{equation*}
and let $\mathfrak g$ denote the Lie algebra of $G$,
\begin{equation*}
\mathfrak g = \big\{ X \in \mathfrak{sl}(3,\mathbb C) \mid {^t\!\bar{X}}\mathbf h + \mathbf h X = 0\big\}.
\end{equation*}
The Maurer--Cartan form of the group $G$ takes the form
\begin{equation*}
\vartheta = A^{-1}{\rm d}A =\begin{pmatrix}
\alpha_1^1 +{\rm i} \beta_1^1 & -{\rm i} \alpha_3^2 - \beta_3^2 & \alpha_3^1\\
\alpha_1^2 +{\rm i} \beta_1^2 &-2{\rm i} \beta_1^1 & \alpha_3^2 +{\rm i}\beta_3^2 \\
 \alpha_1^3 &{\rm i} \alpha_1^2 + \beta_1^2& -\alpha_1^1 +{\rm i} \beta_1^1\\
\end{pmatrix},
\end{equation*}
where the 1-forms
$
 \big(\alpha_1^1 , \beta_1^1 , \alpha_1^2 , \beta_1^2 ,\alpha_1^3 , \alpha_3^2 , \beta_3^2 , \alpha_3^1\big)
 $
form a basis of the dual Lie algebra ${\mathfrak g}^*$.
The center of $G$ is $Z = \big\{\varpi I_3\mid \varpi \in \mathbb C, \,\varpi^3 =1\big\}\cong \mathbb Z_3$, where
$I_3$ denotes the $3\times 3$ identity matrix. Let~$[G]$ denote the quotient Lie
group $G/Z$ and for $A\in G$ let $[A]$ denote its equivalence class in $[G]$.
Thus $[A] = [B]$ if and only if $B = \varpi A$, for some cube root
of unity $\varpi$.
For any $A\in G$, the column vectors $(A_1, A_2, A_3)$ of $A$ form a basis of $\C^{2,1}$
satisfying $\langle A_i, A_j \rangle = h_{ij}$ and $\mathrm{det}(A_1, A_2, A_3) =1$. Such a basis is
referred to as a
{\em lightcone basis}.
On the other hand, a~basis $(\mathbf{u}_1, \mathbf{u}_2, \mathbf{u}_3)$ of $\C^{2,1}$,
such that $\mathrm{det}(\mathbf{u}_1, \mathbf{u}_2, \mathbf{u}_3) =1$ and
$\langle \mathbf{u}_i, \mathbf{u}_j \rangle = \delta_{ij}\epsilon_j$, where
$\epsilon_1 = -1$, $\epsilon_2 = \epsilon_3 = 1$,
is referred to as a unimodular {\em pseudo-unitary basis}.

The group $G$ acts transitively and almost effectively on the left of
$\S$ by
\begin{equation*}
 A[{\bf z}] = [A {\bf z}], \qquad \forall \, A\in {G}, \  \forall \, [{\bf z}] \in\S.
 \end{equation*}
This action descends to an effective action of $[{G}]={ G}/{Z}$ on~$\S$.
It is a classical result of E.~Cartan~\cite{Cartan1932-2,Cartan1932,ChMo1974} that $[{G}]$
is the group of CR automorphisms
of $\S$.

If we choose $[\mathbf{e}_1]=\big[ {^t\!(}1,0,0)\big]\in \mathcal{S}$ as an origin of $\mathcal{S}$,
the natural projection
\begin{equation*}
 \pi_\S\colon \ {G}\ni A \mapsto A[\mathbf{e}_1] =[A_1]\in\S
 \end{equation*}
makes ${G}$ into a (trivial) principal fiber bundle with structure group
\begin{equation*}
 {G}_0=\left\{ A\in {G} \mid A[\mathbf{e}_1]=[\mathbf{e}_1] \right\}.
 \end{equation*}
The elements of ${G}_0$ consist of all $3\times3$ unimodular matrices of the form
\begin{equation}\label{gauge}
{X}(\rho,\theta,v,r)=
\begin{pmatrix}
\rho {\rm e}^{{\rm i}\theta} & -{\rm i}\rho {\rm e}^{-{\rm i}\theta}\bar {v} & {\rm e}^{{\rm i}\theta}(r-\frac{{\rm i}}{2}\rho |v|^2) \\
 0 & {\rm e}^{-2{\rm i}\theta} & v \\
 0 & 0 & \rho^{-1}{\rm e}^{{\rm i}\theta}, \\
 \end{pmatrix},
 \end{equation}
where $v\in \C$, $r \in \R$, $0\leq\theta <2\pi$, and $\rho>0$.

\begin{Remark}
The left-invariant 1-forms $\alpha_1^2$, $\beta_1^2$, $\alpha_1^3$ are linearly independent and generate
the semi-basic 1-forms for the projection $\pi_{\mathcal S}\colon G \to \mathcal S$. So, if $s\colon U\subseteq\S\to{G}$ is
a local cross section of $\pi_\S$,
then $\big(s^*\alpha_1^3 , s^*\alpha_1^2, s^* \beta_1^2 \big)$
defines a coframe on $U$ and $s^*\alpha_1^3 $ is a positive contact form.\end{Remark}

\subsection{Transversal curves}

\begin{Definition}
Let $\gamma \colon J \to\S$ be a smooth immersed curve. We say that $\gamma$ is \emph{transversal} (to the contact
distribution $\mathcal{D})$ if the tangent vector ${\gamma'}(t) \not\in \mathcal{D}|_{\gamma(t)}$, for every $t \in J$.
The parametrization~$\gamma$ is said to be
\emph{positive} if $\zeta({\gamma'}(t) ) > 0$, for every $t$ and for every positive contact form compatible
with the CR structure.
From now on, we assume that the parametrization of a transversal curve is positive.
\end{Definition}

\begin{Definition}
Let $\gamma\colon J \to\S$ be a smooth curve.
A \emph{lift} of $\gamma$ is a map $\Gamma\colon J \to \mathcal{N}$ into the nullcone $\mathcal{N}\subset\C^{2,1}$,
such that $\gamma(t) = [\Gamma(t)]$, for every $t \in J$.
 \end{Definition}

If $\Gamma$ is a lift, any other lift is given by $r\Gamma$, where $r$ is a smooth complex-valued function,
such that $r(t)\neq 0$, for every $t \in J$.
 From the definition of the contact distribution, we have the following.

 \begin{Proposition}
 A parametrized curve $\gamma \colon J \to\S$ is transversal
 and positively oriented if and only if
 $- {\rm i}\langle \Gamma,\Gamma'\rangle|_t > 0$, for every $t\in J$ and for every lift $\Gamma$.
 \end{Proposition}

\begin{Definition}
A \emph{frame field along $\gamma \colon J \to \S$} is a smooth map $A \colon J \to {G}$
such that ${\pi_\S \circ A = \gamma}$. Since the fibration $\pi_{\mathcal{S}}$
is trivial, there exist frame fields along every transversal curve.
If $A=(A_1,A_2,A_3)$ is a frame field along $\gamma$, $A_1$ is a lift of $\gamma$.
\end{Definition}

Let $A$ be a frame field along $\gamma$. Then
\[
A^{-1}A'=
\begin{pmatrix}
a_1^1 +{\rm i} b_1^1 & -{\rm i} a_3^2 - b_3^2 & a_3^1\\
a_1^2 +{\rm i} b_1^2 &-2{\rm i} b_1^1 & a_3^2 +{\rm i}b_3^2 \\
 a_1^3 &{\rm i} a_1^2 + b_1^2& -a_1^1 +{\rm i} b_1^1
 \end{pmatrix},
\]
where $a_1^3$ is a strictly positive real-valued function. Any other frame field along $\gamma$
is given by $\tilde A = AX(\rho,\theta,v,r)$, where
$\rho$ ($\rho>0$), $\theta$, $r \colon J \to \mathbb R$, $v=p+{\rm i}q \colon J\to \mathbb C$ are smooth functions
and $X(\rho,\theta,v,r)\colon J \to {G}_0$ is as in~\eqref{gauge}.
If we let
\[
\tilde A^{-1}\tilde A'=
\begin{pmatrix}
\tilde a_1^1 +{\rm i} \tilde b_1^1 & -{\rm i} \tilde a_3^2 - \tilde b_3^2 & \tilde a_3^1\\
\tilde a_1^2 +{\rm i} \tilde b_1^2 &-2{\rm i} \tilde b_1^1 & \tilde a_3^2 +{\rm i} \tilde b_3^2 \\
\tilde a_1^3 &{\rm i} \tilde a_1^2 + \tilde b_1^2& -\tilde a_1^1 +{\rm i} \tilde b_1^1
\end{pmatrix},
\]
then
\begin{equation*}
\tilde A^{-1}\tilde A'= X^{ -1}A^{-1} A'X + X^{ -1}X',
\end{equation*}
which implies
\[
 \tilde a_1^3= \rho^2 a_1^3,\qquad
 \tilde a_1^2 +{\rm i} \tilde b_1^2 = \rho {\rm e}^{3{\rm i}\theta} \big(a_1^2 +{\rm i} b_1^2 \big) - \rho^2 {\rm e}^{2{\rm i}\theta} (p+{\rm i}q) a_1^3.
 \]
From this it follows
that along any parametrized transversal curve there exists
a frame field $A$ for which $a_1^2 +{\rm i} b_1^2=0$.
Such a frame field
is said to be of \textit{first order}.
\begin{Definition}
Let $\Gamma$ be a lift of a transversal curve $\gamma\colon J \to \S$.
If
\[
 \det(\Gamma,\Gamma',\Gamma'')|_{t_0} = 0,
\]
for some $t_0\in J$, then $\gamma({t_0})$ is called a \emph{CR inflection point}.
The notion of CR inflection point is independent of the lift $\Gamma$.
A transversal curve with no CR inflection points is said to be \emph{generic}.
The notion of a CR inflection point is invariant under reparametrizations and
under the action of the group of CR automorphisms.
\end{Definition}

\begin{Remark}
If $A \colon J \to G$ is a frame field along a transversal curve $\gamma$, then $\gamma({t_0})$ is a~CR inflection point
if $ \det(A_1,A_1',A_1'')\big|_{t_0} = 0$. A transversal curve all of whose points are CR inflection points is
called a \textit{chain}. The notion of chain on a CR manifold goes back to Cartan~\cite{Cartan1932-2, Cartan1932}
(see also~\cite{Jacobo1985} and the literature therein).

 If $\gamma$ is transversal and $\Gamma$ is one of its lifts, then the complex plane $[\Gamma\wedge\Gamma']_{ t}$
 is of type $(1,1)$ and the set of null complex lines contained in $[\Gamma\wedge\Gamma']_{ t}$ is a chain
 which is independent of the choice of the lift $\Gamma$. This chain, denoted by $\mathcal{C}_{ \gamma} |_t$, is
 called the \emph{osculating chain} of $\gamma$ at $\gamma(t)$. By construction, $\mathcal{C}_{ \gamma} |_t$
 is the unique chain passing through~$\gamma(t)$ and tangent to $\gamma$ at the contact point~$\gamma(t)$.
For more details on the {CR}-geometry of transversal curves in the 3-sphere, we refer to~\cite{MNS-Kharkiv}.
{As a basic reference for transversal knots and their topological invariants in the framework of 3-dimensional
contact geometry, we refer to~\cite{Et2} and the literature therein}.
\end{Remark}

\subsection{The canonical frame and the local CR invariants}\label{ss:canonical-frame}

In the following, we will consider generic transversal curves.

\begin{Definition}
Let $\gamma$ be a generic transversal curve. A lift $\Gamma$ of $\gamma$, such that
\[
 \det(\Gamma,\Gamma',\Gamma'')= -1,
\]
is said to be a \emph{Wilczynski lift} (W-lift) of $\gamma$.
 If $\Gamma$ is a Wilczynski lift, any other is given by $\varpi\Gamma$, where $\varpi\in\C$
 is a cube root of unity.
 The function
\[
 a_\gamma = {\rm i}\langle\Gamma,\Gamma'\rangle^{-1}
\]
is smooth, real-valued, and independent of the choice of $\Gamma$.
We call $a_\gamma$ the \emph{strain density} of the parametrized transversal curve $\gamma$.
The linear differential form ${\rm d}s = a_\gamma {\rm d}t$ is called the \emph{infinitesimal strain}.
\end{Definition}

\begin{Proposition}[\cite{MNS-Kharkiv}]\label{2.2.1}
The strain density and the infinitesimal strain are invariant under the action of the CR transformation group.
In addition, if $h\colon I \to J$ is a change of parameter, then the infinitesimal strains ${\rm d}s$ and ${\rm d}\tilde s$
of $\gamma$ and $\tilde \gamma=\gamma\circ h$,
respectively, are related by ${\rm d}\tilde s = h^*({\rm d}s)$.
\end{Proposition}

\begin{proof}
This proof corrects a few misprints contained in the original one.
If $A\in {G}$ and if $\Gamma$ is a Wilczynski lift of $\gamma$, then $\hat{\Gamma} = A\Gamma$ is a
Wilczynski lift of $\hat\gamma = A\gamma$. This implies that $a_\gamma = a_{\hat\gamma}$.
Next, consider a reparametrization $\tilde\gamma=\gamma\circ h$ of $\gamma$. Then, $\Gamma^*=\Gamma\circ h$
is a lift of $\tilde\gamma$, such that
\[
 \det\big(\Gamma^*,(\Gamma^*)' ,(\Gamma^*)'' \big) = -(h')^3.
\]
This implies that $\tilde\Gamma=(h')^{ -1}\Gamma^*$ is a Wilczynski lift of $\tilde\gamma$.
Hence
\[
 \big\langle\tilde\Gamma,\big(\tilde\Gamma\big)'\big\rangle
 =(h')^{ -1}\langle\Gamma,\Gamma'\rangle\circ h.
\]
Therefore, the strain densities of $\gamma$ and $\tilde\gamma$ are related by
$a_{\tilde\gamma}=h'a_\gamma\circ h$.
Consequently, we have
$h^*({\rm d}s)= h'a_\gamma\circ h \, {\rm d}t= a_{\tilde\gamma} {\rm d}t ={\rm d}\tilde s$. \end{proof}

As a straightforward consequence of Proposition~\ref{2.2.1}, we have the following.

\begin{Corollary}
A generic transversal
curve $\gamma$ can be parametrized so that $a_\gamma = 1$.
\end{Corollary}

\begin{Definition}
If $a_\gamma = 1$, we say that $\gamma \colon J \to\S$ is a \emph{natural parametrization},
or a parametrization by the {\em pseudoconformal strain} or {\em pseudoconformal parameter}.
In the following, the {\em natural parameter} will be denoted by $s$.
\end{Definition}

We can state the following.

\begin{Proposition}[\cite{MNS-Kharkiv}]\label{Prop-CF}
Let $\gamma \colon J \to\S$ be a generic transversal curve, pa\-ra\-me\-tri\-zed by the natural parameter. There exists a
(first order) frame field
$\F$ $=$ $(F_1$, $F_2$, $F_3) \colon J \to {G}$ along $\gamma$,
such that $F_1$ is a W-lift and
\begin{equation}\label{MCE-W-frame}
\F^{-1}\F'=
\begin{pmatrix}
{\rm i}\kappa & -{\rm i} & \tau\\
0 &-2{\rm i}\kappa & 1 \\
1 & 0&{\rm i}\kappa\\
 \end{pmatrix}
 = : K_{\kappa,\tau}(s),
 \end{equation}
where $\kappa, \tau \colon J \to \mathbb R$ are smooth functions, called the
\emph{CR bending}
and the \emph{CR twist}, respectively.
 The frame field $\F$ is called a {\em Wilczynski frame}.
If $\F$ is a Wilczynski frame, any other is given by $\varpi\F$,
where $\varpi$ is a cube root of unity.
Thus, there exists a unique frame
field $[\mathcal{F}] \colon J \to [G]$ along $\gamma$, called the {\em canonical frame} of $\gamma$.
\end{Proposition}

\begin{Remark}
Given two smooth functions $\kappa, \tau\colon J\to\R$, there exists a generic transversal
curve $\gamma\colon J\to\S$,
parametrized by the natural parameter,
whose bending is $\kappa$ and whose twist is~$\tau$.
The curve $\gamma$ is unique up to CR automorphisms of~$\S$.
\end{Remark}

\begin{Remark}[cf.~\cite{MNS-Kharkiv}]\quad
\begin{enumerate}\itemsep=0pt
\item[$(1)$] Let $\gamma \colon J \to\S $ be as above and $\F = (F_1, F_2, F_3)\colon J \to {G}$ be a~Wilczynski frame
along $\gamma$. Then,
\[
 \gamma^\# \colon \ J \ni s\mapsto [F_3(s)]_\C\in\S
\]
is an immersed curve, called the \emph{dual} of $\gamma$.
The dual curve is Legendrian (i.e., tangent to the contact distribution)
if and only if $\tau = 0$. Thus, the twist can be viewed as a measure of how the dual curve differs
from being a Legendrian curve.

\item[$(2)$] Generic transversal curves with constant bending and twist have been studied by the authors in~\cite{MNS-Kharkiv}.
In the following we will consider generic transversal curves with nonconstant CR invariant functions.
\end{enumerate}
\end{Remark}

\begin{Remark}
Regarding the CR 3-sphere $\mathrm{S}^3\cong \mathcal{S}$ with its standard pseudo-her\-mi\-tian (PSH) structure $(J,\zeta)$,
Chiu and Ho (cf.~\cite{Chiu-Ho2019})
obtained a complete set of local PSH invariants
for \textit{horizontally regular} curves in $\mathrm{S}^3$
parametrized by the \textit{horizontal arc length} $\mathrm{w}$, namely
the $p$-\textit{curvature} $k_{\textsc{psh}}(\mathrm{w})$ and the $T$-\textit{variation} $\tau_{\textsc{psh}}(\mathrm{w})$.
The canonical PSH frame field
producing
the PSH invariants originates a
CR frame field which can be further adapted to a canonical CR frame following the reduction
procedure developed
in~\cite{MNS-Kharkiv}.
From this one can read the CR invariants.
Thus, in principle, the CR bending $\kappa(s)$ and the CR twist $\tau(s)$ can be
expressed in terms of the PSH invariants $k_{\textsc{psh}}(\mathrm{w}(s))$, $\tau_{\textsc{psh}}(\mathrm{w}(s))$
and their derivatives with respect to $s$.
\end{Remark}
\subsection{Discrete CR invariants of a closed transversal curve}

Referring to~\cite{MNS-Kharkiv,MS}, we briefly recall some CR invariants for closed transversal curves, namely
the notions of wave number, CR spin, and CR turning number (or Maslov index). These invariants will be used in Sections~\ref{s:integrability} and~\ref{s:heuristic}.
The \textit{wave number} is the ratio between the least period $\omega_{\gamma}$ of $\gamma$ and the least period $\omega$ of
the functions $(\kappa,\tau)$. The \textit{CR spin} is the ratio between~$\omega_{\gamma}$ and the least period
of a Wilczynski lift of $\gamma$.
The \textit{CR turning number} is
the degree (winding number) of the map
$F_1-{\rm i}F_3\colon \R/\omega_{\gamma}\Z \cong \mathrm{S}^1 \to \dot{\C}\: = \mathbb{C}\setminus \{0\}$,
where $\mathcal{F}=(F_1,F_2,F_3)$ is a~Wilczynski frame along $\gamma$.

We will also make use of another invariant.

 \begin{Definition}
Let $[{\bf z}]\in \mathbb{CP}^2$ be a spacelike line.
Denote by ${\mathtt C}_{[{\bf z}]}$ the chain of all null lines orthogonal to $[{\bf z}]$, equipped with its positive orientation.
Consider a closed generic transversal curve $\gamma$ with its positive orientation.
Since $\gamma$ is closed and generic, the intersection of $\gamma$ with~${\mathtt C}_{[{\bf z}]}$ is either a finite
set of points, or
the empty set.
The \emph{trace of $\gamma$ with respect to $[{\bf z}]$}, denoted by~${\rm tr}_{[{\mathbf z}]}(\gamma)$,
 is the integer defined as follows: (1) if $\gamma \cap {\mathtt C}_{[{\bf z}]}\neq \varnothing$,
 then ${\rm tr}_{[{\bf z}]}(\gamma)$ counts the number of intersection points of $\gamma$ with ${\mathtt C}_{[{\bf z}]}$
 (since $\gamma$ is not necessarily a simple curve, the intersection points are counted with their multiplicities); (2)
 otherwise, ${\rm tr}_{[{\bf z}]}(\gamma)={\rm Lk}(\gamma, {\mathtt C}_{[{\bf z}]})$, the linking number of
 $\gamma$ with ${\mathtt C}_{[\mathbf{z}]}$. The trace of $\gamma$ is a $G$-equivariant map,
 that is, ${\rm tr}_{[{\bf z}]}(\gamma) = {\rm tr}_{{A}[{\bf z}]}({A}\gamma)$, for every ${A}\in G$.
 \end{Definition}

\section{The total CR twist functional}\label{s:total-twist}

Let $\mathfrak{T}$ be the space of generic transversal curves in $\mathcal{S}$, parametrized by the natural parameter.
We consider the \textit{total CR twist functional} $\mathcal{W} \colon \mathfrak{T} \to \mathbb R$, defined by
\begin{equation*}
 \mathcal{W}[\gamma] = \int_{J_\gamma} { \tau_\gamma  \eta_\gamma}  ,
 \end{equation*}
where ${J_\gamma}$ is the domain of definition of the transversal curve $\gamma$, $\tau_\gamma$
is its twist, and $\eta_\gamma= {\rm d}s_\gamma$ is the infinitesimal strain of $\gamma$ (cf.\ Section~\ref{ss:canonical-frame}).

A curve $\gamma \in \mathfrak{T}$ is said to be a \textit{critical curve}
in $\mathcal{S}$ if it is a critical point of
$\mathcal{W}$ when one considers compactly
supported variations through generic transversal curves.

The main result of this section is the following.

\begin{thmx}\label{thmA}
Let $\gamma\colon J \to \S$ be a generic transversal curve parametrized by the natural parameter.
Then, $\gamma$
is a critical curve
if and only if~\footnote{As usual, we write $[L, K] = LK - KL$ for the commutator of $L$ and $K$.}
\begin{equation}\label{Lax-eqn}
L' (s)=
[L (s), K_{\kappa,\tau}(s)] ,
 \end{equation}
where
\begin{equation}\label{obser-L}
 L = \begin{pmatrix}
 0 & {\rm i} \tau'+ 3(1 - \tau\kappa) & 2 {\rm i} \tau\\
 \tau & 0 & \tau'+ 3{\rm i}(1 - \tau\kappa)\\
 3{\rm i} & -{\rm i} \tau& 0\\
 \end{pmatrix}
 \end{equation}
and $K_{\kappa,\tau}$ is defined as in~\eqref{MCE-W-frame}.
\end{thmx}

\begin{proof}
The proof of Theorem~\ref{thmA} is organized in four steps and three lemmas.

\textbf{Step 1.} We show that generic transversal curves are in 1-1 correspondence
with the integral curves of a suitable Pfaffian differential system.

Let $\gamma \colon J \to \S$ be a generic transversal curve parametrized by the natural parameter.
According to Proposition~\ref{Prop-CF}, the canonical frame of $\gamma$ defines a unique lift
$[\mathcal{F}] \colon J \to [G]$.
The map
\[
 \mathfrak{f} \colon \ J\ni s\longmapsto ( [\F(s)], \kappa(s), \tau(s))\in [G]\times\R^2
 \]
is referred to as the \textit{extended frame} of $\gamma$.
The product space $M := [G]\times\R^2$ is called the \emph{configuration space}. The coordinates
on $\R^2$ will be denoted by $(\kappa, \tau)$.

With some abuse of notation,
we use $\alpha_1^1$, $\beta_1^1$, $\alpha_1^2$, $\beta_1^2$, $\alpha_1^3$,
$\alpha_3^2$, $\beta_3^2$, $\alpha_3^1$ to denote the entries of the Maurer--Cartan form of $[{G}]$
as well as their pull-backs on the configuration space $M$.
By Proposition~\ref{Prop-CF}, the extended frames of $\gamma$ are the integral curves of the Pfaffian
differential system $(\mathcal A, \eta)$ on $M$ generated by the
linearly independent 1-forms
\begin{alignat*}{5}
 & \mu^1=\alpha_1^2,\qquad && \mu^2=\beta_1^2 ,\qquad && \mu^3= \alpha_3^2- {\alpha^3_1},\qquad && \mu^4=\beta_3^2 , & \\
 & \mu^5=\alpha_1^1 ,\qquad && \mu^6=\beta_1^1-\kappa \alpha_1^3,\qquad && \mu^7 =\alpha_3^1-\tau\alpha_1^3,\qquad  &
\end{alignat*}
with the independence condition $\eta :=\alpha_1^3$.

If $\mathfrak{f} = ( [\F], \kappa, \tau) \colon J \to M$ is an integral curve of $(\mathcal{A}, \eta)$, then
$\gamma = [F_1] \colon J \to \mathcal S$ defines a~generic transversal curve, such that $[\F]$ is its canonical frame,
$\kappa$ its bending and $\tau$ its twist.
Accordingly, the integral curves of $(\mathcal{A},\eta)$ are the extended frames of generic transversal curves
in $\mathcal S$.

Thus, generic transversal curves are in 1-1 correspondence with the integral curves of the Pfaffian system
$(\mathcal{A},\eta)$ on the configuration space $M$.

If we put
\[
 \pi^1={\rm d}\kappa,\qquad \pi^2={\rm d}\tau,
\]
the 1-forms $\big(\eta,\mu^1,\dots,\mu^7,\pi^1,\pi^2\big)$ define an absolute parallelism on $M$.
Exterior differentiation and use of the Maurer--Cartan equations of ${G}$ yield the following structure equations
for the coframe $\big(\eta,\mu^1,\mathellipsis,\mu^7,\pi^1,\pi^2\big)$:
\begin{gather}\label{quadratic1}
\begin{cases}
{\rm d}\eta =2\mu^1\wedge\mu^2+2\mu^5\wedge\eta,\\
{\rm d}\pi^1=d\pi^2=0,\\
\end{cases}
\\
\label{quadratic2}
\begin{cases}
{\rm d}\mu^1=-\mu^1\wedge\mu^5+3\mu^2\wedge\mu^6+\big(3\kappa\mu^2-\mu^3\big)\wedge\eta,	\\
{\rm d}\mu^2=-3\mu^1\wedge\mu^6 - \mu^2\wedge\mu^3 - \big(3\kappa\mu^1+ \mu^4\big)\wedge\eta,	\\
{\rm d}\mu^3=-2\mu^1\wedge\mu^2-\mu^1\wedge\mu^7+\mu^3\wedge\mu^5+3\mu^4\wedge\mu^6  - \big(\tau\mu^1- 3\kappa\mu^4+3\mu^5\big)\wedge\eta,	\\
{\rm d}\mu^4=-\mu^2\wedge\mu^7- 3\mu^3\wedge\mu^6+\mu^4\wedge\mu^5 - \big(\tau\mu^2+ 3\kappa\mu^3-3\mu^6\big)\wedge\eta, \\
{\rm d}\mu^5=-\mu^1\wedge\mu^4+\mu^2\wedge\mu^3 + \big(\mu^2-\mu^7\big)\wedge\eta,	\\
{\rm d}\mu^6=-2\kappa\mu^1\wedge\mu^2-\mu^1\wedge\mu^3-\mu^2\wedge\mu^4 - \big(\mu^1+ 2\kappa\mu^5\big)\wedge\eta - \pi^1\wedge\eta,	\\
{\rm d}\mu^7=-2\tau\mu^1\wedge\mu^2-2\mu^3\wedge\mu^4-2\mu^5\wedge\mu^7  + \big(2\mu^4- 2\tau\mu^5\big)\wedge\eta - \pi^2\wedge\eta.	\\
\end{cases}
\end{gather}

\begin{Remark}\label{r:derived-flags}
From the structure equations it follows that the derived flag of $(\mathcal{A},\eta)$ is given by
$\mathcal{A}_{(4)} \subset \mathcal{A}_{ (3) } \subset \mathcal{A}_{ (2)} \subset \mathcal{A}_{ (1)}$, where
$\mathcal{A}_{(4)} = \{0\}$, $\mathcal{A}_{ (3) }= \text{span}\big\{\mu^1\big\}$, $\mathcal{A}_{ (2)} = \text{span}\big\{\mu^1,\mu^2,\mu^3\big\}$,
$\mathcal{A}_{ (1)}$ $=$ $\text{span}\big\{\mu^1$, $\mu^2$, $\mu^3$, $\mu^4$, $\mu^5\big\}$.
Thus, all the derived systems of $(\mathcal{A},\eta)$ have constant rank. For the notion of derived flag, see~\cite{Gr}.
\end{Remark}

\textbf{Step 2.}
We develop a construction due to Griffiths~\cite{Gr} on an affine subbundle of $T^*(M)$
 (cf.\ also~\cite{Bryant1987,Hsu, GM}) in order to derive the Euler--Lagrange equations.

Let $\mathcal{Z} \subset T^*(M)$ be the affine subbundle defined by
the 1-forms $\mu^1,\dots,\mu^7$ and $\lambda := \tau \eta$, namely
\[
 \mathcal{Z}=\lambda+\text{span}\big\{\mu^1,\dots,\mu^7\big\} \subset T^*(M).
\]
 We call $\mathcal{Z}$ the \textit{phase space} of the Pfaffian system $(\mathcal{A},\eta)$.
The 1-forms $\big(\mu^1,\dots,\mu^7, {\lambda}\big)$ induce a~global affine trivialization of $\mathcal{Z}$, which
may be identified with $M \times \mathbb{R}^7$ by the map
\[
M\times \mathbb{R}^7 \ni (([\mathcal{F}], \kappa, \tau), p_1, \dots, p_7) \longmapsto
{\lambda_{|_{([\mathcal{F}], \kappa, \tau)}} }
+ \sum_{j=1}^7 p_j{\mu^j}_{|_{([\mathcal{F}], \kappa, \tau)}} \in \mathcal{Z} ,
\]
where $p_1,\dots, p_7$ are the fiber coordinates of the bundle map $\mathcal{Z} \to M$ with respect to the trivialization.
Under this identification, the restriction to $\mathcal{Z}$ of the Liouville (canonical) 1-form of $T^*(M)$ takes the form
\[
 \xi=\tau  \eta + \sum_{j=1}^7p_j\mu^j .
 \]
Exterior differentiation and use of the quadratic equations~\eqref{quadratic1} and~\eqref{quadratic2} yield
\begin{align*}
 {\rm d}\xi\equiv{}& \pi^2\wedge\eta
+ 2\tau \mu^5 \wedge \eta + \sum_{j=1}^7{\rm d}p_j\wedge\mu^j + p_1\big(3\kappa\mu^2-\mu^3\big)\wedge\eta \\
&{} - p_2\big(3\kappa\mu^1+\mu^4\big)\wedge\eta - p_3\big(\tau\mu^1-3\kappa\mu^4+3\mu^5\big)\wedge\eta \\
 &{} - p_4\big(\tau\mu^2+3\kappa\mu^3-3\mu^6\big)\wedge\eta
 + p_5\big(\mu^2-\mu^7\big)\wedge\eta\\
 &{} - p_6\big(\pi^1+\mu^1+ 2\kappa\mu^5\big)\wedge\eta - p_7\big(\pi^2-2\mu^4+4\tau\mu^5\big)\wedge\eta  ,
\end{align*}
where the sign `$\equiv$' denotes equality modulo the span of $\{ \mu^i\wedge\mu^ j\}_{ i,j = 1,\dots,7}$.

The Cartan system $(\mathcal{C}({\rm d}\xi), \eta)$ of the
2-form ${\rm d}\xi$ is the Pfaffian system on $\mathcal{Z}$ generated by the 1-forms
\[
\left\{X {\lrcorner} \,{\rm d}\xi \mid X \in \mathfrak{X}(\mathcal{Z}) \right\}\subset \Omega^1(\mathcal{Z}),
\]
with independence condition $\eta \neq 0$.

By Step~1, generic transversal curves are in 1-1 correspondence with the integral curves of the
Pfaffian system $(\mathcal{A},\eta)$.

Let $\mathfrak{f}\colon J \to M$ be the extended frame corresponding to the generic transversal curve
$\gamma\colon J \to \S$ parametrized by the natural parameter.
According to Griffiths approach to the calculus of variations (cf.~\cite{Bryant1987, Gr, Hsu, GM}),
if the extended frame $\mathfrak{f}$
admits a lift $y\colon J \to \mathcal{Z}$ to the phase space~$\mathcal{Z}$ which is an integral curve of the Cartan system $(\mathcal{C}({\rm d}\xi), \eta)$,
then $\gamma$ is a critical curve of the total twist functional with respect to compactly supported variations.

As observed by
Bryant~\cite{Bryant1987}, if all the derived systems
of $(\mathcal{A},\eta)$ are of
constant rank, as in the case under discussion (cf.\ Remark~\ref{r:derived-flags}), then the converse is also true.
Hence all extremal trajectories arise as projections of integral curves of the Cartan
system $(\mathcal{C}({\rm d}\xi), \eta)$.

Next, we compute the Cartan system $(\mathcal{C}({\rm d}\xi), \eta)$.
Contracting the 2-form ${\rm d}\xi$ with the vector fields of the tangent frame
\[
 (\partial_\eta,\partial_{ \mu^1},\dots,\partial_{\mu^7 },
 \partial_{ \pi^1},\partial_{ \pi^2},\partial_{ p_1},\mathellipsis,\partial_{ p_7})
\]
on $\mathcal{Z}$, dual to the coframe $\big(\eta, \mu^1,\dots,\mu^7,\pi^1,\pi^2,{\rm d}p_1,\dots,{\rm d}p_7\big)$,
yields the 1-forms
\begin{gather}
\partial_{ p_j}{\lrcorner} \,{\rm d}\xi\equiv\mu^j, \qquad j=1,\dots,7, \nonumber
\\
\label{C-system2}\begin{cases}
-\partial_{ \pi^1}\lrcorner \,{\rm d}\xi\equiv p_6\eta = :\dot\pi_1,\\
-\partial_{ \pi^2}\lrcorner\,{\rm d}\xi\equiv (p_7 -1)\eta =:\dot\pi_2,\\
-\partial_\eta\lrcorner\,{\rm d}\xi\equiv (1 - p_7)\pi^2 =:\dot\eta,
\end{cases}
\\ \nonumber 
\begin{cases}
-\partial_{ \mu^1}\lrcorner\,{\rm d}\xi\equiv {\rm d}p_1 + (3\kappa p_2+\tau p_3+p_6)\eta =: \dot\mu^1,\\
-\partial_{ \mu^2}\lrcorner\,{\rm d}\xi\equiv {\rm d}p_2 - (3\kappa p_1-\tau p_4+p_5)\eta =: \dot\mu^2,\\
-\partial_{ \mu^3}\lrcorner\,{\rm d}\xi\equiv {\rm d}p_3 + (p_1+3\kappa p_4)\eta =: \dot\mu^3,\\
-\partial_{ \mu^4}\lrcorner\,{\rm d}\xi\equiv {\rm d}p_4 + (p_2-3\kappa p_3-2p_7)\eta =:\dot\mu^4,\\
-\partial_{ \mu^5}\lrcorner\,{\rm d}\xi\equiv {\rm d}p_5 - (2\tau - 3p_3- 2\kappa p_6 - 4\tau p_7)\eta=: \dot\mu^5,\\
-\partial_{ \mu^6}\lrcorner\,{\rm d}\xi\equiv {\rm d}p_6-3p_4\eta = : \dot\mu^6,\\
-\partial_{ \mu^7}\lrcorner\,{\rm d}\xi\equiv {\rm d}p_7 + p_5\eta =: \dot\mu^7.
\end{cases}
\end{gather}
We have proved the following.

\begin{lemmaA}
The Cartan system $(\mathcal{C}({\rm d}\xi), \eta)$
is the Pfaffian system on $\mathcal{Z} \cong M\times \mathbb{R}^7$ generated by
the 1-forms
\[
\big\{\mu^1,\dots,\mu^7,\dot\pi_1,\dot\pi_2,\dot\eta,\dot\mu^1,\dots,\dot\mu^7\big\}
\]
and with independence condition $\eta\neq 0$.
\end{lemmaA}

Now, the Cartan system $(\mathcal{C}({\rm d}\xi), \eta)$ is reducible, i.e., there exists a nonempty submanifold
$\mathcal{Y} \subseteq \mathcal{Z}$, called the reduced space, such that: (1) at each point of $\mathcal{Y}$
there exists an integral element
of $(\mathcal{C}({\rm d}\xi), \eta)$ tangent to $\mathcal{Y}$; (2) if $\mathcal{X} \subseteq \mathcal{Z}$ is any other
submanifold with the same property of $\mathcal{Y}$, then $\mathcal{X} \subseteq\mathcal{Y}$.
The reduced space $\mathcal{Y}$ is called the \textit{momentum space} of
the variational problem. Moreover, the restriction of the Cartan system $(\mathcal{C}({\rm d}\xi), \eta)$ to $\mathcal{Y}$
is called the \textit{Euler--Lagrange system}
of the variational problem, and will be denoted by $(\mathcal{J}, \eta)$.

A basic result states that the Pfaffian systems $(\mathcal{C}({\rm d}\xi), \eta)$ and $(\mathcal{J}, \eta)$ have
the same integral curves (cf.~\cite{Gr, GM}).

The system $(\mathcal{J}, \eta)$ can be constructed by an algorithmic procedure (cf.~\cite{Gr}).

\begin{lemmaA}\label{l:EL-system}
The momentum space $\mathcal{Y}$ is the $11$-dimensional submanifold of $\mathcal{Z}$ defined by the
equations
\[
p_7 =1, \qquad p_6= p_5 = p_4=0 , \qquad p_3 = -\frac23\tau, \qquad p_2=2(1-\tau\kappa) .
\]
The Euler--Lagrange system $(\mathcal{J}, \eta)$ is the Pfaffian system
on $\mathcal{Y}\cong M\times \mathbb R$, with independence
condition $\eta\neq 0$, generated by the $1$-forms
\begin{equation}\label{generators-EL-syst}
\begin{cases}
{\mu^1}_{|\mathcal{Y}}, \dots, {\mu^7}_{|\mathcal{Y}}, \\
\sigma_1 = {\rm d}p_1 + 6 \kappa (1-\tau\kappa)\eta -\frac23 \tau^2\eta  ,\\
\sigma_2 = -2\tau {\rm d}\kappa -2\kappa {\rm d}\tau -3kp_1 \eta ,\\
\sigma_3 = - 2{\rm d}\tau + 3 p_1 \eta .\end{cases}
\end{equation}
\end{lemmaA}

\begin{proof}[Proof of Lemma A\ref{l:EL-system}]
Let $V_1({\rm d}\xi) \hookrightarrow \mathbb{P}(T(\mathcal{Z})) \to \mathcal{Z}$ be the totality of 1-dimensional integral elements of
$(\mathcal{C}({\rm d}\xi),\eta)$. In view of~\eqref{C-system2}, we find that
\[
{V_1({\rm d}\xi)}_{|_{(([\mathcal{F}], \kappa, \tau); p_1,\dots,p_7)}} \neq \varnothing \iff p_6 =0, \ p_7 =1.
\]
Thus,
the image $\mathcal{Z}_1 \subset \mathcal{Z}$
of $V_1({\rm d}\xi)$ with respect to the natural projection $V_1({\rm d}\xi) \to \mathcal{Z}$, is given by
\[
 \mathcal{Z}_1= \left\{ (([\mathcal{F}], \kappa, \tau); p_1,\dots,p_7) \in \mathcal{Z} \mid p_6 =0, \, p_7 =1 \right\}.
 \]
Next, the restriction of $\dot\mu^6$ and $\dot\mu^7$ to $\mathcal{Z}_1$ take the form $\dot\mu^6=-3p_4\eta$
and $\dot\mu^7=p_5\eta$.
Thus, the second
reduction
$ \mathcal{Z}_2$ is given by
\[
 \mathcal{Z}_2= \left\{ (([\mathcal{F}], \kappa, \tau); p_1,\dots,p_7) \in \mathcal{Z}_1 \mid p_4 = p_5 = 0 \right\}.
 \]
Considering the restriction of $\dot\mu^4$ and $\dot\mu^5$ to $\mathcal{Z}_2$ yields the equations
\[
 p_2=2(1-\tau\kappa), \qquad p_3 = -\frac23\tau  ,
\]
which define the third reduction
$\mathcal{Z}_3$. Now, the restriction $\mathcal{C}_3({\rm d}\xi)$ to
$\mathcal{Z}_3$ of the Cartan system~$\mathcal{C}({\rm d}\xi)$ is generated by the 1-forms $\mu^1, \dots, \mu^7$
and
\[
\begin{cases}
\sigma_1 = {\rm d}p_1 + 6 \kappa (1-\tau\kappa)\eta -\frac23 \tau^2\eta  ,\\
\sigma_2 = {\rm d}p_2 -3kp_1 \eta = -2\tau {\rm d}\kappa -2\kappa {\rm d}\tau -3kp_1 \eta ,\\
\sigma_3 = - 2{\rm d}\tau + 3 p_1 \eta .
\end{cases}
\]
This implies that there exists an integral element of $V_1({\rm d}\xi)$ over each point of $\mathcal{Z}_3$,
i.e., for each $p\in \mathcal{Z}_3$, ${V_1({\rm d}\xi)}_{|p} \neq \varnothing$. Hence, $\mathcal{Y} := \mathcal{Z}_3$
is the momentum space and $(\mathcal{J}, \eta) := (\mathcal{C}_3({\rm d}\xi), \eta)$ is the
reduced system of
$(\mathcal{C}({\rm d}\xi), \eta)$. \end{proof}

\textbf{Step 3.} We derive the Euler--Lagrange equations.

By the previous discussion, all the extremal trajectories of $\S$ arise as
projections of the
integral curves of the Euler--Lagrange system. If $y\colon J \to \mathcal{Y}$ is an integral curve
of the Euler--Lagrange system $(\mathcal{J}, \eta)$ and $\mathtt{pr}\colon \mathcal{Y} \to \S$
is the natural projection of $\mathcal{Y}$ onto $\S$, then $\gamma= \mathtt{pr}\circ y\colon J \to \S$
is a critical curve of the total twist functional with respect to compactly supported variations.

We can prove the following.

\begin{lemmaA}\label{l:eqns-motion}
A curve $y\colon J \to \mathcal{Y}$ is an integral curve of the Euler--Lagrange system $(\mathcal{J}, \eta)$ if and only if
the bending $\kappa$ and the twist $\tau$ of the transversal curve $\gamma = \mathtt{pr}\circ y\colon J \to \S$
satisfy the equations
\begin{gather}
 2\kappa \tau' + \tau \kappa' = 0 ,\label{eqns-motion1} \\
 \tau'' + 9 \kappa(1 -\tau\kappa) - \tau^2 = 0\label{eqns-motion2} .
\end{gather}
\end{lemmaA}

\begin{proof}[Proof of Lemma A\ref{l:eqns-motion}]

If $y= (([\mathcal{F}], \kappa, \tau); p_1) \colon J \to \mathcal{Y}$ is an integral curve of the Euler--Lagrange system
$(\mathcal{J}, \eta)$,
the projection $\gamma = \mathtt{pr}\circ y$ is the smooth curve $\gamma(s) = [{F}_1(s)]$,
where ${F}_1$ is the first column of $\mathcal{F}$.
%
The equations
\[
\mu^1= \cdots= \mu^7 =0
\]
 together with the
independence condition $\eta\neq 0$ tell us that $([\mathcal{F}], \kappa, \tau)$
is an integral curve of the Pfaffian system $(\mathcal{A}, \eta)$ on the configuration space $M$.
Hence $\gamma$ is a generic transversal curve with bending $\kappa$, twist $\tau$ and $\mathcal{F}$
is a Wilczynski frame along $\gamma$.
%
%
Next, for the smooth function $\kappa, \tau\colon J \to \mathbb{R}$, let $\kappa'$, $\kappa''$ and $\tau'$, $\tau''$, etc.,
be defined by
\[
 {\rm d}\kappa = \kappa' \eta,\qquad {\rm d}\kappa' = \kappa''\eta, \qquad {\rm d}\tau = \tau' \eta, \qquad {\rm d}\tau' = \tau''\eta .
 \]
With reference to~\eqref{generators-EL-syst}, equation $\sigma_3 = 0$ implies
\[
 p_1 = \frac23 \tau' .
 \]
Further, $\sigma_2=0$ gives
\[
 2\kappa \tau' + \tau \kappa' = 0 .
 \]
%
Finally, equation $\sigma_1 =0$ yields
\[
 \tau'' + 9 \kappa(1 -\tau\kappa) - \tau^2 = 0 .
 \]
%
%

Conversely, let $\gamma \colon J \to \mathcal{S}$ be a generic transversal curve, parametrized by the natural
parameter, satisfying~\eqref{eqns-motion1} and~\eqref{eqns-motion2} and let
$[\mathcal{F}]$ its canonical frame.
Then,
\[
 y(s) = \bigg(([\mathcal{F}], \kappa, \tau) ; \frac23 \tau' \bigg)
 \]
is, by construction, an integral curve of the Euler--Lagrange system $(\mathcal{J}, \eta)$.
\end{proof}

\textbf{Step 4.}
We eventually provide a Lax formulation for the Euler--Lagrange equations (cf.~\eqref{eqns-motion1} and~\eqref{eqns-motion2})
of a critical curve
$\gamma\colon J \to \S$.

Using the Killing form of $\mathfrak{g}$, the dual Lie algebra $\mathfrak{g}^*$ can be identified with $\mathfrak{h}={\rm i}\mathfrak{g}$,
the $G$-module of traceless selfadjoint endomorphisms of $\mathbb{C}^{2,1}$.
Under this identification, the restriction to $\mathcal{Y}$ of the tautological 1-form $\xi$ goes over
to an element of $\mathfrak{h}$ which originates the $\mathfrak{h}$-valued function $L \colon J \to \mathfrak{h}$
given by
\begin{equation}\label{L}
L(s)=
\begin{pmatrix}
0 & {\rm i} \tau'+ 3(1 - \tau\kappa) & 2 {\rm i} \tau\\
\tau & 0 & \tau'+ 3{\rm i}(1 - \tau\kappa)\\
3{\rm i} & -{\rm i} \tau \hphantom{-}& 0
\end{pmatrix}.
\end{equation}
A direct computation shows that the Euler--Lagrange equations~\eqref{eqns-motion1} and~\eqref{eqns-motion2}
of the critical curve
$\gamma$ are satisfied if and only if
\[
 L' (s)= [L (s), K_{\kappa,\tau}(s) ] ,
 \]
where $K_{\kappa,\tau}$ is given by~\eqref{MCE-W-frame}. This concludes the proof of Theorem~\ref{thmA}.
\end{proof}

As a consequence of Theorem~\ref{thmA}, we have the following.

\begin{Corollary}\label{cor-thmA}
Let $\gamma \colon J \to \S$ be a generic transversal curve parametrized by the natural parameter.
Let $[\mathcal{F}] \colon J \to [G]$ be the canonical frame of $\gamma$ and let $L \colon J \to \mathfrak{h}$
be as in~\eqref{L}. If $\gamma$ is a critical curve,
the Lax equation~\eqref{Lax-eqn}
implies that
\begin{equation*}
\mathcal{F}(s)  L(s) \mathcal{F}^{-1}(s)= \mathfrak{M} , \qquad \forall \, s\in J,
 \end{equation*}
where $\mathfrak{M}$ is a fixed element of
$\mathfrak h$
corresponding to a chosen value $L(s_0)$ of $L(s)$.
\end{Corollary}

\begin{Definition}
The element $\mathfrak{M}\in \mathfrak{h}$ is called the \textit{momentum} of
the critical curve~$\gamma$.
\end{Definition}

The characteristic polynomial of the momentum $\mathfrak{M}$ is
\begin{equation*}
 -x^3 -6\kappa\tau^2 x +54\kappa\tau -27\kappa^2 \tau^2 +2\tau^3 -3\tau'^2 -27  .
 \end{equation*}
 The conservation of the momentum along {$\gamma$}
yields the two conservation laws
\begin{gather*}
\kappa \tau^2 = {c_1} \label{CM1} ,\\
 - 18 \kappa \tau + 9 \kappa^2 \tau^2 -\frac23 \tau^3 + \tau'^2 = C_2 -9 \label{CM2} ,
\end{gather*}
for real constants $c_1$ and $C_2$. We let $c_2 := C_2-9$.
Using this notation,
the (opposite of the) characteristic polynomial of the momentum is
\begin{equation*}
 {\rm Q}(x)=x^3+6c_1x+(27+3c_2).
 \end{equation*}

 If $c_1\neq 0$, the twist and the bending are never zero and the conservation laws can be rewritten as
\begin{equation}\label{cons-momentum-bis}
 \kappa=c_1\tau^{-2},\qquad \frac{3}{2} \tau^2\tau'^2=\tau^5+\frac{3}{2}c_2\tau^2+27c_1\tau-\frac{27}{2}c_1^2.
 \end{equation}
If $c_1=0$, it can be easily proved that $\kappa=0$ and the second conservation law takes the form
\begin{equation*}
 \tau'^2 =\frac{2}{3}\tau^3+c_2.
 \end{equation*}
%
\begin{Definition}
The pair of real constants ${\bf c}=(c_1, c_2)$ is called the
\textit{modulus} of the critical curve~$\gamma$.

\end{Definition}

\begin{Remark}
For the application of Griffiths' approach to
other geometric variational problems, the reader is referred to
\cite{DMN,EMN-JMAA,Gr,GM,MN-CQG,MN-SIAM,MN-CAG}.
\end{Remark}

\section{The CR twist of a critical curve}\label{s:twist}

\subsection{Phase types}


For ${\bf c}=(c_1,c_2)\in \R^2$, we denote by ${\rm P}_{{\bf c}}$
the quintic polynomial in principal form given by
%
%
\begin{equation*}
{\rm P}_{{\bf c}}(x)=x^5+\frac{3}{2}c_2x^2+27c_1x-\frac{27}{2}c_1^2
 \end{equation*}
 and by ${\rm Q}_{{\bf c}}$ the cubic polynomial given by
\begin{equation}\label{Cubic_Poly}
{\rm Q}_{{\bf c}}(x)=x^3+6c_1x+(27+3c_2).
 \end{equation}

Excluding the case ${\bf c}=0$, ${\rm P}_{{\bf c}}$ possesses at least a pair of complex conjugate roots.

\begin{Definition}
We adopt the following terminology.
\begin{itemize}\itemsep=0pt

\item ${\bf c}\in \R^2$ is of \emph{phase type ${\mathcal A}$} if ${\rm P}_{{\bf c}}$ has four complex roots $a_j\pm {\rm i}b_j$, $j=1,2$, $0<b_1<b_2$,
and a simple real root $e_1$;

\item ${\bf c}\in \R^2$ is of \emph{phase type ${\mathcal B}$} if ${\rm P}_{{\bf c}}$ has two complex roots $a \pm {\rm i}b$,
$b>0$, and three simple real roots
$e_1<e_2<e_3$;

\item ${\bf c}\in \R^2$ is of \emph{phase type} ${\mathcal C}$ if ${\rm P}_{{\bf c}}$ has a multiple {real} root.
\end{itemize}
In the latter case, two possibilities may occur: (1) ${\rm P}_{{\bf c}}$ has a double real root and a simple real root; or
(2) ${\rm P}_{{\bf c}}$ has a {real} root of multiplicity $5$.

By the same letters, we also denote the corresponding sets of moduli of
phase types ${\mathcal A}$, ${\mathcal B}$, and ${\mathcal C}$, respectively.
\end{Definition}

Next, we give a more detailed description of the sets ${\mathcal A}$, ${\mathcal B}$, and ${\mathcal C}$.
%
%
To this end, we start by defining the \emph{separatrix curve}. Let
$(m,n)$ be the homogeneous coordinates of $\mathbb{RP}^1$ and let $[(m_*,n_*)]$ be the point of $\mathbb{RP}^1$ such that $3m_*^3+6m_*^2n_*+4m_*n_*^2+2n_*^3=0$ (i.e., $m_*=1$ and $n_*\approx-0.72212$).

\begin{Definition}
The \emph{separatrix curve} $\Xi\subset \R^2$ is the image of the parametrized curve
$\xi =(\xi_1,\xi_2) \colon \mathbb{RP}^1\setminus\{[(m_*,n_*)]\}\to \R^2$, defined by
\begin{gather*}
{\xi_1([(m,n)])}=\frac{6\sqrt[3]{2}mn^{4/3}\big(3m^2+2mn+n^2\big)^{4/3}}{\big(3m^3+6m^2n+4mn^2+2n^3\big)^{5/3}},\\
{\xi_2([(m,n)])}=-\frac{36n\big(3m^2+2mn+n^2\big)\big(4m^3+3m^2n+2mn^2+n^3\big)}{\big(3m^3+6m^2n+4mn^2 + 2n^3\big)^{2}}.
\end{gather*}
\end{Definition}

\begin{Remark}\label{remarkseparatrix}
The map $\xi$ is injective and $\Xi$ has a cusp at
$\xi([(1,1)])=\big(\frac{4}{5}\big(\frac{6}{5}\big)^{2/3},-\frac{48}{5}\big)$.
It is regular elsewhere. In addition,
$\Xi$ has a horizontal inflection point at ${\xi([(0,1)])}=(0,-9)$. Let~${\rm J}_{\xi}$ be the interval
$(\arctan(n_*),\arctan(n_*)+\pi)\approx (-0.625418,2.51617)$. Then,
$\widetilde{\xi}\colon t\in {\rm J}_{\xi}\to \xi(\cos(t),\sin(t))\in \Xi$ is another parametrization of $\Xi$.
The inflection point is $\widetilde{\xi}(\pi/2)$. The ``negative part" $\Xi_-=\Xi\cap \big\{{\bf c}\in \R^2 \mid c_1<0\big\}$ of $\Xi$
is parametrized by the restriction of $\widetilde{\xi}$ to $\hat{{\rm J}}_{\xi}=(\pi/2,\pi +\arctan(n_*))$.
The left picture of Figure~\ref{FIG1} reproduces the separatrix curve (in black); the negative part of the separatrix
curve is highlighted in dashed-yellow.
The cusp is the red point and the horizontal inflection point is coloured in green.
\end{Remark}

\begin{Definition}
The (open) \emph{upper and lower domains} bounded by the separatrix curve $\Xi$ are denoted by ${\mathcal M}_{\pm}$.
In Figure~\ref{FIG1},
the upper domain $\mathcal{M}_+$ is coloured in three orange tones: orange, dark-orange and light-orange;
the lower domain $\mathcal{M}_-$ is coloured in two brown tones: light-brown and brown.
\end{Definition}

\begin{figure}[t]\centering
\includegraphics[height=6cm,width=6cm]{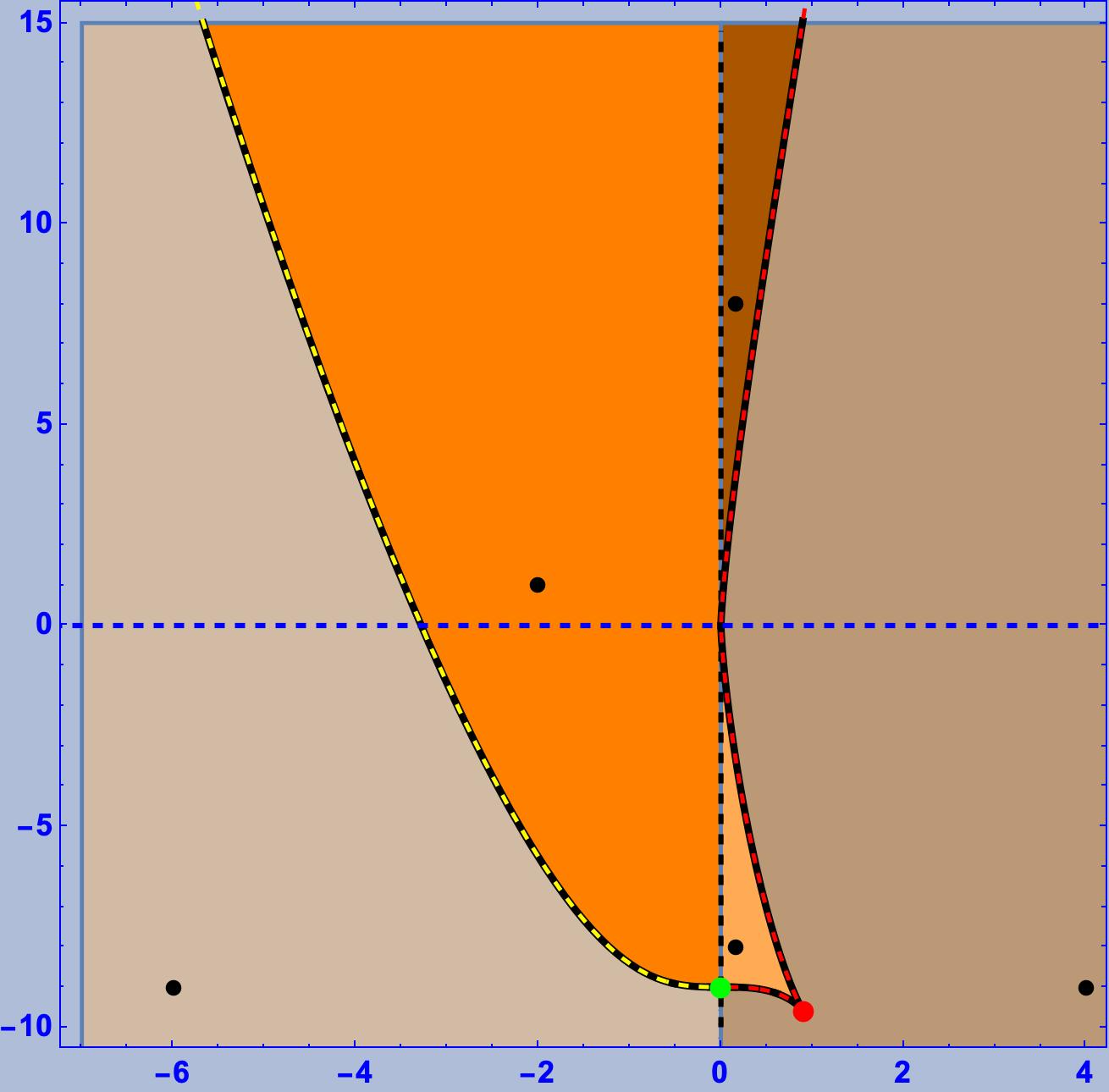}\qquad
\includegraphics[height=6cm,width=6cm]{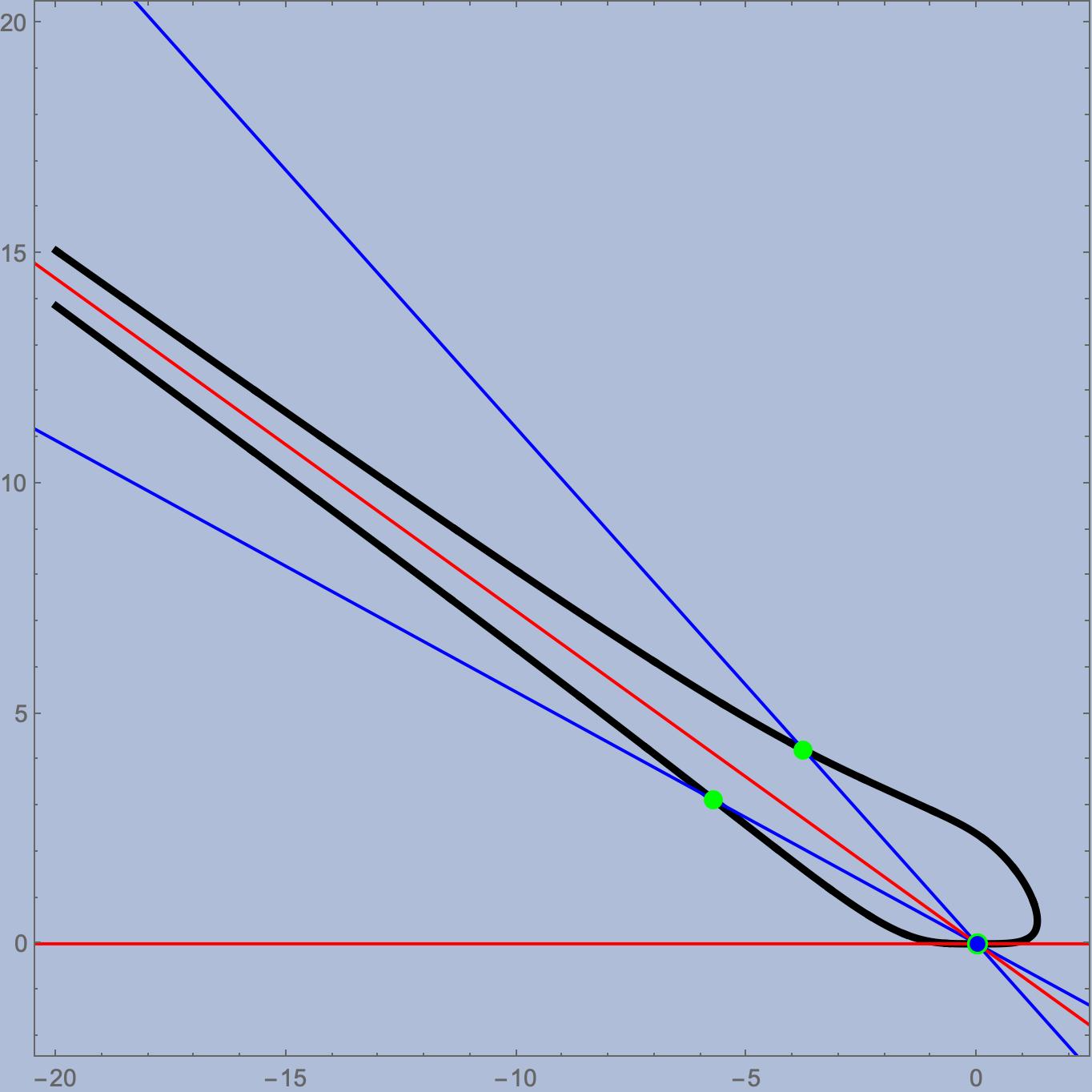}
\caption{On the left: the separatrix curve (in black), the upper domain $\mathcal{M}_+$ (coloured in three orange tones)
and the lower domain $\mathcal{M}_-$ (coloured in two brown tones). On the right: the curve ${\mathtt C}$ and its parametrization
obtained by intersecting ${\mathtt C}$ with lines through the origin.}\label{FIG1}
\end{figure}

\begin{Proposition}
The polynomial ${\rm P}_{{\bf c}}$ has multiple roots if and only if ${\bf c}\in \Xi\cup Oy$, has four complex roots
 if and only if ${\bf c}\in {\mathcal M}_-\setminus (Oy\cap {\mathcal M}_- )$, and has three distinct real roots if and only if
 ${\bf c}\in {\mathcal M}_+\setminus (Oy\cap {\mathcal M}_+ )$. Equivalently,
\[
{\mathcal A}={\mathcal M}_-\setminus (Oy\cap {\mathcal M}_-),\qquad {\mathcal B}
 ={\mathcal M}_+\setminus (Oy\cap {\mathcal M}_+),\qquad
 {\mathcal C}=\Xi\cup Oy.
\]
\end{Proposition}

\begin{proof}
First, we prove the following claim.

\begin{Claim*} ${\rm P}_{{\bf c}}$ has a double root $a_3\neq 0$ if and only {if} ${\bf c}$ belongs
to the separatrix curve minus the cusp.
\end{Claim*}

Note that $c_1\neq 0$ (otherwise the double root would be $0$). Let $a_4$ be the other simple real root and $b_1+{\rm i}b_2$,
$b_1-{\rm i}b_2$, $b_2>0$, be the two complex conjugate roots.
Since the sum of the roots of ${\rm P}_{{\bf c}}$ is zero, we have $b_1=-\frac{1}{2}(2a_3+a_4)$.
Since the coefficient of $x^3$ is zero and $b_2>0$, we get $b_2=\sqrt{2a_3^2+a_3a_4+3a_4^2/4}$.
Expanding $(x-a_3)^2(x-a_4)(x-b_1-{\rm i}b_2)(x-b_1+{\rm i}b_2)$
and comparing the {coefficients} of the
monomials $x^n$, $n=1,\dots, 4$, with the coefficients of ${\rm P}_{{\bf c}}$ we may write $c_1$ and
$c_2$ as functions of $a_3$ and $a_4$,
\begin{gather}
 c_1=\frac{1}{27}\big(3a_3^4+6a_3^3a_4+4a_3^2a_4^2+2a_3a_4^3\big),\nonumber\\
c_2=-\frac{2}{3}\big(4a_3^3+3a_3^2a_4+2a_3a_4^2+a_4^3\big).\label{eq1multp}
\end{gather}
In addition,
\[
 c_1^2=\frac{2}{27}\big(3a_3^4a_4+2a_3^3a_4^2+a_3^2a_4^3\big).
\]
Taking into account that $a_3\neq 0$, it follows that $(a_3,a_4)$ belongs to the algebraic curve
${\mathtt C}$ (the black curve on the right picture in Figure~\ref{FIG1}) defined by the equation
\[
 54y\big(3x^2+2xy+y^2\big)-\big(3x^3+6x^2y+4xy^2+2y^3\big)^2=0.
\]
Now, consider the line $\ell_{m,n}$ through the origin, with homogeneous coordinates $(m,n)$, i.e., the line with parametric
equations $p_{m,n}(t)=(mt,nt)$.
If $(m,n)\neq (1,0)$ and $3m^3+6m^2n+4mn^2+2n^3\neq 0$ (we are excluding the two red lines
on the right picture in Figure~\ref{FIG1}), $\ell_{m,n}$
intersects ${\mathtt C}$ when $t=0$ and $t=t_{m,n}$,
where
\[
 t_{m,n}=\frac{3\sqrt[3]{2}\sqrt[3]{n\big(3m^2+2mn+n^2\big)}}{\sqrt[3]{\big(3m^3+6m^2n+4mn^2+2n^3\big)^2}}.
\]
If {$(m,n)= (1,0)$}
or $3m^3+6m^2n+4mn^2+2n^3= 0$, $\ell_{m,n}$ intersects ${\mathtt C}$ only at the origin
(see the right picture in Figure~\ref{FIG1}). Hence
$\beta \colon [(m,n)]\to t_{m,n}\cdot (m,n)$, $[(m,n)]\neq [(1,0)]$,
$3m^3+6m^2n+4mn^2+2n^3\neq 0$, is a parametrization of ${\mathtt C}\setminus \{(0,0)\}$.
Thus, using~\eqref{eq1multp}, the map
\[
 [(m,n)]\to (c_1(\beta([(m,n])),c_2(\beta([(m,n]))\in \R^2
\]
is a parametrization of the set of all ${\bf c}$, $c_1\neq 0$, such that $\mathrm{P}_{{\bf c}}$ has multiple roots.
It is now a~computational matter to check that
$(c_1(\beta([(m,n])),c_2(\beta([(m,n]))=\xi([(m,n)])$. This proves the claim. It also shows that ${\rm P}_{{\bf c}}$
has multiple roots if and only if ${\bf c}\in \Xi\cup Oy$.

To prove the other assertions, we begin by observing that the discriminant of the derived polynomial ${\rm P}'_{{\bf c}}$
is negative. Hence
${\rm P}'_{{\bf c}}$ has two distinct real roots and a pair of complex conjugate roots.
Denote by $x'_{{\bf c}}$ and $x''_{{\bf c}}$ the real roots of ${\rm P}'_{{\bf c}}$,
ordered so that
$x'_{{\bf c}}<x''_{{\bf c}}$. Observe that $x'_{{\bf c}}$ and~$x''_{{\bf c}}$ are differentiable functions of $\textbf{c}$.
Then, ${\rm P}_{{\bf c}}$ possesses three distinct real roots if and only if $x'_{{\bf c}}\cdot x''_{{\bf c}}<0$,
one simple real root if and only if $x'_{{\bf c}}\cdot x''_{{\bf c}}>0$,
and a multiple root if and only if~$x'_{{\bf c}}\cdot x''_{{\bf c}}=0$.

From the first part of the proof, the set of all ${\bf c}\in \R^2$, such that ${\rm P}_{{\bf c}}$ has only simple roots
is the complement of $\Xi\cup Oy$. This set has five connected components:
%
%
%
\begin{gather}
{\mathcal M}_+' =  \{{\bf c}\in {\mathcal M}_+\setminus (Oy\cap {\mathcal M}_+ ) \mid c_1<0  \}, \nonumber \\
{\mathcal M}_+'' = \{{\bf c}\in {\mathcal M}_+\setminus (Oy\cap {\mathcal M}_+) \mid c_1>0\  \text{and}\  c_2>0  \}, \nonumber \\
{\mathcal M}_+''' = \{{\bf c}\in{\mathcal M}_+\setminus (Oy\cap {\mathcal M}_+) \mid c_1>0\ \text{and}\  c_2<0  \}, \nonumber \\
{\mathcal M}_-' = \{{\bf c}\in {\mathcal M}_-\setminus (Oy\cap {\mathcal M}_-) \mid c_1<0  \}, \nonumber \\
{\mathcal M}_-'' = \{{\bf c}\in {\mathcal M}_-\setminus (Oy\cap {\mathcal M}_-) \mid c_1>0  \}.\label{connectedcomponents}
 \end{gather}
Referring to
the left picture in Figure~\ref{FIG1}, ${\mathcal M}_+'$ is the orange domain, ${\mathcal M}_+''$ is the dark-orange domain,
${\mathcal M}_+'''$ is the light-orange domain, ${\mathcal M}_-'$ is the light-brown domain, and ${\mathcal M}_-''$ is the brown domain.

Consider the following points (the black points in Figure~\ref{FIG1}):
\begin{alignat*}{4}
& {\bf c}_1=(-2,1)\in {\mathcal M}'_+, \qquad && {\bf c}_2=(1/6,8)\in {\mathcal M}_+'', \qquad && {\bf c}_3=(1/6,-8)\in {\mathcal M}_+''', &\\
&{\bf c}_4=(-6,-9)\in {\mathcal M}_-', \qquad && {\bf c}_5=(4,-9)\in {\mathcal M}_-''. &
\end{alignat*}

Using Klein's formulas for the icosahedral solution of a quintic polynomial in principal form (cf.~\cite{K,Na,Tr}),\footnote{We used the
Trott and Adamchik code (cf.~\cite{Tr}) implementing Klein's formulas in the software \textsc{Mathematica}.} we find that the polynomials
${\rm P}_{{\bf c}_j}$, $j=1,2,3$, have three distinct real roots and that ${\rm P}_{{\bf c}_j}$, $j=4,5$, have one real root.
The domain ${\mathcal M}'_+$ is connected and the function $ {\mathcal M}'_+ \ni {\bf c} \mapsto x'_{{\bf c}}\cdot x''_{{\bf c}}$ is differentiable and nowhere zero. Since $x'_{{\bf c}_1}\cdot x''_{{\bf c}_1}<0$, it follows that
${\bf c}\mapsto x'_{{\bf c}}\cdot x''_{{\bf c}}$ is strictly negative. Then, ${\rm P}_{{\bf c}}$ has three distinct real roots,
for every ${\bf c}\in {\mathcal M}'_+$. Similarly, ${\rm P}_{{\bf c}}$ has three distinct real roots, for every
${\bf c}\in {\mathcal M}''_+\cup {\mathcal M}'''_+$ and a unique real root for every
${\bf c}\in {\mathcal M}'_-\cup {\mathcal M}''_-$. This concludes the proof.
\end{proof}

\begin{figure}[t]\centering
\includegraphics[height=6cm,width=6cm]{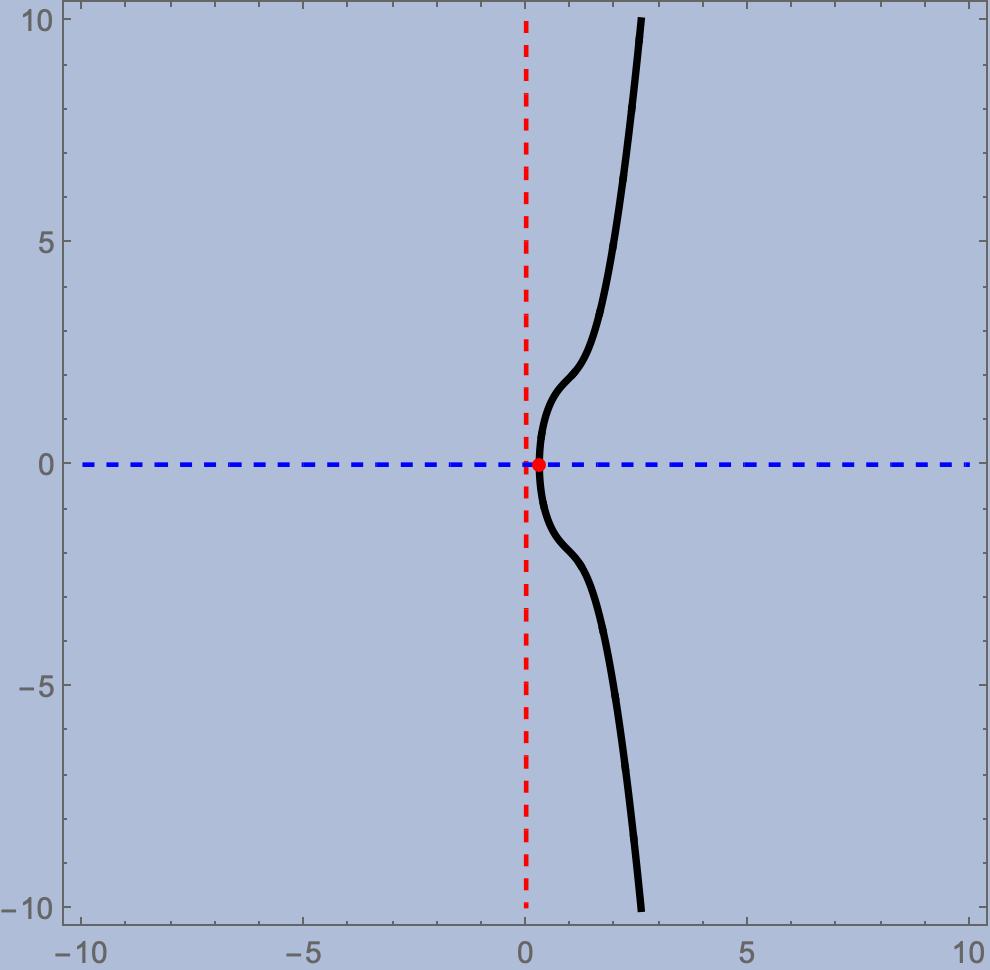}\qquad
\includegraphics[height=6cm,width=6cm]{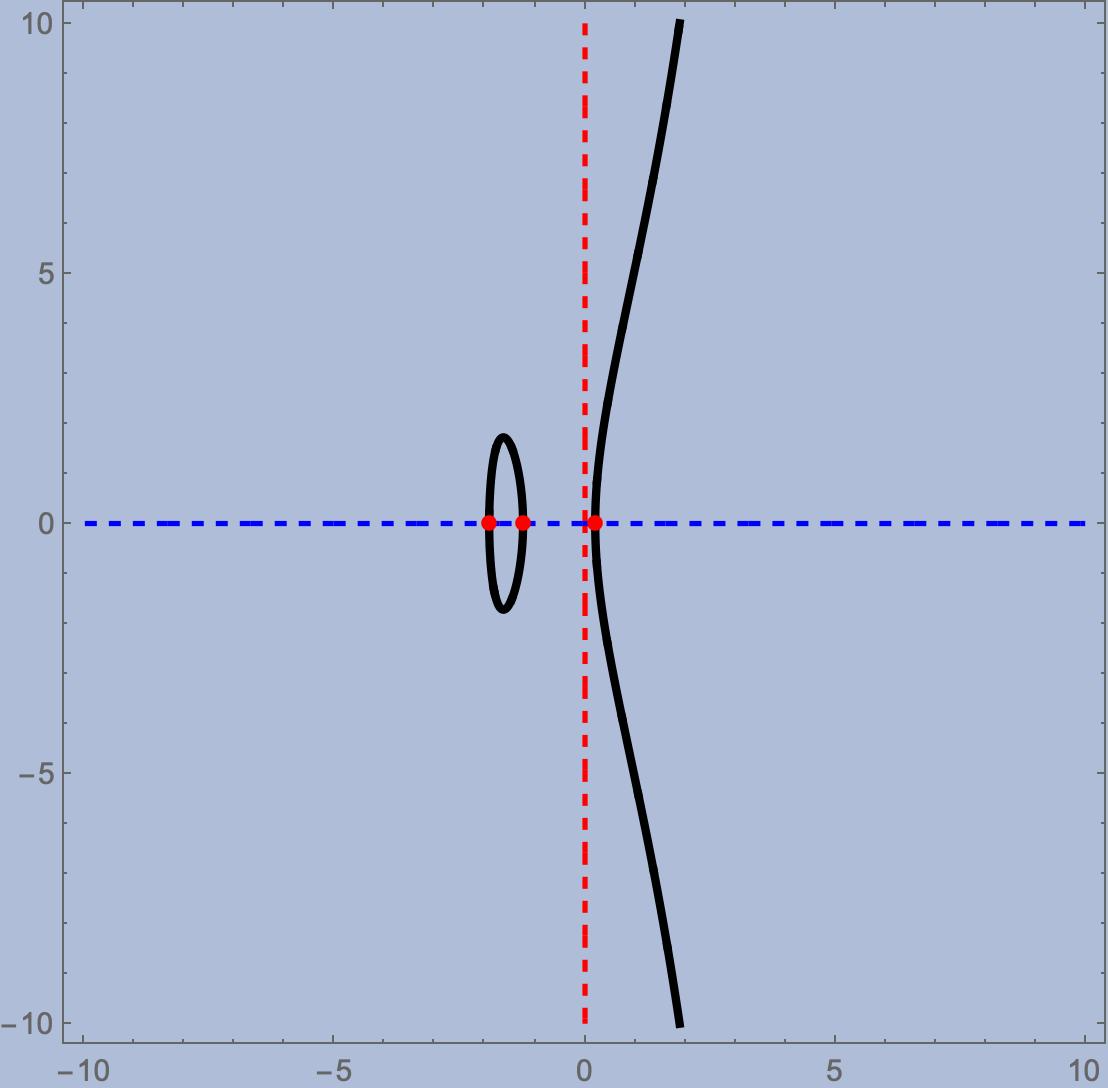}
\caption{On the left: the phase curve of ${\bf c}\in {\mathcal A}$. On the right: the phase curve of ${\bf c}\in {\mathcal B}$.}\label{FIG2}
\end{figure}

\begin{Remark}\label{remarksignrootsP}
The real roots of ${\rm P}_{{\bf c}_1}$ are $e_1=-2.44175<e_2=-0.9904<0<e_3=2.87645$ and those of
${\rm P}_{{\bf c}_2}$ are $e_1=-2.14118<e_2=-0.448099<0<e_3=0.0701938$. Instead, the roots of ${\rm P}_{{\bf c}_3}$ are
$0<e_1=0.12498<e_2=0.250656<2.15383$.
Since the product $e_2({\bf c}) e_3({\bf c})$ is a continuous function on the connected components
${\mathcal M}_+'$, ${\mathcal M}_+''$, and ${\mathcal M}_+'''$, we deduce that the lowest roots of ${\rm P}_{{\bf c}}$
are negative if
${\bf c}\in {\mathcal M}_+'\cup {\mathcal M}_+''$ and positive if ${\bf c}\in {\mathcal M}_+'''$.
\end{Remark}

\subsection{Phase curves and signatures}

\begin{Definition}
Let $\Sigma_{{\bf c}}$ be the {real} algebraic curve defined by $y^2={\rm P}_{\bf c}(x)$.
We call $\Sigma_{{\bf c}}$ the \emph{phase curve} of ${\bf c}$.
\end{Definition}

If ${\bf c}\in \mathcal A\cup \mathcal B$, $\Sigma_{{\bf c}}$ is a smooth real cycle of a hyperelliptic curve of genus $2$.
If ${\bf c}\in \mathcal C$, and ${\bf c}\neq 0$, $\Sigma_{{\bf c}}$ is a singular real cycle of an elliptic curve.
If ${\bf c}=0$, $\Sigma_{{\bf c}}$ is a singular rational curve.
{The following facts can be easily verified:}
\begin{itemize}\itemsep=0pt

\item if ${\bf c}\in {\mathcal A}$, $\Sigma_{{\bf c}}$ is connected, unbounded, and intersects the $Ox$-axis at $(e_1,0)$
(see Figure~\ref{FIG2});

\item if ${\bf c}\in {\mathcal B}$, $\Sigma_{{\bf c}}$ has two smooth connected components, one is compact and the other
is unbounded. Let $\Sigma'_{{\bf c}}$ be the compact connected component and $\Sigma''_{{\bf c}}$ be the noncompact one.
$\Sigma'_{{\bf c}}$ intersects the $Ox$-axis at $(e_1,0)$ and $(e_2,0)$, while $\Sigma''_{{\bf c}}$ intersects
the $Ox$-axis at $(e_3,0)$ (see Figure~\ref{FIG2});

\item if ${\bf c}\in {\mathcal C}$ and $c_1 \neq 0$ , $\Sigma_{{\bf c}}$ has a smooth, unbounded connected
component $\Sigma''_{{\bf c}}$ and an isolated singular point $(e_1,0)$, where $e_1=e_2$ is the double real root
of ${\rm P}_{\bf c}(x)$.
The unbounded connected component intersects the $Ox$-axis at $(e_3,0)$, where $e_3$ is the simple real root
of~${\rm P}_{\bf c}(x)$ (see Figure~\ref{FIG3}). If $c_1=0$ and $c_2\neq0$, $\Sigma_{{\bf c}}$ is connected, with
an ordinary double point (see Figure~\ref{FIG3}). If ${\bf c}=0$, $\Sigma_{{\bf c}}$ is connected with a cusp
at the origin (see Figure~\ref{FIG3}).
\end{itemize}

\begin{Definition}
Let $\gamma$ be a critical curve with nonconstant twist and modulus ${\bf c}$. Let ${\rm J}_{\gamma}\subset \R$ be
the maximal interval of definition of $\gamma$.
%
%
With reference to~\eqref{cons-momentum-bis}, we adapt to our context the terminology used in~\cite{COST,KRV,MNJMIV}
and call
\[
\sigma_{\gamma} \colon \ {\rm J}_{\gamma} \to \R^2, \qquad s\mapsto \big(\tau(s),\sqrt{3/2} \tau(s){\tau'}(s)\big)
\]
the \emph{signature} of $\gamma$.
\end{Definition}

\begin{Remark}
From the Poincar\'e--Bendixson theorem, it follows that the twist of $\gamma$ is periodic if and only if
$\sigma_{\gamma}({\rm J}_{\gamma})$ is compact. Observing that $\sigma_{\gamma}({\rm J}_{\gamma})$ is
one of the 1-dimensional connected components of $\Sigma_{{\bf c}}$, we can conclude that the twist is a periodic function if and only if
${\bf c}\in {\mathcal B}$ and $\sigma_{\gamma}({\rm J}_{\gamma})=\Sigma'_{{\bf c}}$.
\end{Remark}

\begin{figure}[t]\centering
\includegraphics[height=4cm,width=4cm]{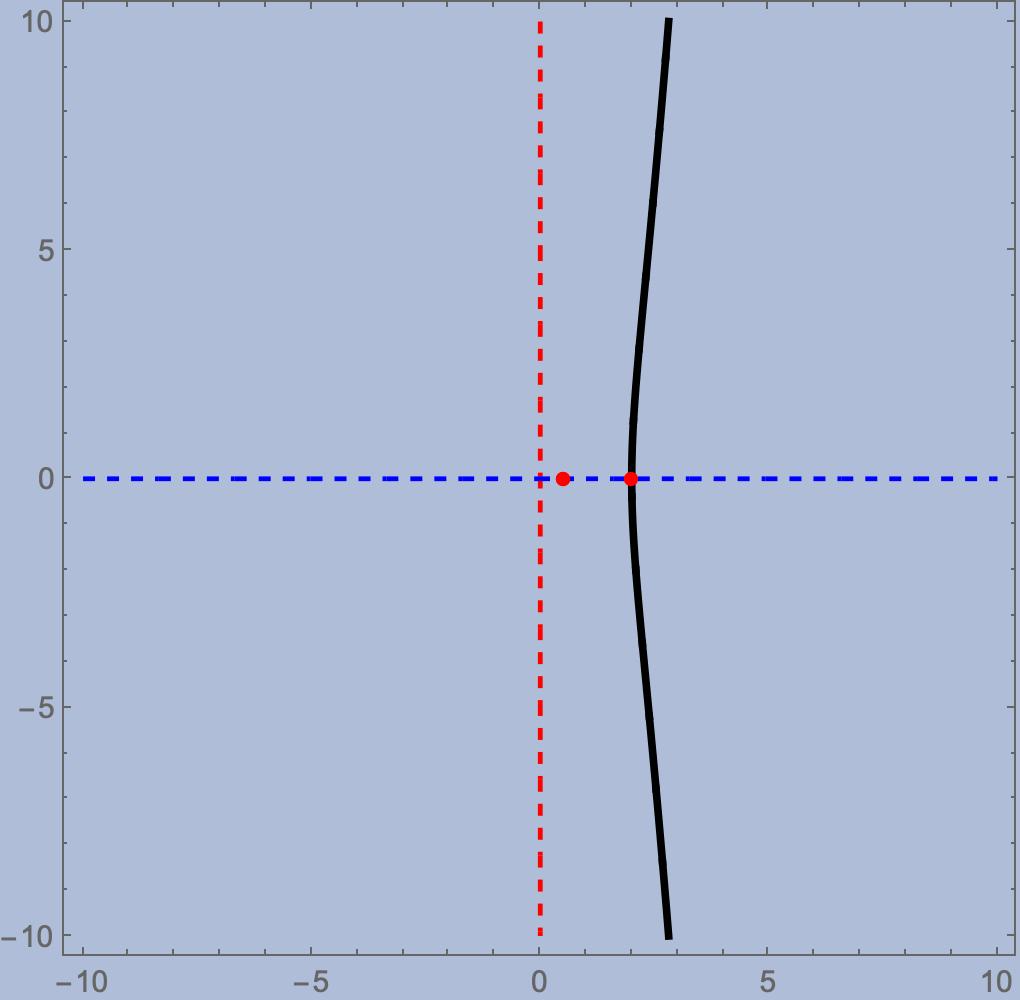}\qquad
\includegraphics[height=4cm,width=4cm]{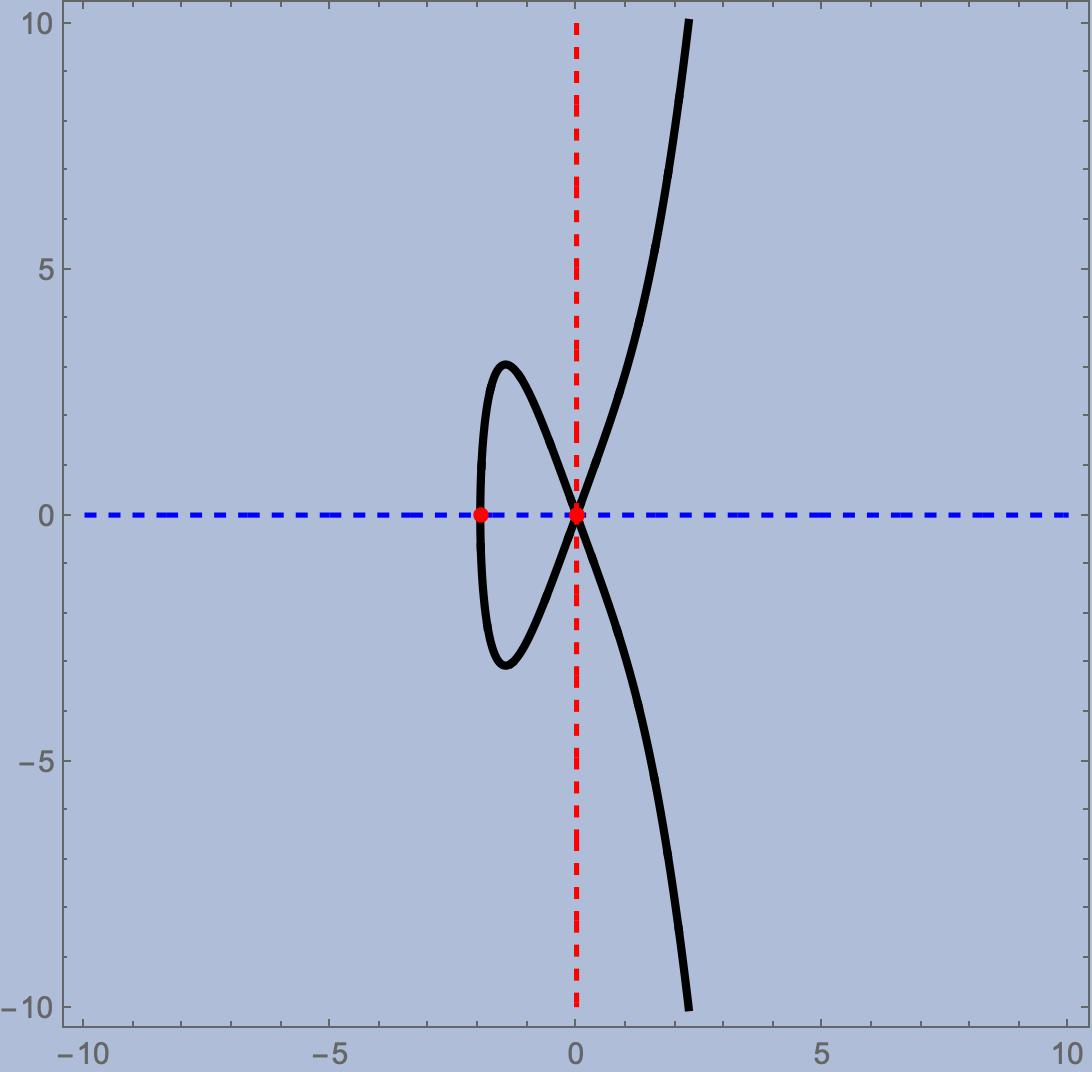}\qquad
\includegraphics[height=4cm,width=4cm]{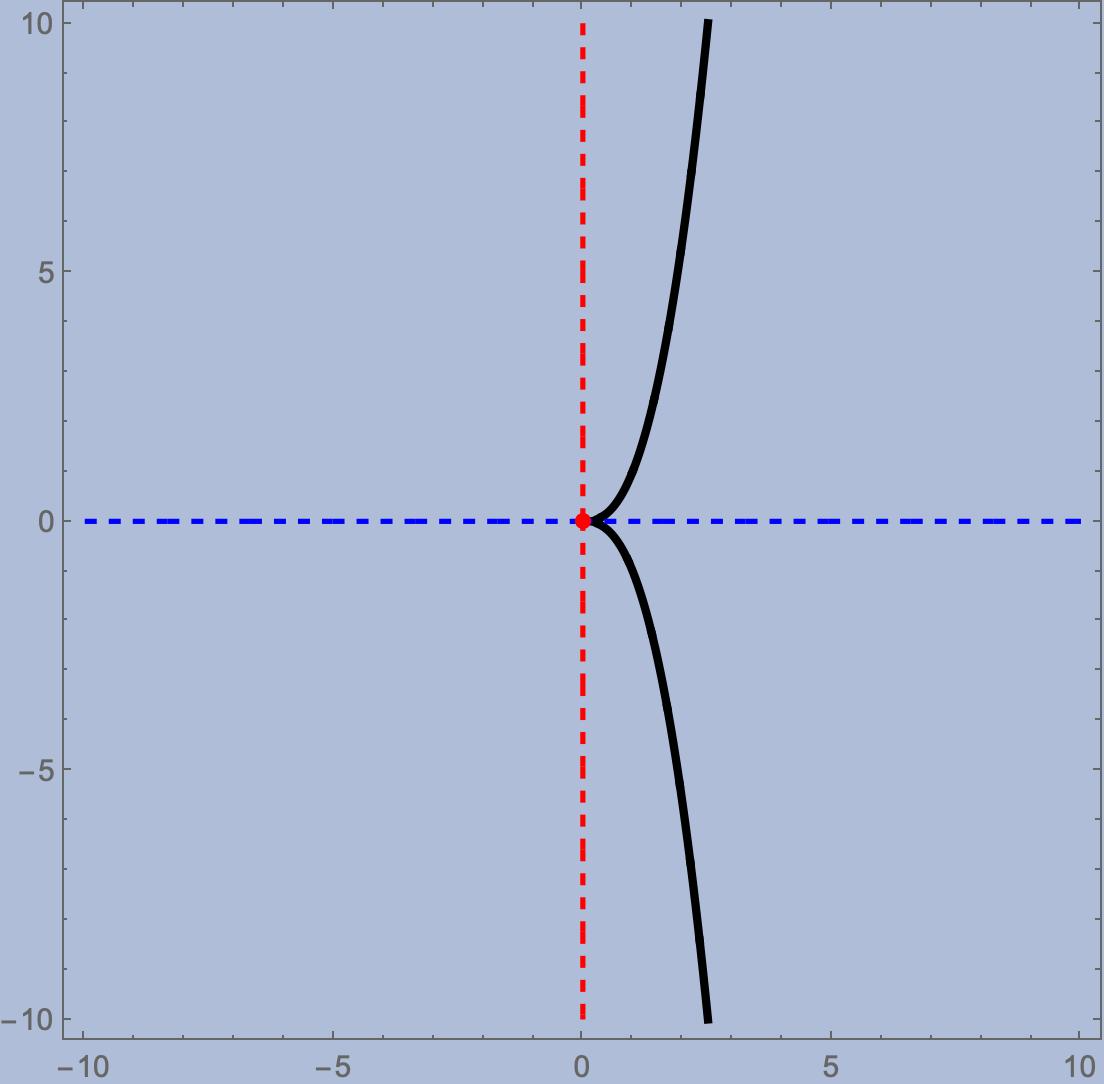}
\caption{On the left: the phase curve of ${\bf c}\in {\mathcal C}$, $c_1\neq 0$. On the center:
the phase curve of ${\bf c}\in {\mathcal C}$, $c_1=0$ and $c_2\neq 0$. On the right: the phase curve
of ${\bf c}=(0,0)$.}\label{FIG3}
\end{figure}

\begin{Definition}
A critical curve $\gamma$ with modulus ${\bf c}$ is said to be of type ${\mathcal B}'$ if ${\bf c}\in {\mathcal B}$
and $\sigma_{\gamma}({\rm J}_{\gamma})=\Sigma'_{{\bf c}}$; it is said to be of type ${\mathcal B}''$ if ${\bf c}\in {\mathcal B}$
and $\sigma_{\gamma}({\rm J}_{\gamma})=\Sigma''_{{\bf c}}$.
\end{Definition}

\subsection{The twist of a critical curve}

\subsubsection[The twist of a critical curve of type A]{The twist of a critical curve of type $\boldsymbol{\mathcal A}$}

Let $\gamma$ be a critical curve of type ${\mathcal A}$, i.e., with modulus ${\bf c} \in {\mathcal A}$.
Then ${\rm P}_{{\bf c}}$ has a unique real root $e_1$. The polynomial ${\rm P}_{{\bf c}}(x)$ is positive if $x>e_1$
and is negative if $x<e_1$.
Since ${\rm P}_{{\bf c}}(0)=-27c_1^2/2<0$, the root is positive.
Let $\omega_{{\bf c}}>0$ be the improper hyperelliptic integral of the first kind defined by
\[
 \omega_{{\bf c}} = \sqrt{\frac{3}{2}}\int_{e_1}^{+\infty}\frac{\tau {\rm d}\tau}{\sqrt{{\rm P}_{{\bf c}}(\tau)}}>0.
\]
The incomplete hyperelliptic integral
\[
 h_{{\bf c}}(\tau) = \sqrt{\frac{3}{2}}\int_{e_1}^{\tau}\frac{u {\rm d}u}{\sqrt{{\rm P}_{{\bf c}}(u)}},\qquad u\ge e_1
\]
is a strictly increasing diffeomorphism of $[e_1,+\infty)$ onto $[0,\omega_{{\bf c}})$ (see Figure~\ref{FIG4}).
The twist is the unique even function
$\tau_{{\bf c}}\colon (-\omega_{{\bf c}},\omega_{{\bf c}})\to \R$, such that $\tau_{{\bf c}}=h_{{\bf c}}^{-1}$ on $[0,\omega_{{\bf c}})$.
The maximal domain of definition is ${\rm J}_{{\bf c}}=(-\omega_{{\bf c}},\omega_{{\bf c}})$.
$\tau_{{\bf c}}$ is strictly positive, with vertical asymptotes as $s\to \mp \omega_{{\bf c}}^{\pm}$ (see Figure~\ref{FIG4}).
Note that $\tau_{{\bf c}}$ is the solution of the Cauchy problem
\[
 {\tau}''=\tau^2-9c_1\tau^{-2}\big(1-c_1\tau^{-1}\big),\qquad \tau(0)=e_1,\qquad {\tau'}(0)=0.
\]

\begin{figure}[t]\centering
\includegraphics[height=6cm,width=6cm]{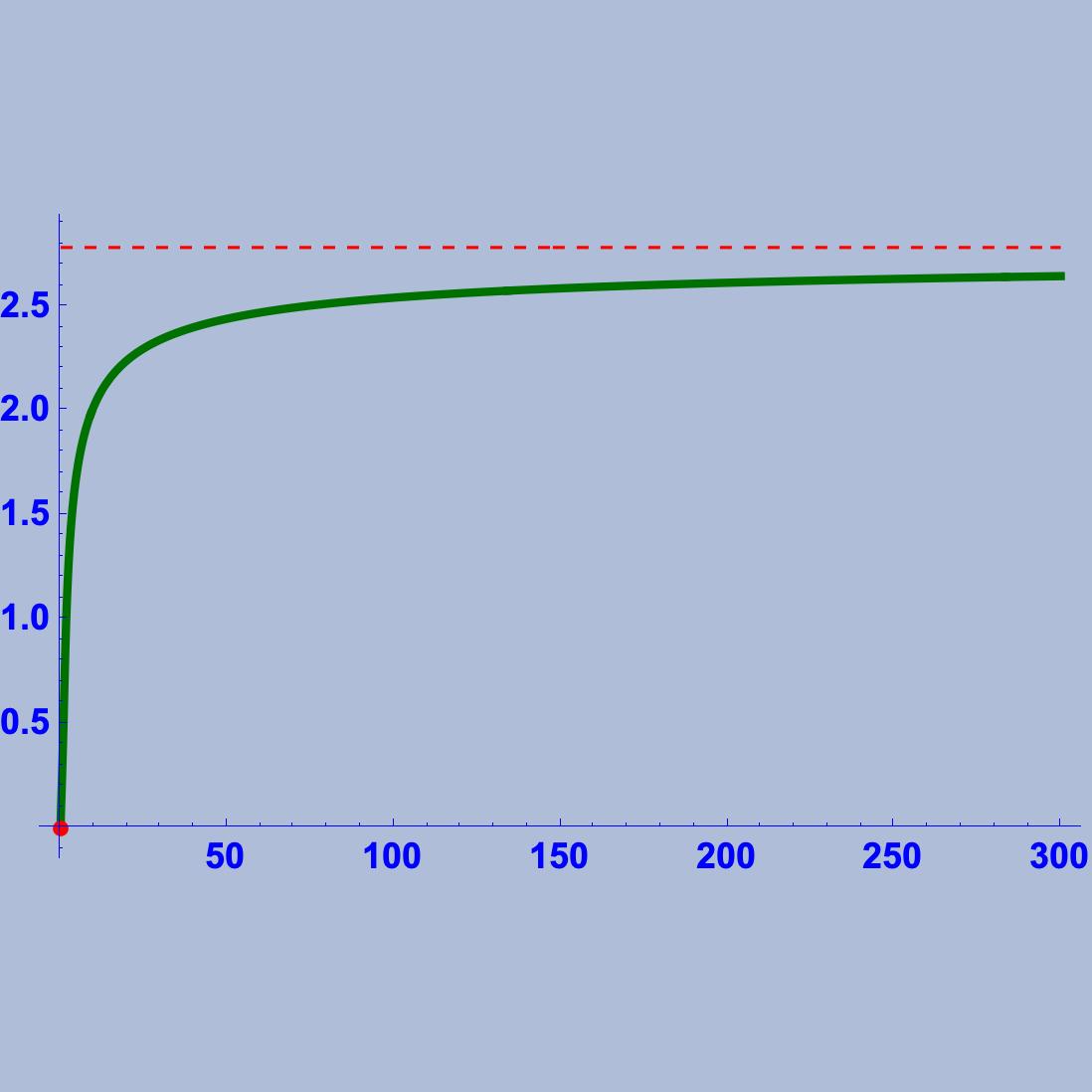}\qquad
\includegraphics[height=6cm,width=6cm]{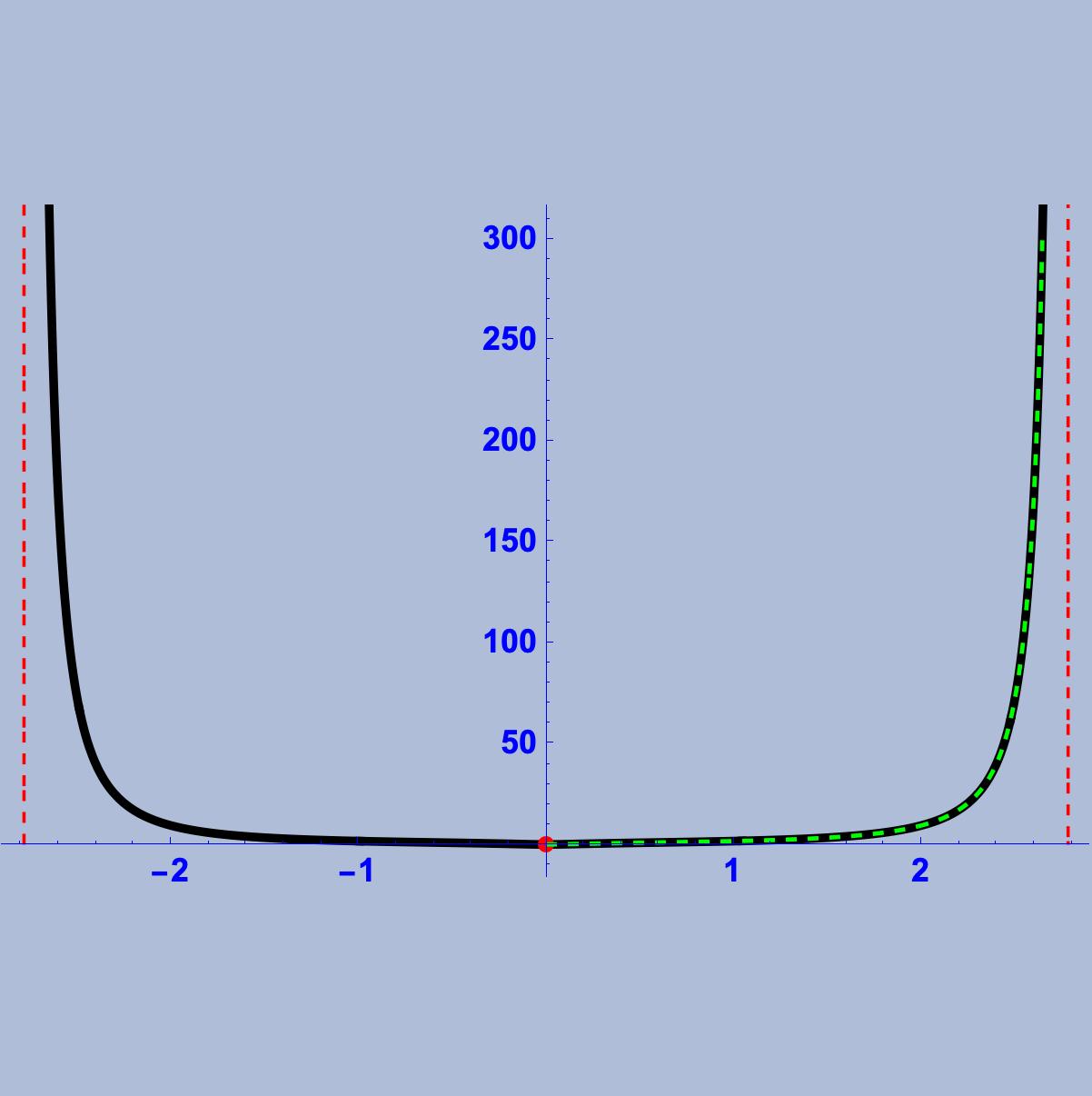}
\caption{On the left: the graph of the function $h_{{\bf c}}$, ${\bf c}=(1/2,-4.8)\in {\mathcal A}$. The red line is the horizontal asymptote $y=e_1$. On the right: the graph of the twist. The red lines are the vertical asymptotes $x=\pm \omega_{{\bf c}}$.}\label{FIG4}
\end{figure}

\subsubsection[The twist of a critical curve of type B']{The twist of a critical curve of type $\boldsymbol{{\mathcal B}'}$}

Let $e_1<e_2<e_3$ be the simple real roots of ${\rm P}_{{\bf c}}$. The highest root $e_3$ is positive.
The lower roots~$e_1$ and~$e_2$ are either both negative or both positive and ${\rm P}_{{\bf c}}$ is positive on $(e_1,e_2)$.
Let $\omega_{{\bf c}}>0$ be the complete hyperelliptic integral of the first kind
\begin{equation}\label{halfpB1}
 \omega_{{\bf c}} ={\rm sign}(e_1) \sqrt{\frac{3}{2}}\int_{e_2}^{e_1}\frac{\tau {\rm d}\tau}{\sqrt{{\rm P}_{{\bf c}}(\tau)}}>0.
 \end{equation}
Let $h_{{\bf c}}$ be the incomplete hyperelliptic integrals of the first kind
\[
 h_{{\bf c}}(\tau) =
\begin{cases} \displaystyle \sqrt{\frac{3}{2}}\int_{e_2}^{\tau}\frac{u{\rm d}u}{\sqrt{{\rm P}_{{\bf c}}(u)}}, & \tau\in [e_1,e_2], \ e_1<e_2<0,\\
 \displaystyle \sqrt{\frac{3}{2}}\int_{e_1}^{\tau}\frac{u {\rm d}u}{\sqrt{{\rm P}_{{\bf c}}(u)}},& \tau\in [e_1,e_2],\  0<e_1<e_2.
\end{cases}
 \]
The function $h_{{\bf c}}$ is a diffeomorphism of $[e_1,e_2]$ onto $[0,\omega_{{\bf c}}]$, strictly decreasing if $e_1<e_2<0$ and strictly
increasing if $0<e_1<e_2$ (see Figure~\ref{FIG5}). The twist $\tau_{{\bf c}}$ is the even periodic function with least period $2\omega_{{\bf c}}$, {obtained by extending periodically the function
$\tau(s) = h_{{\bf c}}^{-1}(s)$ defined on $[0,\omega_{{\bf c}}]$ and on $[-\omega_{{\bf c}},0]$, respectively.}
\begin{itemize}\itemsep=0pt
\item If $e_1<e_2<0$, then $\tau_{{\bf c}}$ is strictly negative with minimum value $e_1$ and maximum value $e_2$, attained, respectively, at $s\equiv \omega_{{\bf c}}\mod 2\omega_{{\bf c}}$ and at $s\equiv 0\mod 2\omega_{{\bf c}}$ (see Figure~\ref{FIG5}).
\item If $0<e_1<e_2$, then $\tau_{{\bf c}}$ is strictly positive, with minimum value $e_1$ and maximum value $e_2$, attained, respectively, {at} $s\equiv 0\mod 2\omega_{{\bf c}}$ and at $s\equiv \omega_{{\bf c}}\mod 2\omega_{{\bf c}}$.
\end{itemize}
Observe that $\tau_{{\bf c}}$ is the solution of the Cauchy problem
\begin{gather}\label{twistnB1'}
\begin{split}
(i)\quad &{\tau''}=\tau^2-9c_1\tau^{-2}\big(1-c_1\tau^{-1}\big),\qquad \tau(0)=e_2,\qquad {\tau'}(0)=0 \qquad {\rm if}\ e_1<e_2<0,\\
(ii)\quad &{\tau''}=\tau^2-9c_1\tau^{-2}\big(1-c_1\tau^{-1}\big),\qquad \tau(0)=e_1,\qquad {\tau'}(0)=0 \qquad {\rm if}\ 0<e_1<e_2.
 \end{split}
 \end{gather}

\begin{figure}[t]\centering
\includegraphics[height=6cm,width=6cm]{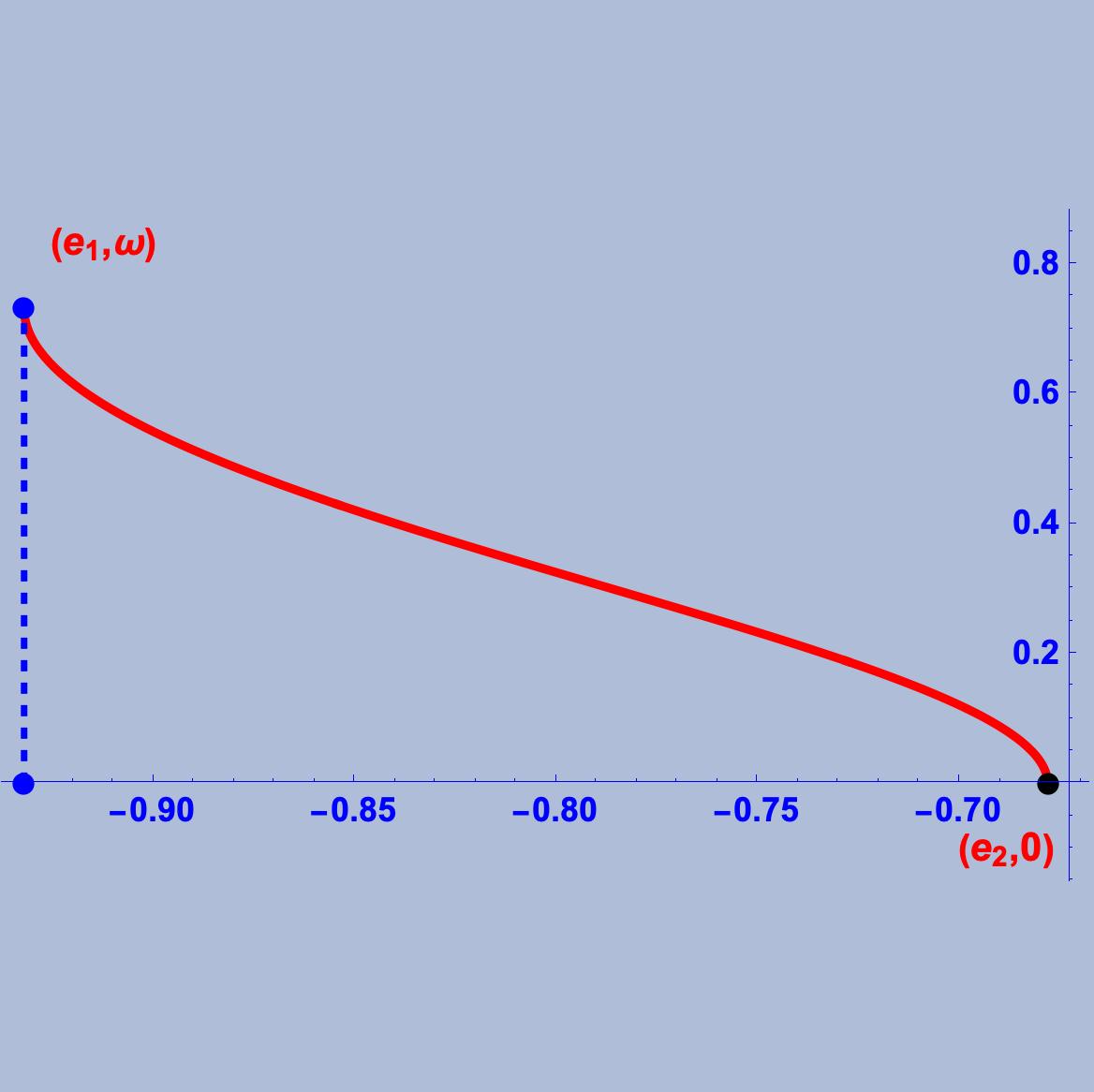}\qquad
\includegraphics[height=6cm,width=6cm]{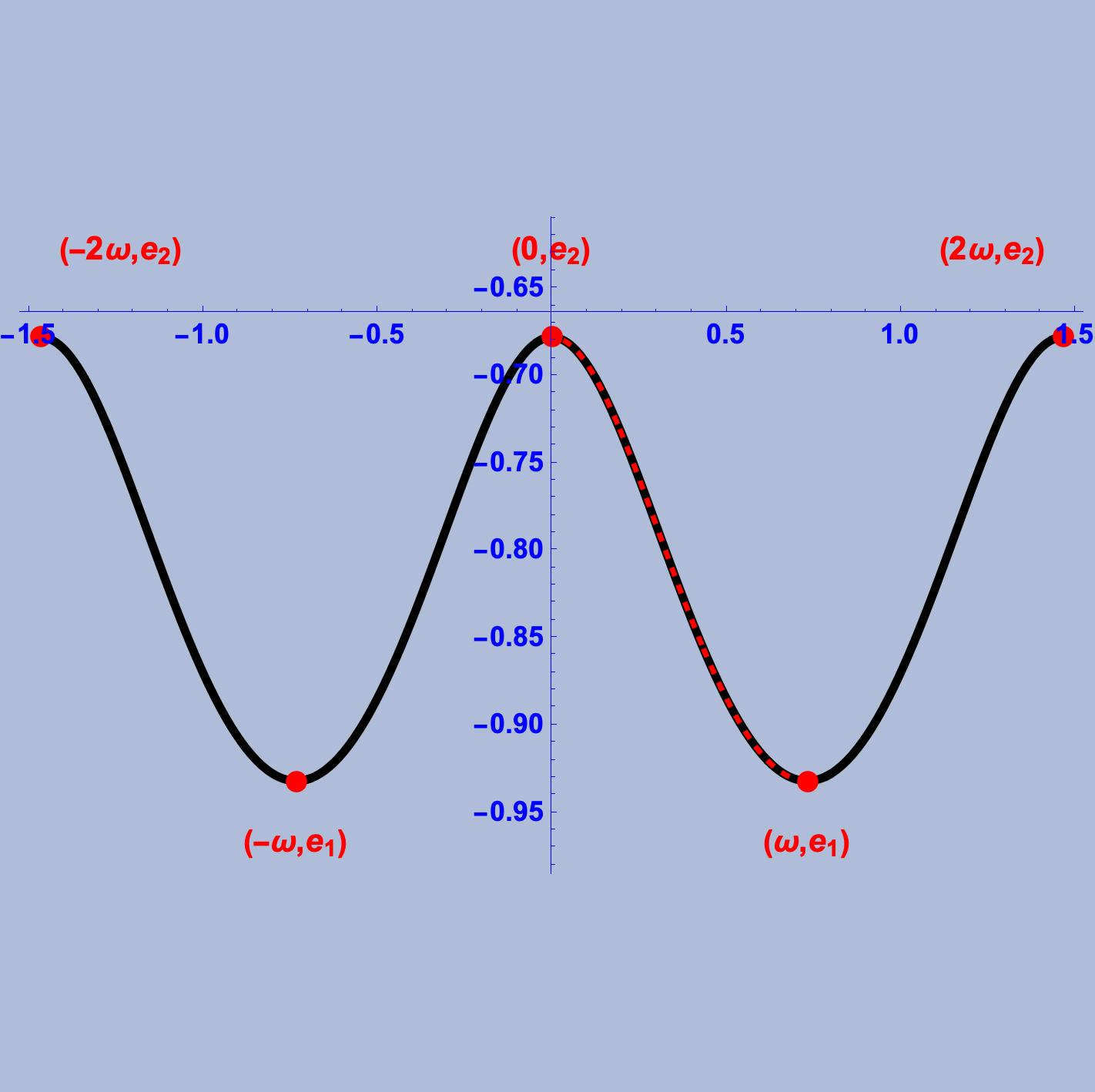}
\caption{On the left: the graph of the function $h_{{\bf c}}$ for a critical curve of type ${\mathcal B}'$, with modulus ${\bf c}\approx(-0.828424,-8.349417)\in {\mathcal B}$. On the right: the graph of the twist, an even periodic function with least period $2\omega_{{\bf c}}$. The lowest roots $e_1$ and $e_2$ are negative.}\label{FIG5}
\end{figure}

\subsubsection[The twist of a critical curve of type B'']{The twist of a critical curve of type $\boldsymbol{{\mathcal B}''}$}

The twist of a critical curve of type ${\mathcal B}''$ can be constructed as in the case of a critical curve of
type ${\mathcal A}$.
More precisely, let $e_3>0$ be the highest real root of ${\rm P}_{{\bf c}}$ and $\omega_{{\bf c}} $ be the
improper hyperelliptic integral of the first kind given by
\[
 \omega_{{\bf c}} = \sqrt{\frac{3}{2}}\int_{e_3}^{+\infty}\frac{\tau {\rm d}\tau}{\sqrt{{\rm P}_{{\bf c}}(\tau)}}>0.
\]
Let $h_{{\bf c}}(\tau)$ be the incomplete {hyperelliptic} integral
\[
 h_{{\bf c}}(\tau) = \sqrt{\frac{3}{2}}\int_{e_3}^{\tau}\frac{u {\rm d}u}{\sqrt{{\rm P}_{{\bf c}}(u)}},\qquad \tau \ge e_3.
\]
Then, $h_{{\bf c}}$ is a strictly increasing diffeomorphism of $[e_3,+\infty)$ onto $[0,\omega_{{\bf c}})$.
The twist is the unique even function
$\tau_{{\bf c}}\colon(-\omega_{{\bf c}},\omega_{{\bf c}})\to \R$, such that $\tau_{{\bf c}}=h_{{\bf c}}^{-1}$ on
$[0,\omega_{{\bf c}})$.
The maximal interval of definition of $\tau_{{\bf c}}$ is ${\rm J}_{{\bf c}}=(-\omega_{{\bf c}},\omega_{{\bf c}})$.
The function $\tau_{{\bf c}}$ is positive, with vertical asymptotes as
$s\to \mp \omega_{{\bf c}}^{\pm}$, and
is the solution of the Cauchy problem
\[
 {\tau''}=\tau^2-9c_1\tau^{-2}\big(1-c_1\tau^{-1}\big),\qquad \tau(0)=e_3,\qquad {\tau'}(0)=0.
\]

\subsubsection[The twist of a critical curve of type C with c\_1 neq 0]{The twist of a critical curve of type $\boldsymbol{\mathcal C}$ with $\boldsymbol{c_1\neq 0}$}

The twist of a critical curve of type ${\mathcal C}$,
with $c_1\neq 0$, can be constructed as for curves of types ${\mathcal A}$ or ${\mathcal B}''$.
Let $e_3>0$ be simple real root of
${\rm P}_{{\bf c}}$ and $\omega_{{\bf c}}$ be the improper {elliptic} integral of the first kind
\[
 \omega_{{\bf c}} = \sqrt{\frac{3}{2}}\int_{e_3}^{+\infty}\frac{\tau {\rm d}\tau}{\sqrt{{\rm P}_{{\bf c}}(\tau)}}>0.
\]
Let $h_{{\bf c}}(\tau)$ be the incomplete elliptic integral
\[
 h_{{\bf c}}(\tau) = \sqrt{\frac{3}{2}}\int_{e_3}^{\tau}\frac{u {\rm d}u}{\sqrt{{\rm P}_{{\bf c}}(u)}},\qquad \tau\ge e_3.
\]
Then, $h_{{\bf c}}$ is a strictly increasing diffeomorphism of $[e_3,+\infty)$ onto $[0,\omega_{{\bf c}})$.
The twist $\tau_{{\bf c}}$ is the unique even function
$\tau_{{\bf c}}\colon(-\omega_{{\bf c}},\omega_{{\bf c}})\to \R$, such that $\tau_{{\bf c}}=h_{{\bf c}}^{-1}$ on
$[0,\omega_{{\bf c}})$.
The maximal interval of definition of $\tau_{{\bf c}}$ is ${\rm J}_{{\bf c}}=(-\omega_{{\bf c}},\omega_{{\bf c}})$.
The twist is positive, with vertical asymptotes as $s\to \mp\omega_{{\bf c}}^{\pm}$. Note that $\tau_{{\bf c}}$ is the solution of the Cauchy problem
\[
{\tau''}=\tau^2-9c_1\tau^{-2}\big(1-c_1\tau^{-1}\big),\qquad \tau(0)=e_3,\qquad {\tau'}(0)=0.
\]

\subsubsection[The twist of a critical curve with c\_1= 0]{The twist of a critical curve with $\boldsymbol{c_1= 0}$}

If $c_1=0$, the bending vanishes identically and the twist is a solution of the second order ODE ${\tau''}-\tau^2=0$.
Then,
\[
\tau(s)=\sqrt[3]{6}\wp\bigg(\frac{s+a}{\sqrt[3]{6}}\bigg|0,g_3\bigg),\qquad g_3=-\sqrt[3]{2/243} c_2, \qquad \kappa(s)=0,
\]
where $a$ is an unessential constant and $\wp(-,g_2,g_3)$ is the Weierstrass function with invariants~$g_2$,~$g_3$.

\begin{figure}[t]\centering
\includegraphics[height=6cm,width=7cm]{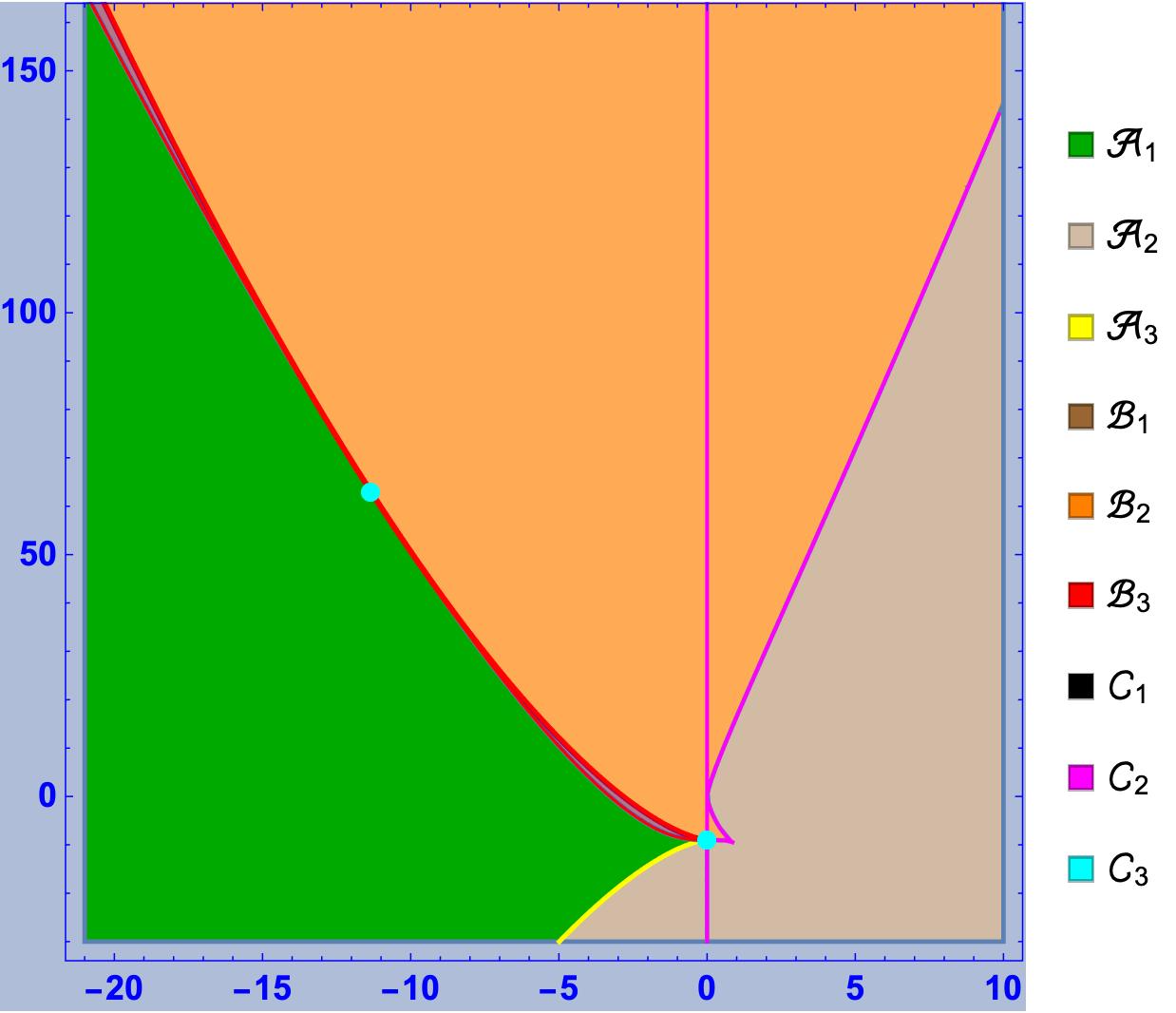}
\caption{The nine regions ${\mathcal A}_j$, ${\mathcal B}_j$ and ${\mathcal C}_j$, $j=1,2,3$.}\label{FIG6}
\end{figure}

\subsection[Orbit types and the twelve classes of critical curves with nonconstant twist]{Orbit types and the twelve classes of critical curves\\ with nonconstant twist}

The moduli of the critical curves can be classified depending on the properties of the eigenvalues of the momenta.

\begin{Definition}\label{orbittype}
For $\mathbf{c} = (c_1,c_2) \in \R^2$, let $\Delta_1({\bf c})=-27\big(32c_1^3+9(9+c_2)^2\big)$ be the
discriminant of the cubic polynomial ${\rm Q}_{{\bf c}}$ (cf.~\eqref{Cubic_Poly}).
We say that
 ${\bf c}\in \R^2$ is
\begin{enumerate}\itemsep=0pt
\item Of \textit{orbit type} 1 (in symbols, ${\bf c}\in {\rm OT}_1$) if $\Delta_1({\bf c})>0$;
the momentum of a critical curve with modulus ${\bf c}\in {\rm OT}_1$ has three distinct real eigenvalues:
$\lambda_1=-(\lambda_2+\lambda_3)<0<\lambda_2<\lambda_3$.
\item
Of \textit{orbit type} 2 (in symbols, ${\bf c}\in {\rm OT}_2$) if $\Delta_1({\bf c})<0$;
the momentum of a critical curve
with modulus ${\bf c}\in {\rm OT}_2$ has a real eigenvalue $\lambda_1$ and two complex conjugate roots: $\lambda_2$,
with positive imaginary part, and $\lambda_{3}=\overline{\lambda}_2$.
%
\item Of \textit{orbit type} 3 (in symbols, ${\bf c}\in {\rm OT}_3$) if $\Delta_1({\bf c})=0$;
the momentum of a critical curve with modulus ${\bf c}\in {\rm OT}_3$ has an eigenvalue with algebraic multiplicity
greater than one ($>1$).
\end{enumerate}
\end{Definition}

Correspondingly, $\R^2$ is partitioned into nine regions (see Figure~\ref{FIG6}):
\[
{\mathcal A}_j={\mathcal A}\cap {\rm OT}_j,\qquad {\mathcal B}_j={\mathcal B}\cap {\rm OT}_j,\qquad
 {\mathcal C}_j={\mathcal C}\cap {\rm OT}_j,\qquad j=1,2,3.
\]

\begin{Definition}
Let $\gamma$ be a critical curve with modulus ${\bf c}$ and $j\in \{1,2,3\}$.
We say that $\gamma$ is of \textit{type} ${\mathcal A}_j$ if ${\bf c}\in {\mathcal A}_j$;
of \textit{type} ${\mathcal B}'_j$ if ${\bf c}\in {\mathcal B}_j$
and the image of its signature $\sigma_{\gamma}$ is compact;
of \textit{type}~${\mathcal B}''_j$ if ${\bf c}\in {\mathcal B}_j$ and the image of $\sigma_{\gamma}$ is unbounded;
and
of \textit{type} ${\mathcal C}_j$ if ${\bf c}\in {\mathcal C}_j$, $j=1,2,3$.
\end{Definition}

\begin{Remark}
The only critical curves with periodic twist are those of the types ${\mathcal B}'_j$, $j=1,2,3$.
Consequently, critical curves of the other types cannot be closed.
 \end{Remark}

\begin{Remark}\label{r:connectedcomponents}
${\mathcal B}_1$ lies in the half-plane $\{(c_1,c_2) \mid c_1<0\}$;
it is bounded below by $\Xi'=\{{\bf c}\in \Xi \mid c_1< 0\}$ and above by
$\Delta'=\big\{{\bf c}\in \R^2 \mid \Delta_1({\bf c})=0,\, c_2> -9\big\}$.
The curves $\Xi'$ and $\Delta'$ intersect each other tangentially
at ${\bf c}'=(c'_1,c'_2)\approx (-11.339754, 63.004420)$ (see Figure~\ref{FIG7}).
Thus, ${\mathcal B}_1$ has two connected components:
\[
 {\mathcal B}_1^-=\{{\bf c}\in {\mathcal B}_1 \mid c_1\in (c'_1,-9) \}, \qquad
 {\mathcal B}_1^+=\{{\bf c}\in {\mathcal B}_1 \mid c_1 > c'_1 \}.
\]
Referring to Remark~\ref{remarkseparatrix}, $\Xi'$ is parametrized by the restriction of $\widetilde{\xi}$ to the interval
$\hat{{\rm J}}_{\xi}=(\pi/2,\pi +\arctan(n_*))$. Let $t'$ be the point of
$\hat{{\rm J}}_{\xi}$ such that $\widetilde{\xi}(t')={\bf c}'$, ($t'\approx 2.3008$). Put $\hat{{\rm J}}_{\xi}^-=(\pi/2,t')$ and
 $\hat{{\rm J}}_{\xi}^+=(t',\pi +\arctan(n_*))$. The restriction of $\widetilde{\xi}$ to $\hat{{\rm J}}_{\xi}^-$ is a parametrization of
 $\Xi_-=\{{\bf c}\in \Xi'\mid c_1\in (c'_1,0)\}$ and the restriction to $\hat{{\rm J}}_{\xi}^+$ is a parametrization of
 $\Xi_+=\{{\bf c}\in \Xi'\mid c_1<c'_1\}$.
 Consequently, ${\mathcal B}_1^{\pm}$ are parametrized by
\begin{equation}\label{parconc}
 \psi_{\pm} \colon \ K_{\pm} \ni (t,s)\longmapsto \big(\widetilde{\xi}(t)-p(t)\big)s+p(t),
 \end{equation}
where $ K_{\pm}$ are the rectangles $\hat{{\rm J}}_{\xi}^{\pm}\times (0,1)$ and
\[
 p(t)=\bigg(\widetilde{\xi}_1(t),\frac{1}{3}\bigg(4\sqrt{-2\widetilde{\xi}_1(t)^3}-27\bigg)\bigg).
\]
\end{Remark}

\begin{figure}[t]\centering
\includegraphics[height=6cm,width=6cm]{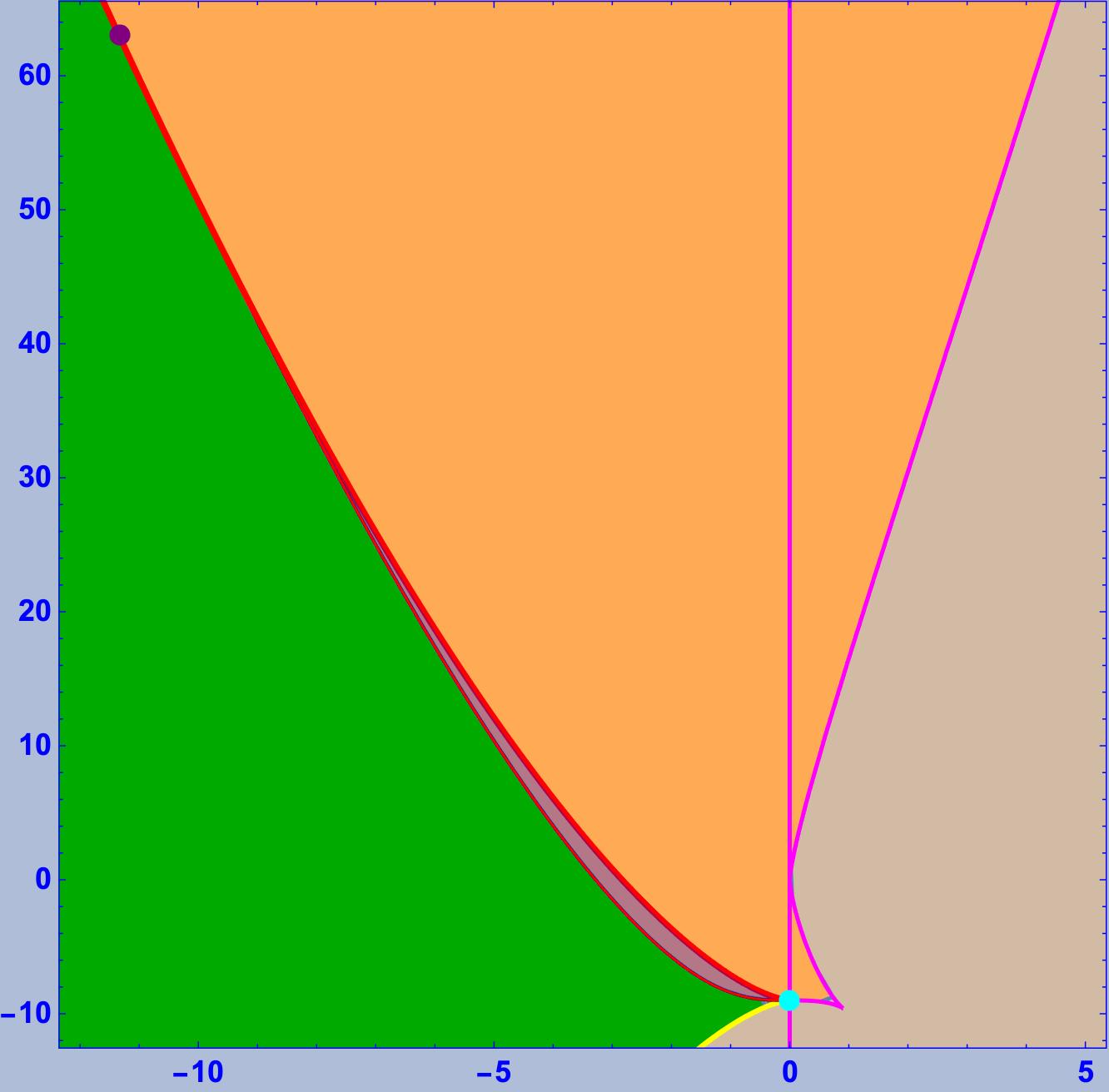}\qquad
\includegraphics[height=6cm,width=6cm]{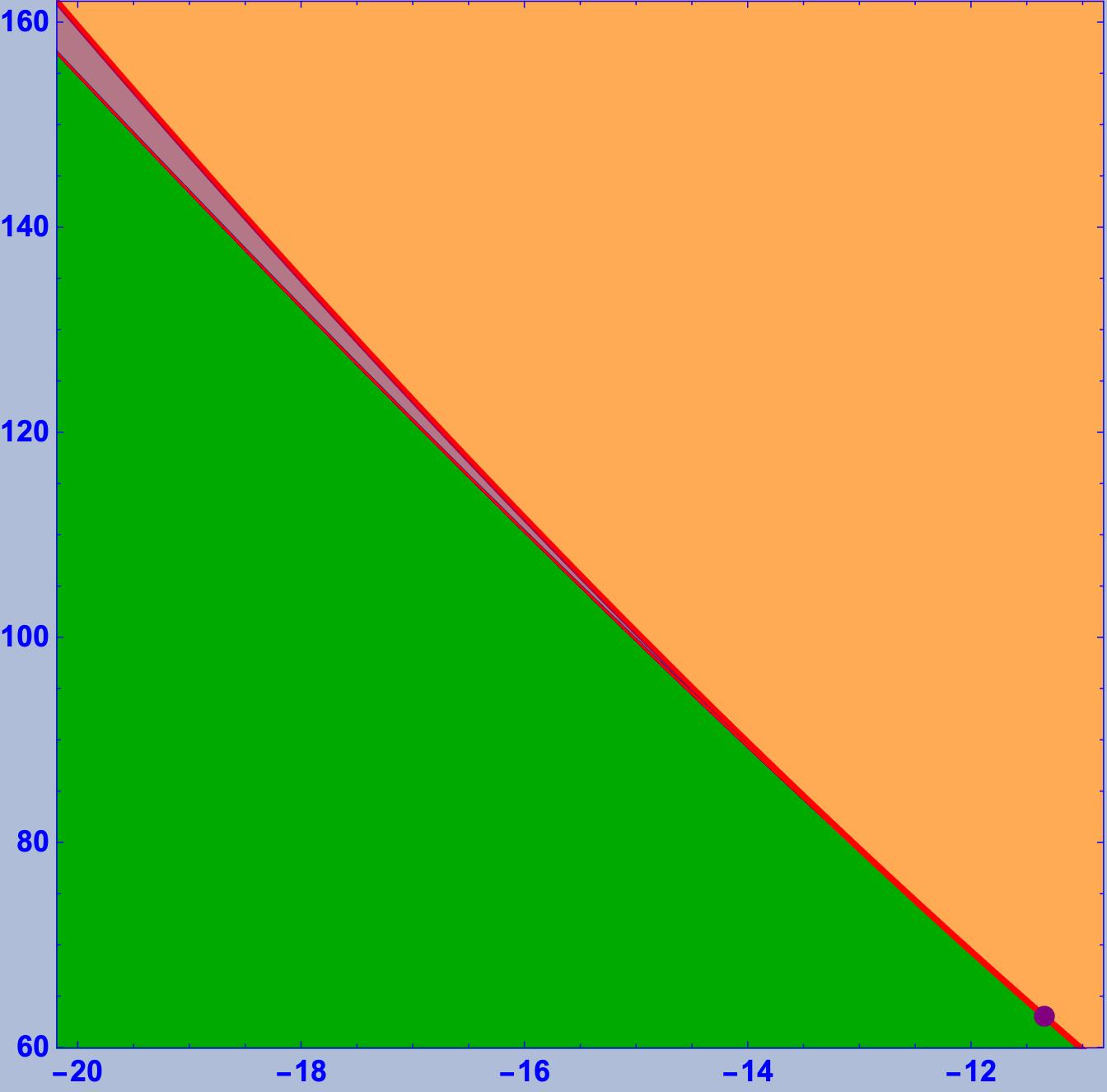}
\caption{On the left: the connected component ${\mathcal B}_1^-$ (dark brown) of ${\mathcal B}_1$.
The point coloured in cyan is the inflection point of $\Xi$ (the union of the arcs coloured in black and magenta
and of the two points) and the cusp of $\Delta_1=0$ (the union of the yellow and red arcs and of the two points).
The point coloured in purple is the point of tangential contact. On the right: the connected
component ${\mathcal B}_1^+$ (dark brow) of ${\mathcal B}_1$.}\label{FIG7}
\end{figure}

\section{Integrability by quadratures}\label{s:integrability}

\subsection{Integrability by quadratures of general critical curves}

\begin{Definition}\label{def:general}
Let $\Delta_2$ be the polynomial
\[
{\Delta_2({\bf c})=9c_1^3\big(c_1^3+216\big)+6c_1^3c_2(c_2+36)+(c_2+9)(c_2+18)^3.}
\]
A critical curve $\gamma$ with modulus ${\bf c}$ is said to be \emph{general} if $\Delta_1({\bf c})\Delta_2({\bf c})\neq 0$.
Since $\Delta_1({\bf c})\neq 0$, the momentum ${\mathfrak M}_{\gamma}$ of a general critical curve $\gamma$ has
three distinct eigenvalues $\lambda_1$, $\lambda_2$, $\lambda_3$, sorted as in Definition~\ref{orbittype}.
\end{Definition}

Let ${\rm J}$ be the maximal interval of definition of the twist (it can be computed in terms of the modulus).
Define ${\bf y}_j\colon{\rm J}\to \C^{1,2}$, $j=1,2,3$, by
\begin{equation}\label{sections}
 \mathbf{y}_j={}^t\!\biggl(\tau(3-{\rm i}{\tau'})-\lambda_j^2-3c_1,  9-\frac{9c_1}{\tau}-\lambda_j\tau-3{\rm i}{\tau'},  {\rm i}\big(\tau^2-3\lambda_j\big)\biggr).
 \end{equation}
Let ${\rm V} \colon {\rm J} \to \mathfrak{gl}(3,\C)$ be the matrix-valued map with column vectors ${\bf y}_1$, ${\bf y}_2$ and ${\bf y}_3$.
Let ${\rm D}(z_1,z_2,z_3)$ denote the diagonal matrix with $z_j$ as the $j$th element on the diagonal.
Recall that, if $c_1\neq 0$, then $\tau$ is nowhere zero.

We can prove the following.

\begin{thmx} \label{thmss:integrability1}
Let $\gamma\colon \mathrm{J} \to \S$ be a general critical curve. The functions ${\rm{det}}({\rm V})$ and $\tau^2-3\lambda_j$,
$j=1,2,3$, are nowhere zero.
Let $r_j$ be continuous determinations of $\sqrt{\tau^2-3\lambda_j}$ and let $\phi_j$ be the functions
defined by\footnote{If $c_1\neq 0$, the denominator of the integrand in nowhere zero and the $\phi_j$ are real-analytic.
If $c_1=0$, the integrand reduces to $\big(3\tau-\lambda_j^2\big)\big(3\lambda_j-\tau^2\big)^{-1}$. Thus, also in this case the functions
$\phi_j$ are real-analytic.}
\begin{equation}\label{angular}
 \phi_j(s)=\int_0^s \frac{3c_1\lambda_j-\big(4c_1+\lambda_j^2\big)\tau^2(u)+3\tau^3(u)}{\tau^2(u)\big(3\lambda_j-\tau^2(u)\big)}{\rm d}u.
 \end{equation}
Then, $\gamma$ is congruent to
\[
{\rm J} \ni s \longmapsto \big[{\rm M}{\rm D}\big(r_1{\rm e}^{{\rm i}\phi_1}, r_2{\rm e}^{{\rm i}\phi_2}, r_3{\rm e}^{{\rm i}\phi_3}\big) {\rm V}^{-1} \mathbf{e}_1\big] \in \S,
\]
where ${\rm M}={\rm V}(0) {\rm D}(r_1(0),r_2(0),r_3(0))^{-1}$.
\end{thmx}

\begin{proof} The proof of Theorem~\ref{thmss:integrability1} is organized into three lemmas.

\begin{lemmaB}\label{lemma1Main1}
The following statements hold true:
\begin{enumerate}\itemsep=0pt
\item[$(1)$] if the momentum has three distinct real eigenvalues, then $\pm \sqrt{3\lambda_2}$ and $\pm\sqrt{3\lambda_3}$
cannot be roots of ${\rm P}_{{\bf c}}$;
\item[$(2)$] if the momentum has two complex conjugate eigenvalues and a positive real eigenvalue $\lambda_1$, then $\pm \sqrt{3\lambda_1}$ cannot be roots of ${\rm P}_{{\bf c}}$.
\end{enumerate}
\end{lemmaB}

\begin{proof}[Proof of Lemma B\ref{lemma1Main1}]
First, note that the image of the parametrized curve
\[
 \alpha(t)=\biggl(-t\bigg(\frac{t^3}{3}+\sqrt{3}\bigg),t^3\bigg(\frac{t^3}{3}+2\sqrt{3}\bigg)-9\biggr)
\]
is contained in the zero locus of $\Delta_2$. This can be proved by a direct computation. Secondly,
from the expression of $\mathrm{Q}_\mathbf{c}$, it follows that
\begin{gather}
c_1=-\frac{1}{6}\big(\lambda_2^2+\lambda_2\lambda_3+\lambda_3^2\big)=-\frac{1}{6}\big(\lambda_1^2+\lambda_1\lambda_2+\lambda_2^2\big), \nonumber\\
 c_2=\frac{1}{3}\big(\lambda_2^2\lambda_3+\lambda_2\lambda_3^2-27\big)=\frac{1}{3}\big(\lambda_1^2\lambda_2+\lambda_1\lambda_2^2-27\big).\label{clemma1Main1}
 \end{gather}

1. Suppose that the momentum has three distinct real eigenvalues.
By contradiction, suppose that $\sqrt{3\lambda_2}$ is a root of ${\rm P}_{{\bf c}}$. Then
\begin{gather*}
0=-\frac{8}{3}{\rm P}_{{\bf c}}\big(\sqrt{3\lambda_2}\big) =\lambda_3^4+2\lambda_2\lambda_3^3-\big(\lambda_2^2-12\sqrt{3}\lambda_2^{1/2}\big)\lambda_3^2-2\big(\lambda_2^3-6\sqrt{3}\lambda_2^{3/2}\big)\lambda_3\\
\hphantom{0=-\frac{8}{3}{\rm P}_{{\bf c}}\big(\sqrt{3\lambda_2}\big) =}{}
+\big(\lambda_2^4-12\sqrt{3}\lambda_2^{5/2}+108\lambda_2\big).
\end{gather*}
Solving this equation with respect to $\lambda_3$, {taking into account that $\lambda_3>0$}, we obtain
\[
 \lambda_3=\frac{1}{2}\biggl(-\lambda_2+\sqrt{5\lambda_2^2-24\sqrt{3\lambda_2}}\biggr).
\]
Substituting into~\eqref{clemma1Main1}, we find
\[
 c_1=\sqrt{3\lambda_2}-\frac{1}{3}\lambda_2^2,\qquad c_2=-9-2\sqrt{3}\lambda_2^{3/2}+\frac{1}{3}\lambda_2^3.
\]
Then, ${\bf c}=\alpha\bigl(-\sqrt{\lambda_2}\bigr)$. This implies that ${\bf c}$ belongs to the zero locus of $\Delta_2$,
which is a~contradiction.
By an analogous argument, we prove that also $-\sqrt{3\lambda_2}$ cannot be a root of~${\rm P}_{{\bf c}}$.
By interchanging the role of $\lambda_2$ and $\lambda_3$ and arguing as above, it follows that
also~$\pm\sqrt{3\lambda_3}$ cannot be roots of ${\rm P}_{{\bf c}}$.

2. Next, suppose that the momentum has two complex conjugate eigenvalues and a nonnegative real
eigenvalue $\lambda_1$.
Recall that the eigenvalues are sorted so that the imaginary part of $\lambda_2$ is positive. By contradiction,
suppose that $\sqrt{3\lambda_1}$ is a root of ${\rm P}_{{\bf c}}$. Then,
 \begin{gather*}
 0=-\frac{8}{3}{\rm P}_{{\bf c}}\big(\sqrt{3\lambda_1}\big) =\lambda_2^4+2\lambda_1\lambda_2^3-\big(\lambda_1^2-12\sqrt{3}\lambda_1^{1/2}\big)\lambda_2^2-2\big(\lambda_1^3-6\sqrt{3}\lambda_1^{3/2}\big)\lambda_2\\
\hphantom{0=-\frac{8}{3}{\rm P}_{{\bf c}}\big(\sqrt{3\lambda_1}\big)=}{}
+\big(\lambda_1^4-12\sqrt{3}\lambda_1^{5/2}+108\lambda_1\big).
\end{gather*}

Solving this equation with respect to $\lambda_2$, taking into account that the imaginary part of~$\lambda_2$ is positive,
we find
\[
 \lambda_2=\frac{1}{2}\biggl(-\lambda_1+\sqrt{5\lambda_1^2-24\sqrt{3\lambda_1}}\biggr).
\]
Substituting into~\eqref{clemma1Main1} yields ${\bf c}=\alpha\bigl(-\sqrt{\lambda_1}\bigr)$. Thus, ${\bf c}$ is a root of $\Delta_2$,
which is a contradiction. An analogous argument shows that $-\sqrt{3\lambda_1}$ cannot be a root of ${\rm P}_{{\bf c}}$.
This concludes the proof of the lemma.
\end{proof}

\begin{lemmaB}\label{lemma2Main1}
${\rm{det}}({\rm V})(s)\neq 0$, for every $s\in {\rm J}_{\gamma}$.
\end{lemmaB}

\begin{proof}[Proof of Lemma B\ref{lemma2Main1}]
Let ${\mathbb L}_j$ be the 1-dimensional eigenspaces of the momentum
${\mathfrak M}_{\gamma}$ relative to the eigenvalues $\lambda_j$.
Let $L$ be as in (\ref{obser-L}). By Corollary~\ref{cor-thmA} of Theorem~\ref{thmA},
we have ${{\mathcal F} L {\mathcal F}^{-1}={\mathfrak M}}$,
where ${\mathcal F}$ is a
Wilczynski
frame field along $\gamma$.
Then,
${L}(s)$ and ${\mathfrak M}$ have the same eigenvalues. Next, consider the line bundles
\[
 \Lambda_j=\big\{(s,{\bf y})\in {\rm J}_{\gamma}\times \C^{1,2} \mid {L}(s){\bf y}=\lambda_j{\bf y}\big\},\qquad j=1,2,3.
\]
Note that $(s,{\bf y})\in \Lambda_j$ if and only if ${\mathcal F}(s){\bf y}\in {\mathbb L}_j$. Let ${\bf y}_j$, $j=1,2,3$,
be as in (\ref{sections}). A direct computation shows that ${L}{\bf y}_j=\lambda_j {\bf y}_j$. Thus,
${\bf y}_j$
is a cross section of the eigenbundle $\Lambda_j$. Hence, ${\rm{det}}({\rm V})(s)\neq 0$ if and only if
 ${\bf y}_j(s)\neq \vec{0}$, for every $s$.

{\bf Case I.} The eigenvalues of the momentum are real and distinct. Let ${\rm y}_j^i$, $i=1,2,3$, denote
the components of ${\bf y}_j$. Since $\lambda_1$ is negative, it follows from~\eqref{sections} that ${\rm y}^3_1(s)\neq 0$,
for every $s$, and hence $\mathbf{y}_1(s)\neq \vec{0}$. We prove that {${\bf y}_2(s)\neq \vec{0}$.
Suppose, by contradiction, that ${\bf y}_2(s_*)=\vec{0}$}, for some $s_*\in {\rm J}_{\gamma}$.
From ${\rm y}^1_2(s_*)={\rm y}^2_2(s_*)=0$, it follows that {$\tau'(s_*)=0$}.
Hence $e:=\tau(s_*)$ is a root of~${\rm P}_{{\bf c}}$. From ${\rm y}^3_2(s_*)=0$,
it follows that $e=\pm \sqrt{3\lambda_2}$, which contradicts Lemma B\ref{lemma1Main1}. An analogous argument leads to
the conclusion that ${\bf y}_3(s)\neq \vec{0}$, for every $s\in {\rm J}_{\gamma}$.

{\bf Case II.} The momentum has a real eigenvalue $\lambda_1$ and two complex conjugate eigenvalues~$\lambda_2$,~$\lambda_3$
($\lambda_2$ with positive imaginary part).
Since $\lambda_2$ and $\lambda_3$ have nonzero imaginary parts and~$\tau$ is real valued,
${\rm y}_2^3(s)\neq 0$ and ${\rm y}_3^3(s)\neq 0$, for every $s$. If $\lambda_1<0$, then ${\rm y}_1^3(s)\neq 0$,
for every~$s$.
If $\lambda_1\ge 0$, suppose, by contradiction, that ${\bf y}_1(s_*)=\vec{0}$.
From ${\rm y}_1^1(s_*)={\rm y}^2_1(s_*)=0$,
we infer that $\tau'(s_*)=0$. Hence $e=\tau(s_*)$ is a root of ${\rm P}_{{\bf c}}$.
From ${\rm y}^3_1(s_*)=0$, we have $e=\pm \sqrt{3\lambda_1}$, which contradicts Lemma~B\ref{lemma1Main1}.
\end{proof}

We are now in a position to conclude the proof. For $j=1,2,3$, let ${\bf w}_j$ be defined by
\[
 {\bf w}_j={\mathcal F}{\bf y}_j\colon \ {\rm J}_{\gamma}\to \C^{1,2}.
\]
Then, ${\bf w}_j(s)\in {\mathbb L}_j$ and ${\bf w}_j(s)\neq \vec{0}$, for every $s$.
Thus, there exist smooth functions $\Phi_j\colon{\rm J}_{\gamma}\to \C$, such that
${\bf w}'_j=\Phi_j {\bf w}_j$. From~\eqref{MCE-W-frame}, we have
\begin{equation}\label{intf}
 \Phi_j {\bf y}_j={\bf y}'_j+K{\bf y}_j,\qquad j=1,2,3,
 \end{equation}
where
\[
K= \begin{pmatrix}
{\rm i}c_1\tau^{-2} & -{\rm i} & \tau\\
0 &-2{\rm i}c_1\tau^{-2} & 1 \\
1 & 0&{\rm i}c_1\tau^{-2}\\
 \end{pmatrix}.
\]
Then, the third component of ${\bf y}'_j+K{\bf y}_j$ is equal to
\[
 3\tau+3c_1\lambda_j\tau^{-2}-\big(\lambda_j^2+4c_1\big)+{\rm i}\tau{\tau}'.
\]
Hence, using~\eqref{intf} we obtain
\begin{equation}\label{intf2}
 \Phi_j=-\frac{\tau \tau'}{3\lambda_j-\tau^2}
 +{\rm i}\frac{3c_1\lambda_j-\big(4c_1+\lambda_j^2\big)\tau^2+3\tau^3}{\tau^2\big(3\lambda_j-\tau^2\big)}.
 \end{equation}

\begin{lemmaB}\label{lemma3Main1}
The functions $3\lambda_j-\tau^2$, $j=1,2,3$, are nowhere zero.
 \end{lemmaB}

\begin{proof}[Proof of Lemma B\ref{lemma3Main1}]
 The statement is obvious if $\lambda_j$ is real and negative or complex, with nonzero imaginary part. If
$\lambda_j$ is real non-negative,
the smoothness of $\Phi_j$ implies that $\tau \tau'\big(3\lambda_j-\tau^2\big)^{-1}$ is differentiable.
Then $\big(3\lambda_j-\tau^2\big)(s)\neq 0$, for every $s$, such that $\tau(s)\tau'(s)\neq 0$. {If $\tau(s)\tau'(s)= 0$,
it follows that $\tau(s)$ is a root of the polynomial ${\rm P}_{{\bf c}}$. Therefore, by Lemma B\ref{lemma1Main1}, we have
 that $\big(3\lambda_j-\tau^2\big)(s)\neq 0$.}
\end{proof}

From~\eqref{intf2}, we have
\[
 \int_0^s \Phi_j {\rm d}u =\log \Big(\sqrt{\tau^2-3\lambda_j}\Big)+{\rm i}\phi_j+b_j, \qquad j=1,2,3,
\]
where $b_j$ is a constant of integration,
$\sqrt{\tau^2-3\lambda_j}$
is a continuous {determination} of the square root of
$\tau^2-3\lambda_j$
and
$\log\big(\sqrt{\tau^2-3\lambda_j}\big)$
is a continuous determination of the logarithm of
$\sqrt{\tau^2-3\lambda_j}$.
{Since} ${\bf w}'_j=\Phi_j {\bf w}_j$, we obtain
\[
 {\mathcal F}{\bf y}_j r_j^{-1}{\rm e}^{-{\rm i}\phi_j} = {\bf m}_j, \qquad j=1,2,3,
\]
where $ {\bf m}_j$ is a constant vector belonging to the eigenspace ${\mathbb L}_j$ of ${\mathfrak M}$. This implies
\begin{equation*}
 {\mathcal F}={\rm M}\mathrm{D}\big(r_1{\rm e}^{{\rm i}\phi_1},r_2{\rm e}^{{\rm i}\phi_2},r_2{\rm e}^{{\rm i}\phi_2}\big){\rm V}^{-1},
 \end{equation*}
where ${\rm M}$ is an {invertible} matrix such that ${\rm M}^{-1} {\mathfrak M} {\rm M}=\mathrm{D}(\lambda_1,\lambda_2,\lambda_3)$.
By possibly replacing~$\gamma$ with a congruent curve, we may suppose that
${\mathcal F}(0)=
I_3$. Then, since $\phi_j(0)=0$, we have
${\rm M}={\rm V}(0){\rm D}(r_1(0),r_2(0),r_3(0))^{-1}$. This concludes the proof of Theorem~\ref{thmss:integrability1}.
\end{proof}

\subsection[Integrability by quadratures of general critical curves of type B\_1']{Integrability by quadratures of general critical curves of type $\boldsymbol{{\mathcal B}_1'}$}

We now specialize the above procedure to the case of general critical curves of type ${\mathcal B}_1'$ (i.e., general critical
curves with modulus ${\bf c}\in {\mathcal B}_1$ and with periodic twist). Let ${\mathcal M}_+'$ be as in
\eqref{connectedcomponents}.
Since ${\mathcal B}_1$ is contained in ${\mathcal M}_+'$, the lowest roots $e_1$ and $e_2$ of ${\rm P}_{{\bf c}}$ are negative,
for every ${\bf c}\in {\mathcal B}_1$ (cf.\ Remark~\ref{remarksignrootsP}).

\begin{Lemma} \label{Lemma1ss:integrability2}
Let $\gamma$ be a general critical curve of type ${\mathcal B}_1'$. The $\lambda_1$-eigenspace of the momentum is spacelike.
\end{Lemma}

\begin{proof}
Let ${\bf y}_j$ be as in~\eqref{sections}. Then ${\mathcal F}(s){\bf y}_j(s)$ belongs to the
$\lambda_j$-eigenspace of ${\mathfrak M}$, for every $s\in \R$.
Using the conservation law {$\frac32\tau^2 (\tau')^2={\rm P}_{{\bf c}}(\tau)$} (cf.~\eqref{cons-momentum-bis})
and taking into account that $\lambda_j^3+6c_1\lambda_j+3(9+c_2)=0$, {we} compute
\[
 \langle {\mathcal F}{\bf y}_j,{\mathcal F}{\bf y}_j\rangle
 = \langle {\bf y}_j,{\bf y}_j\rangle =3\big(\tau^2-3\lambda_j\big)\big(2c_1+\lambda_j^2\big).
\]
Moreover, since $\lambda_1=-(\lambda_2+\lambda_3)$ and $c_1=-\big(\lambda_2^2+\lambda_2\lambda_3+\lambda_3^2\big)/6$,
we have
\begin{gather*}
2c_1+\lambda_1^2=\frac{1}{3}(2\lambda_2+\lambda_3)(\lambda_2+2\lambda_3)>0,\\
2c_1+\lambda_3^2=\frac{1}{3}(\lambda_3-\lambda_2)(\lambda_2+2\lambda_3)>0,\\
2c_1+\lambda_2^2=-\frac{1}{3}(\lambda_3-\lambda_2)(2\lambda_2+\lambda_3)<0.
\end{gather*}
From the fact that $\lambda_1<0$, it follows that $\langle {\bf y}_1,{\bf y}_1\rangle > 0$.
This proves that the $\lambda_1$-eigenspace of the momentum is spacelike.
\end{proof}

\begin{Definition}
There are two possible cases: either the $\lambda_3$-eigenspace of ${\mathfrak M}$ is spacelike, or else is timelike.
In the first case, we say that $\gamma$ is \emph{positively polarized}, while in the second case, we say that
$\gamma$ is \emph{negatively polarized}.
\end{Definition}

\begin{Remark}
In view of the above lemma, $\gamma$ is positively polarized if and only if
{$e_1^2-3\lambda_3>0$}
and is negatively polarized if and only if
{$e_2^2-3\lambda_3<0$}.
It is a linear algebra exercise to prove the existence of ${A}\in {G}$, such that
${A}^{-1}{\mathfrak M}{A}={\mathfrak M}_{\lambda_1,\lambda_2,\lambda_3}$, where
\begin{equation}\label{canformmom}
{\mathfrak M}_{\lambda_1,\lambda_2,\lambda_3}=\begin{pmatrix}
\frac{1}{2}(\lambda_2+\lambda_3) & 0& \frac{\varepsilon{\rm i}}{2}(\lambda_2-\lambda_3)\\
0 &\lambda_1& 0 \\
- \frac{\varepsilon{\rm i}}{2}(\lambda_2-\lambda_3)&0&\frac{1}{2}(\lambda_2+\lambda_3) \\
 \end{pmatrix},
 \end{equation}
where
{$\varepsilon=\pm 1$}
accounts for the polarization of $\gamma$ (see below).
It is clear that any critical curve of type ${\mathcal B}'_1$ is congruent to a critical curve whose momentum
is in the canonical form ${\mathfrak M}_{\lambda_1,\lambda_2,\lambda_3}$.
 \end{Remark}

 \begin{Definition}
 A critical curve of type ${\mathcal B}'_1$ is said to be in a \emph{standard configuration} if its momentum is in
 the canonical form~\eqref{canformmom}. Two standard configurations with the same twist are congruent
 with respect to the left action of the maximal compact abelian subgroup
 ${\mathbb T}^2=\{{\rm A}\in {\rm G} \mid {\rm A}{\bf e}_2\wedge {\bf e}_2=0\}$.
 \end{Definition}

Let ${\bf c}\in {\mathcal B}_1$, such that {$\Delta_1({\bf c})\Delta_2({\bf c})\neq 0$}.
Let $e_1<e_2<e_3$ be the real roots of ${\rm P}_{{\bf c}}$ and let
$\lambda_1=-(\lambda_2+\lambda_3)<0<\lambda_2<\lambda_3$ be the roots of ${\rm Q}_{{\bf c}}$.
Let $\tau$ be the periodic function defined as in {the first of the}~\eqref{twistnB1'} and $\phi_j$, $j=1,2,3$,
be as in~\eqref{angular}.
Let $\rho_j$ be the constants
\begin{gather*}
\rho_1=\frac{1}{\sqrt{(2\lambda_2+\lambda_3)(\lambda_2+2\lambda_3)}},\qquad
 \rho_2=\frac{1}{\sqrt{2(\lambda_3-\lambda_2)(2\lambda_2+\lambda_3)}},\\
\rho_3=\frac{1}{\sqrt{2(\lambda_3-\lambda_2)(\lambda_2+2\lambda_3)}}
\end{gather*}
and $z_j$ be the functions
\begin{gather}
z_1=\rho_1\sqrt{3(\lambda_2+\lambda_3)+\tau^2}{\rm e}^{{\rm i}\phi_1},\qquad
z_2=\rho_2\sqrt{3\lambda_2-\tau^2}{\rm e}^{{\rm i}\phi_2},\qquad
z_3=\rho_3\sqrt{3\lambda_3-\tau^2}{\rm e}^{{\rm i}\phi_3}.\!\!\!\label{functions}
\end{gather}

Let $\varepsilon =-{\rm sign}\big(e_2^2-3\lambda_3\big)$.
We can state the following.

\begin{thmx} \label{Theorem1ss:integrability2} 
A general critical curve of type ${\mathcal B}'_1$ with modulus ${\bf c}$ is congruent to
\begin{equation}\label{standard}
\gamma\colon \ \R \ni s \longmapsto \big[{^t\!(}z_2+z_3,\varepsilon {\rm i}z_1,-\varepsilon {\rm i}(z_2-z_3))\big]\in \S.\end{equation}
In addition, $\gamma$ is in a standard configuration.
\end{thmx}

\begin{proof}
Let $\widetilde{\gamma}$ be a critical curve of type ${\mathcal B}'_1$ with modulus ${\bf c}$. Let ${\mathcal F}$ be
a
Wilczynski frame along~$\widetilde{\gamma}$. Suppose $\varepsilon =1$ (i.e., $\langle {\bf y}_3,{\bf y}_3\rangle<0$).
Let ${\bf u}_j$ be the maps defined by
 \begin{gather*}
{\bf u}_1=\frac{1}{\sqrt{3}\sqrt{2c_1+\lambda_1^2}\sqrt{\tau^2-3\lambda_1}}{\bf y}_1,\qquad
{\bf u}_2=\frac{1}{\sqrt{3}\sqrt{-\big(2c_1+\lambda_2^2\big)}\sqrt{\tau^2-3\lambda_2}}{\bf y}_2,\\
{\bf u}_3=\frac{1}{\sqrt{3}\sqrt{2c_1+\lambda_3^2}\sqrt{\tau^2-3\lambda_3}}{\bf y}_3.
\end{gather*}

Consider the map
${U}=({\bf u}_3,{\bf u}_2, {\bf u}_1)\colon\R\to {\rm GL}(3,\C)$.
From Theorem~\ref{thmss:integrability1}
and Lemma~\ref{Lemma1ss:integrability2}, we have
\begin{itemize}\itemsep=0pt
\item $\langle {\bf u}_1,{\bf u}_1 \rangle= \langle {\bf u}_2,{\bf u}_2 \rangle=-\langle {\bf u}_3,{\bf u}_3 \rangle
 =1$, and $\langle {\bf u}_i,{\bf u}_j\rangle = 0$, for $i\neq j$, that is, ${U}(s)$ is a~pseudo-unitary
 basis of $\C^{1,2}$, for every $s\in \R$;

\item ${U}^{-1}{L}{U}=D(\lambda_3,\lambda_2,\lambda_1)$.
\end{itemize}
Using again Theorem~\ref{thmss:integrability1}, we obtain
\begin{gather}\label{nf}
{\mathcal F} {U} D\big({\rm e}^{-{\rm i}\phi_3}, {\rm e}^{-{\rm i}\phi_2}, {\rm e}^{-{\rm i}\phi_1}\big)=
{\rm M}D\Big(\sqrt{3\big(2c_1+\lambda_3^2\big)}, \sqrt{-3\big(2c_1+\lambda_2^2\big)}, \sqrt{3\big(2c_1+\lambda_1^2\big)}\Big)^{-1},
 \end{gather}
where the matrix ${\rm M}\in {\rm GL}(3,\C)$ diagonalizes the momentum of $\widetilde{\gamma}$, i.e.,
${\rm M}^{-1}\widetilde{{\mathfrak M}}{\rm M}=
D(\lambda_3,\lambda_2,\lambda_1)$.
In particular, the column vectors of the right hand side of~\eqref{nf}, denoted by $\widetilde{{\rm B}}$,
constitutes a~pseudo-unitary basis. Let $\epsilon$ be the inverse of a cubic root of ${\widetilde{\rm B}}$.
Then, the column vectors of~$\epsilon {\widetilde{\rm B}}$ constitute a unimodular pseudo-unitary basis.
Therefore, there exists a unique ${A}\in {G}$, such that
$\epsilon {A}{\widetilde{\rm B}}={\rm B}$, where
\begin{equation}\label{standardpsh}
 {\rm B}=\begin{pmatrix}
 \frac{1}{\sqrt{2} }&-\frac{1}{\sqrt{2}} & 0\\
 0 & 0 &{\rm i}\\
 \frac{{\rm i}}{\sqrt{2}} &\frac{{\rm i}}{\sqrt{2}} & 0\\
 \end{pmatrix}.
 \end{equation}
Then
\begin{equation}\label{nf2}
{A}{\mathcal F}=\epsilon^{-1}{\rm B} D\big({\rm e}^{{\rm i}\phi_3}, {\rm e}^{{\rm i}\phi_2}, {\rm e}^{{\rm i}\phi_1}\big){U}^{-1}=
\epsilon^{-1}{\rm B}D\big({\rm e}^{{\rm i}\phi_3}, {\rm e}^{{\rm i}\phi_2}, {\rm e}^{{\rm i}\phi_1}\big)D(-1,1,1){^t\!{}\bar{U}}{\bf h}.
 \end{equation}
It is now a computational matter to check that the first column vector of the right hand side of~\eqref{nf2} is
$\epsilon^{-1}\,{^t\!(}z_2+z_3, {\rm i}z_1,- {\rm i}(z_2-z_3))$.
This implies $\widetilde{\gamma}={A}^{-1}\gamma$ (i.e., $ \widetilde{\gamma}$ and $\gamma$ are congruent to each other).

Taking into account that ${U}^{-1}{L}{U}=D(\lambda_3,\lambda_2,\lambda_1)$ and using~\eqref{nf}, the momentum
of $\widetilde{\gamma}$ is
$\widetilde{{\mathfrak M}}=\widetilde{{\rm B}}D(\lambda_3,\lambda_2,\lambda_1)\widetilde{{\rm B}}^{-1}$.
Therefore, the momentum of $\gamma$ is
\[
{\mathfrak M}={A}\widetilde{{\rm B}}D(\lambda_3,\lambda_2,\lambda_1)\widetilde{{\rm B}}^{-1}{A}^{-1}=
{\rm B}D(\lambda_3,\lambda_2,\lambda_1){\rm B}^{-1}={\mathfrak M}_{\lambda_1,\lambda_2,\lambda_3}.
\]
This proves that $\gamma$ is in standard configuration.

If $\varepsilon =-1$ (i.e., $\langle {\bf y}_3,{\bf y}_3\rangle>0$), considering ${U}=({\bf u}_2,{\bf u}_3,{\bf u}_1)$
and arguing as above, we get the same conclusion.
\end{proof}

\begin{Remark}
 Theorem~\ref{Theorem1ss:integrability2} implies that a standard configuration $\gamma$
 does not pass through the pole
 $[\mathbf{e}_3]$
 of the Heisenberg projection $\pi_{H}$. Thus $\check{\gamma}:=\pi_{H}\circ \gamma$ is a transversal curve of $\R^3$,
 which does not
intersect the $Oz$-axis.
 \end{Remark}

 \begin{Remark}
 Breaking the integrands into partial fractions, the integrals
\[
{f_j(\tau)} =
\sqrt{\frac{3}{2}}\int_{e_2}^{\tau}\frac{3c_1\lambda_j-\big(4c_1+\lambda_j^2\big)
+3\tau^3}{\tau\big(3\lambda_j-\tau^2\big)\sqrt{{\rm P}_{{\bf c}}(\tau)}}{\rm d}\tau,\qquad j=1,2,3,
\]
can be written as linear combinations of standard hyperelliptic integrals of the first and third kind.
Then $\phi_j$ is the odd quasi-periodic function with quasi-period $2\omega$ such that
{$\phi_j(s)=f_j[\tau(s)]$}.
In practice, we compute $\tau$ and $\phi_j$, $j=1,2,3$, by numerically solving the following system of ODE,
\begin{gather}
\tau''=\tau^2-9c_1\tau^{-2}\big(1-c_1\tau^{-1}\big),\nonumber\\
\phi_j'=\frac{3c_1\lambda_j-\big(4c_1+\lambda_j^2\big)\tau^2+3\tau^3}{\tau^2\big(3\lambda_j-\tau^2\big)},\qquad j =1,2,3,\label{nsol}
\end{gather}
with initial conditions
\begin{equation}\label{icon}
 \tau(0)=e_2,\qquad \tau'(0)=0,\qquad \phi_j(0)=0, \qquad j=1,2,3.\end{equation}
 \end{Remark}

\subsection{Closing conditions}

From Theorem~\ref{Theorem1ss:integrability2}, it follows that a critical curve of type ${\mathcal B}'_1$
is closed if and only if
\[
 {\mathfrak P}_j=\frac{1}{2\pi}\phi_j(2\omega)\in \Q,\qquad j=1,2,3.
\]
On the other hand,
\begin{equation}\label{period}
 \frac{1}{2\pi}\phi_j(2\omega)=\frac{1}{\pi}\int_{e_2}^{e_1} \frac{\sqrt{3}\big(3c_1\lambda_j-\big(4c_1+\lambda_j^2\big)\tau^2+3\tau^3\big)} {\sqrt{2}\tau\big(3\lambda_j-\tau^2\big)\sqrt{{\rm P}_{{\bf c}}(\tau)}}{\rm d}\tau.
 \end{equation}
Thus, $\gamma$ is closed if and only if the complete hyperelliptic integrals on the right hand side of~\eqref{period} are rational.
For a closed critical curve $\gamma$, we put ${\mathfrak P}_j= q_j=m_j/n_j$, where $n_j>0$ and $\gcd(m_j,n_j)=1$.
We call $q_j$, the \emph{quantum numbers} of $\gamma$.
By construction, ${\rm e}^{{\rm i}2\pi{\mathfrak P}_1}, {\rm e}^{{\rm i}2\pi {\mathfrak P}_2}$ and~${\rm e}^{{\rm i}2\pi {\mathfrak P}_3}$ are the eigenvalues of the \emph{monodromy} ${\mathtt M}_{\gamma}={\mathcal F}(2\omega){\mathcal F}(0)^{-1}$ of $\gamma$.
Since ${\rm det}({\mathtt M}_{\gamma})=1$, we have
\[
 \sum_{j=1}^{3} {\mathfrak P}_j\equiv 0,\quad \mod{\mathbb Z}.
\]
 Then, $\gamma$ is closed if and only if two among the integrals ${\mathfrak P}_j$, $j=1,2,3$, are rational.

\begin{Remark}
 The closing conditions can be rephrased as follows. Consider the even quasi-periodic functions $\phi_1$, $\phi_3$.
 Then, the critical curve is closed if and only if the jumps $\phi_j|_0^{2\omega}$, $j=1,3$, are rational.
 \end{Remark}

\begin{Example}\label{ex1p1}
We now consider an example, which will be taken up again in the last section. Choose
 ${\bf c}\approx (-0.8284243304411575,-8.349417691746162)\in {\mathcal B}_1^-$.
 The real roots of the quintic polynomial are
\[
 e_1\approx -0.931924<e_2\approx -0.678034<0<e_3\approx 2.79051
\]
 and the eigenvalues of the momentum are
\[
 \lambda_1\approx -2.40462<0<\lambda_2\approx 0.40614<\lambda_3\approx 1.99848.
\]
The half-period of the twist is computed by numerically evaluating the hyperelliptic integral~\eqref{halfpB1}.
 We evaluate
 $\tau$, $\phi_1$, $\phi_2$, $\phi_3$ by solving numerically
 the system~\eqref{nsol}, with initial conditions~\eqref{icon} on the interval $[-4\omega,4\omega]$.
 Figure~\ref{FIGANG} reproduces the graph of the quasi-periodic function $\phi_1$ on the interval
 $[-4\omega,4\omega]$ (the graph of the twist was depicted in Figure~\ref{FIG5}). The red point on the $Ox$-axis is
 $2\omega$ and the length of the arrows is the jump $\phi_1\vert_0^{2\omega}$.
 In this example,
\[
 \biggl|-\frac{1}{2\pi}\frac{2}{15}-\phi_1|_0^{2\omega}\biggr|=1.6151\cdot 10^{-8},\qquad
 \biggl|-\frac{1}{2\pi}\frac{10}{21}-\phi_3|_0^{2\omega}\biggr|=4.46887\cdot 10^{-8}.
\]
So, modulo negligible numerical errors, the corresponding critical curves are closed, with quantum number
$q_1=-2/15$ and $q_3=-10/21$.
In the last section, we will explain how we computed the modulus. A standard configuration of a curve with modulus
${\bf c}$ is represented in Figure~\ref{FIG12}.

\begin{figure}[t]\centering
\includegraphics[height=4.5cm,width=6.1cm]{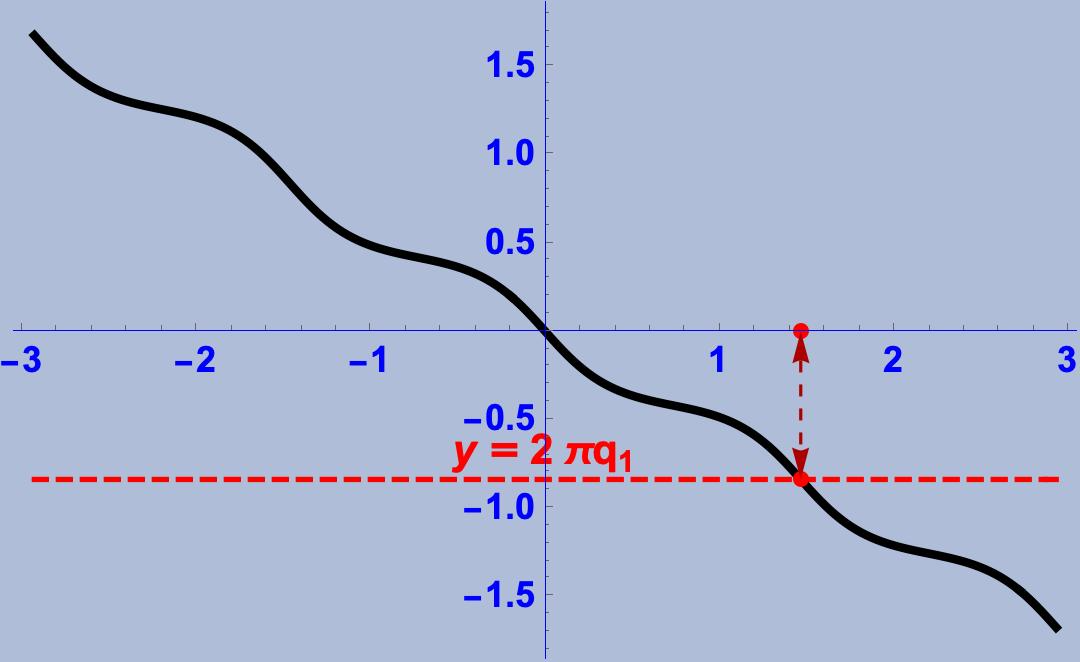}
\caption{The graph of $\phi_1$ for ${\bf c}\approx (-0.8284243304411575,-8.349417691746162)\in {\mathcal B}_1^-$.}\label{FIGANG}
\end{figure}
 \end{Example}

\subsection{Discrete global invariants of a closed critical curve}\label{ss:disinvcritical}

Consider a closed general critical curve $\gamma$ of type ${\mathcal B}_1'$,
with modulus ${\bf c}$
and quantum numbers $q_1=m_1/n_1$, $q_2=m_2/n_2$, $q_3=m_3/n_3$, $q_1+q_2+q_3\equiv 0 \mod \Z$.
The half-period $\omega$ of the twist is given by the complete hyperelliptic integral~\eqref{halfpB1}.
Let
${\mathtt M}_{\gamma}={\mathcal F}(\omega){\mathcal F}(0)^{-1}$ be the monodromy of $\gamma$.
The monodromy does not depend on the choice of the canonical lift.
It is a diagonalizable element of ${G}$ with eigenvalues ${\rm e}^{2\pi {\rm i} q_1}$, ${\rm e}^{2\pi {\rm i} q_2}$, and ${\rm e}^{2\pi {\rm i} q_3}$.
Thus, ${\mathtt M}_{\gamma}$ has finite order $n=\mathrm{lcm}(n_1,n_3)$.
The momentum ${\mathfrak M}_{\gamma}$ has three distinct real eigenvalues,
so its stabilizer is a maximal compact abelian subgroup ${\mathbb T}^2_{\gamma}\cong {\rm S}^1\times {\rm S}^1$ of ${G}$
\big(if $\gamma$ is a standard configuration, ${\mathbb T}^2_{\gamma}={\mathbb T}^2$\big).
Since $[{\mathtt M}_{\gamma},{\mathfrak M}_{\gamma}]=0$, ${\mathtt M}_{\gamma}\in {\mathbb T}^2_{\gamma}$.
Let $s_1$, $s_3$ be the integers defined by $n=s_1n_1=s_3n_3$. The \textit{CR spin} of $\gamma$ is $1/3$ if and only if
$n\equiv 0 \mod 3$ and $m_1s_1\equiv m_3s_3\not\equiv 0 \mod 3$.
The \textit{wave number} $\mathbf{n}_{\gamma}$ of $\gamma$ is $n$ if the spin is 1 and $n/3$ if the spin is $1/3$.
Let $|[\gamma]|$ denote the trajectory of $\gamma$.
The stabilizer $\hat{{G}}_{\gamma}=\{[A]\in [{G}]\mid [{A}]\cdot |[\gamma]| = |[\gamma]|\}$ is spanned by
$[{\mathfrak M}_{\gamma}]$ and is a cyclic group of order ${\bf n}_{\gamma}$.
Geometrically, $\hat{{G}}_{\gamma}$ is the \emph{symmetry group} of the critical curve $\gamma$.
The \textit{CR turning number} ${\it w}_{\gamma}$ is the degree of the map
$\R/2n\omega \Z$ $\ni$ $s\mapsto F_1-{\rm i} F_3$ $\in$ $\dot{\C} := \mathbb{C}\setminus \{0\}$,
where the $F_j$'s
are the components of a Wilczynski
frame
along $\gamma$.
Without loss of generality, we may suppose that $\gamma$ is in a standard configuration. From
\eqref{standard}, it follows that ${\it w}_{\gamma}$ is the degree of
$ \R/2n\omega \Z\ni s\mapsto z_3\in \dot{\C}$, if $\varepsilon_\gamma = 1$,
and is the degree of
$ \R/2n\omega \Z\ni s\mapsto z_2\in \dot{\C}$, if $\varepsilon_\gamma = -1$.
Therefore,
\[
 {\it w}_{\gamma}=
 \begin{cases} s_3m_3 & {\rm if}\ \varepsilon_{\gamma}=1,\\
 s_2m_2 & {\rm if}\ \varepsilon_{\gamma}=-1.
 \end{cases}
 \]
A closed critical curve $\gamma$ has an additional discrete CR invariant, denoted by ${\rm tr}_*(\gamma)$,
the trace of $\gamma$ with respect to
the spacelike $\lambda_1$-eigenspace of the momentum.
To clarify the geometrical meaning of the trace, it is convenient to consider a standard configuration.
In this case, ${\mathbb L}_1$ is spanned by $\mathbf{e}_2\in \C^{1,2}$
and the corresponding chain is the intersection of $\S$ with the projective line $z_2=0$. The Heisenberg projection
of this chain is the upward oriented $Oz$-axis. Thus, $\mathrm{tr}_*(\gamma)$ is the linking number ${\rm Lk}\big(\check{\gamma},Oz^{\uparrow}\big)$ of the Heisenberg projection of $\gamma$ with the upward oriented $Oz$-axis.

\begin{Proposition}
Let $\gamma$ be as above. Then
\[{\rm tr}_*(\gamma)=\begin{cases} (q_1-q_3){\bf n}_{\gamma} & {\rm if}\ \varepsilon_{\gamma}=1,\\
(q_1-q_2){\bf n}_{\gamma} & {\rm if}\ \varepsilon_{\gamma}=-1.
\end{cases}
\]
\end{Proposition}

\begin{proof}
Without loss of generality, we may assume that $\gamma$ is in standard configuration.
The Heisenberg projection of $\gamma$ is
\[
 \check{\gamma}={^t\!}\bigg(\operatorname{Re}\bigg(\frac{\varepsilon iz_1,}{z_2+z_3}\bigg),
 {\operatorname{Im}\bigg(\frac{\varepsilon {\rm i}z_1,}{z_2+z_3}\bigg)}, \operatorname{Re}\bigg(\frac{-\varepsilon {\rm i}(z_2-z_3)}{z_2+z_3}\bigg) \bigg).
\]
Since $\check{\gamma}$ does not intersect the $Oz$-axis, the linking number
${\rm Lk}\big(\check{\gamma},Oz^{\uparrow}\big)$ is the degree of
\[
 \R/2{\bf n}_{\gamma}\omega \Z \ni s \longmapsto \frac{z_1}{z_2+z_3}\in \dot{\C}.
\]
From~\eqref{functions}, it follows that this degree is the degree of
\[
 f\colon \ \R/2{\bf n}_{\gamma}\omega \Z\ni s\longmapsto
 \frac{\rho_1\sqrt{3(\lambda_2 + \lambda_3)+\tau^2(s)} {\rm e}^{{\rm i}\phi_1}}{\rho_2\sqrt{3\lambda_2-\tau^2(s)} {\rm e}^{{\rm i}\phi_2}+ \rho_3\sqrt{3\lambda_3-\tau^2(s)} {\rm e}^{{\rm i} \phi_3}}.
\]
Suppose that $\gamma$ is
negatively
polarized. Then, $\tau^2-3\lambda_3<\tau^2-3\lambda_2<0$ and $0<\tau^2<3\lambda_2$. Therefore,
\[
 0<\frac{\rho_2\sqrt{3\lambda_2-\tau^2}}{\rho_3\sqrt{3\lambda_3-\tau^2}}=
\sqrt{\frac{\big(3\lambda_2-\tau^2\big)(\lambda_2+2\lambda_3)}{\big(3\lambda_3-\tau^2\big)(2\lambda_2+\lambda_3)}}
\le \sqrt{\frac{\lambda_2(\lambda_2+2\lambda_3)}{\lambda_3(2\lambda_2+\lambda_3)}}<1.
\]
Thus
\[
 f= \frac{\rho_1\sqrt{3(\lambda_2 + \lambda_3)+\tau^2}}{\rho_3\sqrt{3\lambda_3-\tau^2}}
 \frac{{\rm e}^{{\rm i}(\phi_1-\phi_3)}}{1+h {\rm e}^{{\rm i}(\phi_2-\phi_3)}},
\]
where
\[
 h = \frac{\rho_2\sqrt{3\lambda_2-\tau^2}}{\rho_3\sqrt{3\lambda_3-\tau^2}}.
\]
Since $0<h<1$, the image of $1+h{\rm e}^{{\rm i}(\phi_2-\phi_3)}$
is a curve contained in a disk of radius $<1$ centered at $(1,0)$. Hence $1+h{\rm e}^{{\rm i}(\phi_2-\phi_3)}$ is null-homotopic in $\dot{\C}$. This implies
\[
 \deg (f)=\frac{1}{2\pi}(\phi_1-\phi_3)\Big|_0^{2{\bf n}_{\gamma}\omega} = {\bf n}_{\gamma}(q_1-q_3).
\]

Suppose that $\gamma$ is
positively
polarized. Then, $\tau^2-3\lambda_2>\tau^2-3\lambda_3>0$. In particular, $\tau^2>3\lambda_3>0$ and
\[
0<\frac{\rho_3\sqrt{\tau^2-3\lambda_3}}{\rho_2\sqrt{\tau^2-3\lambda_2}}
=\sqrt{\frac{(\tau^2-3\lambda_3)(2\lambda_2+\lambda_3)}{(\tau^2-3\lambda_2)(\lambda_2+2\lambda_3)}}
< \sqrt{
\frac{2\lambda_2+\lambda_3}
{\lambda_2+2\lambda_3}}<1.
\]
Then
\[
 f= -{\rm i}\frac{\rho_1\sqrt{3(\lambda_2 + \lambda_3)+\tau^2}}{\rho_2\sqrt{\tau^2-3\lambda_2}}\frac{{\rm e}^{{\rm i}(\phi_1-\phi_2)}}{1+\tilde{h} {\rm e}^{{\rm i}(\phi_3-\phi_2)}},
\]
where
\[
 \tilde{h} =\frac{\rho_3\sqrt{\tau^2-3\lambda_3}}{\rho_2\sqrt{\tau^2-3\lambda_2}}.
\]
Since $0<\tilde{h}<1$, the image of $1+\tilde{h}{\rm e}^{{\rm i}(\phi_3-\phi_2)}$
is a curve contained in a disk of radius $<1$ centered at $(1,0)$.
Hence $1+\tilde{h}{\rm e}^{{\rm i}(\phi_3-\phi_2)}$
is null-homotopic in $\dot{\C}$.
This implies
\[
 \deg (f)=\frac{1}{2\pi} (\phi_1-\phi_2)\Big|_0^{2{\bf n}_{\gamma}\omega} = {\bf n}_{\gamma}(q_1-q_2).\tag*{\qed}
\]\renewcommand{\qed}{}
\end{proof}

Summarizing: \emph{the quantum numbers of a closed critical curve are determined by the wave number,
the CR spin, the CR turning number, and the trace}.

\section[Experimental evidence of the existence of infinite countably many closed
 critical curves of type B'\_1 and examples]{Experimental evidence of the existence of infinite countably\\ many closed
 critical curves of type $\boldsymbol{{\mathcal B}'_1}$ and examples}\label{s:heuristic}

 This section is of an experimental nature. We use numerical tools, implemented in the software \textsc{Mathematica}~13.3,
 to support the claim that there exist
 countably many closed critical curves of type ${\mathcal B}_1'$, with moduli belonging to the connected component
 ${\mathcal B}_1^-$ of ${\mathcal B}_1$ (cf.\ Remark~\ref{r:connectedcomponents}).
 The same reasoning applies, as well, if the modulus belongs to the other connected component~${\mathcal B}_1^+$ of~${\mathcal B}_1$.
 We parametrize ${\mathcal B}_1^-$ by the map $\psi_-\colon K_-\to {\mathcal B}_1^-$, defined in~\eqref{parconc},
 where $K_{-}$ is the rectangle $\hat{{\rm J}}_{\xi}^{-}\times (0,1)$, $\hat{{\rm J}}_{\xi}^{-}=(\pi/2, 2.3008)$.
 We take ${\bf p}=(p_1,p_2)\in K_-$ as the fundamental parameters.
 The modulus $\mathbf{c}=(c_1,c_2)$, the roots $e_1<e_2<0<e_3$ of the quintic polynomial, and the eigenvalues
 $\lambda_1=-(\lambda_2+\lambda_3)<0<\lambda_2<\lambda_3$ of the momentum are explicit functions of the
 parameters $(p_1,p_2)$. Let $K_-^*$ be the open set of the {general} parameters, that is,
\[
 K_-^* = \left\{{\bf p}\in K_- \mid \Delta_1( \psi_-({\bf p}))\Delta_2( \psi_-({\bf p}))\neq 0\right\}.
\]
The complete {hyperelliptic} integrals ${\mathfrak P}_j$ can be evaluated numerically as functions
of ${\bf p}\in K_-^*$.
Consider the real analytic map
${\mathfrak P}=({\mathfrak P}_1, {\mathfrak P}_3)\colon K_-^*\to \R^2$.\footnote{Actually, ${\mathfrak P}$ is real-analytic on all $K_-$. Instead, ${\mathfrak P}_2$ has a jump discontinuity at the exceptional locus.}
Choose ${\bf p}_*=(2,1/2)\in K_-^*$ and plot the graphs of the functions
 $f_{11}(p_1)={\mathfrak P}_1(p_1,1/2)$, $f_{12}(p_2)={\mathfrak P}_1(2,p_2)$,
 $f_{31}(p_1)={\mathfrak P}_3(p_1,1/2)$,
 and $f_{32}(p_2)=
{\mathfrak P}_3(2,p_2) $ (see Figures~\ref{FIG8} and~\ref{FIG9}).

\begin{figure}[t]\centering
\includegraphics[height=6cm,width=6cm]{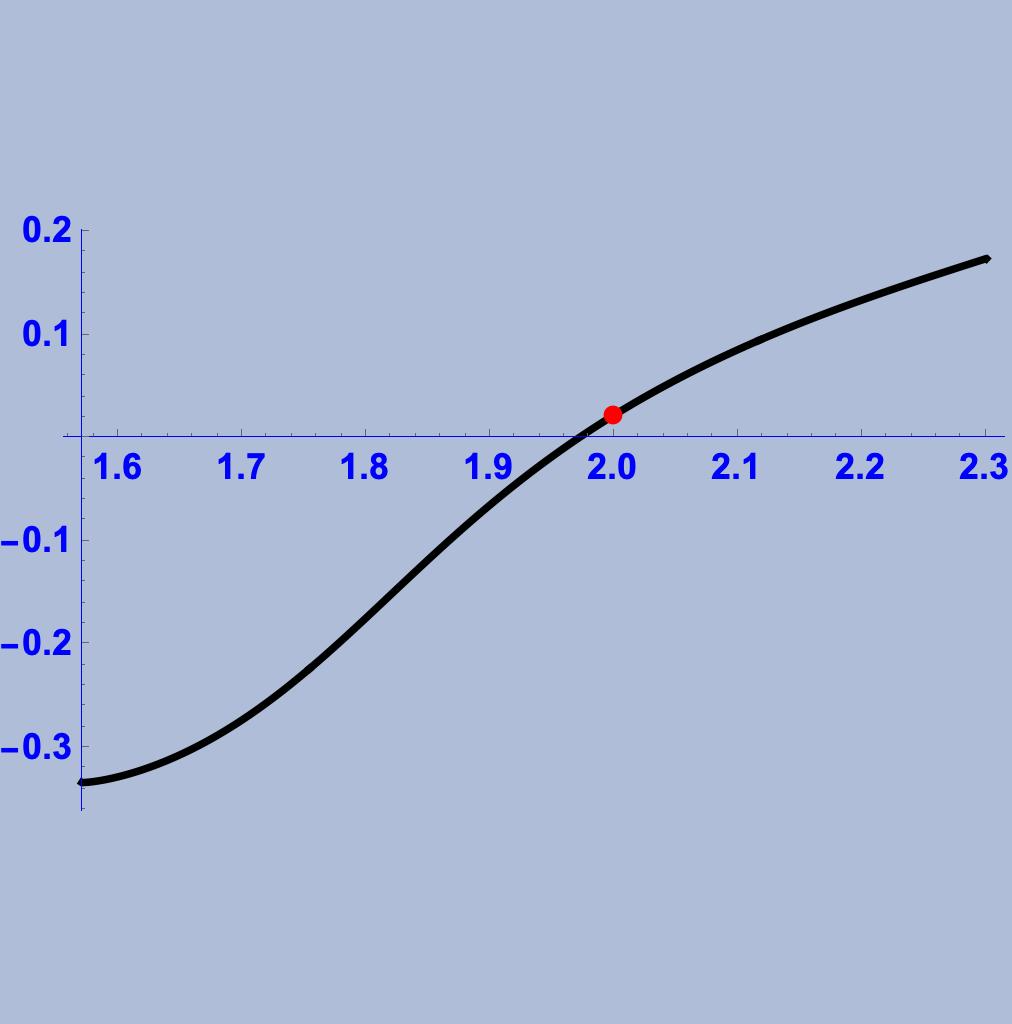}\qquad
\includegraphics[height=6cm,width=6cm]{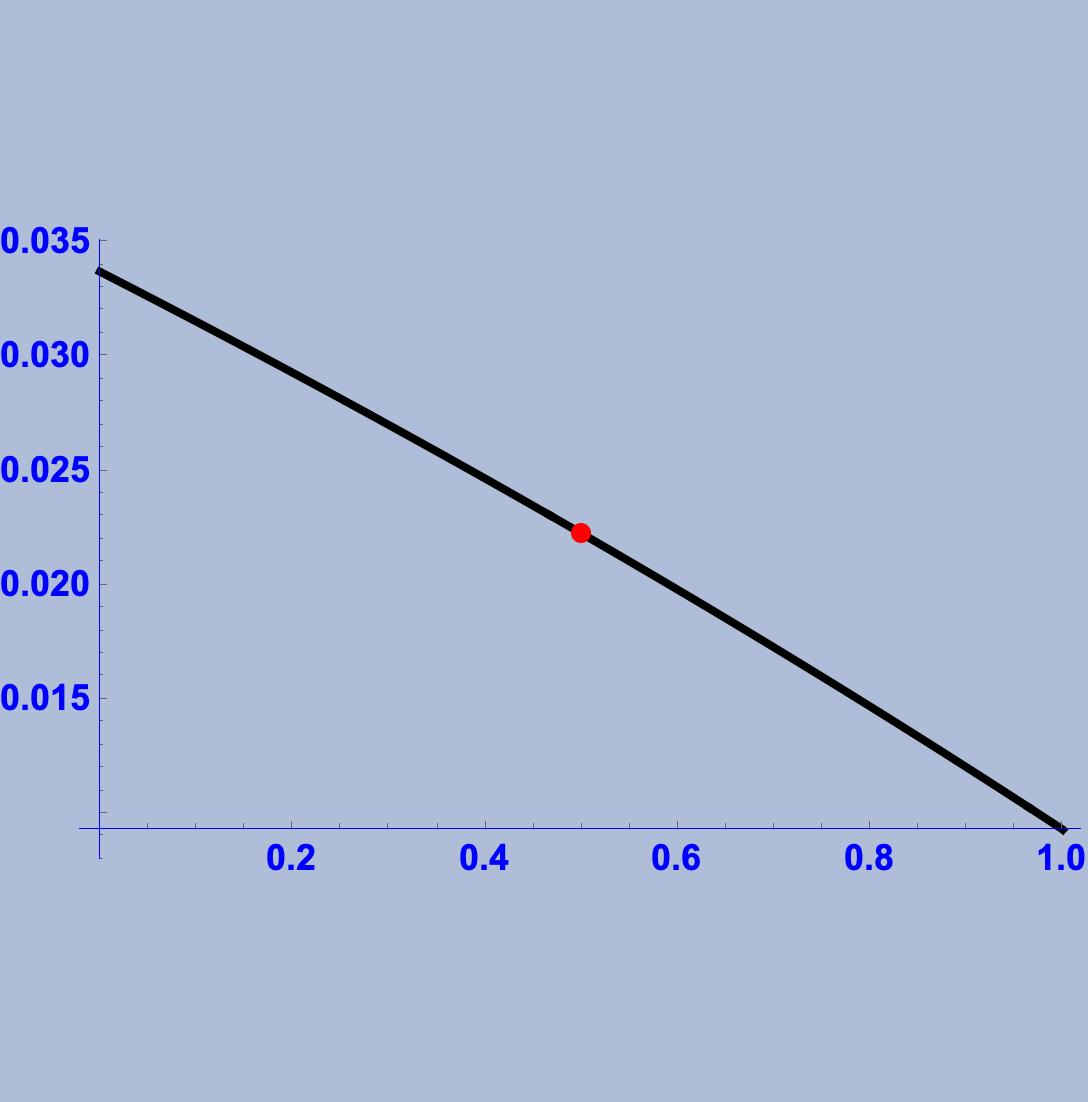}
\caption{On the left: the graph of the function $f_{11}$. On the right: the graph of the function $f_{12}$.}\label{FIG8}
\end{figure}

\begin{figure}[t]\centering
\includegraphics[height=6cm,width=6cm]{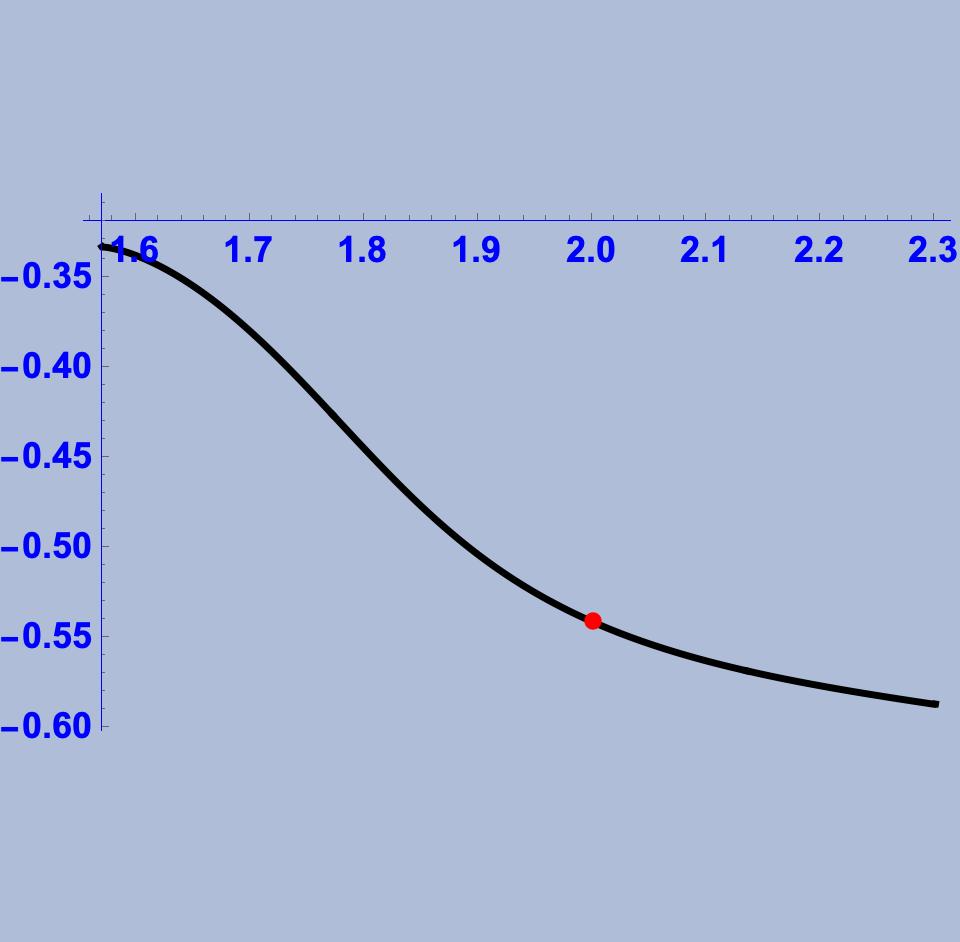}\qquad
\includegraphics[height=6cm,width=6cm]{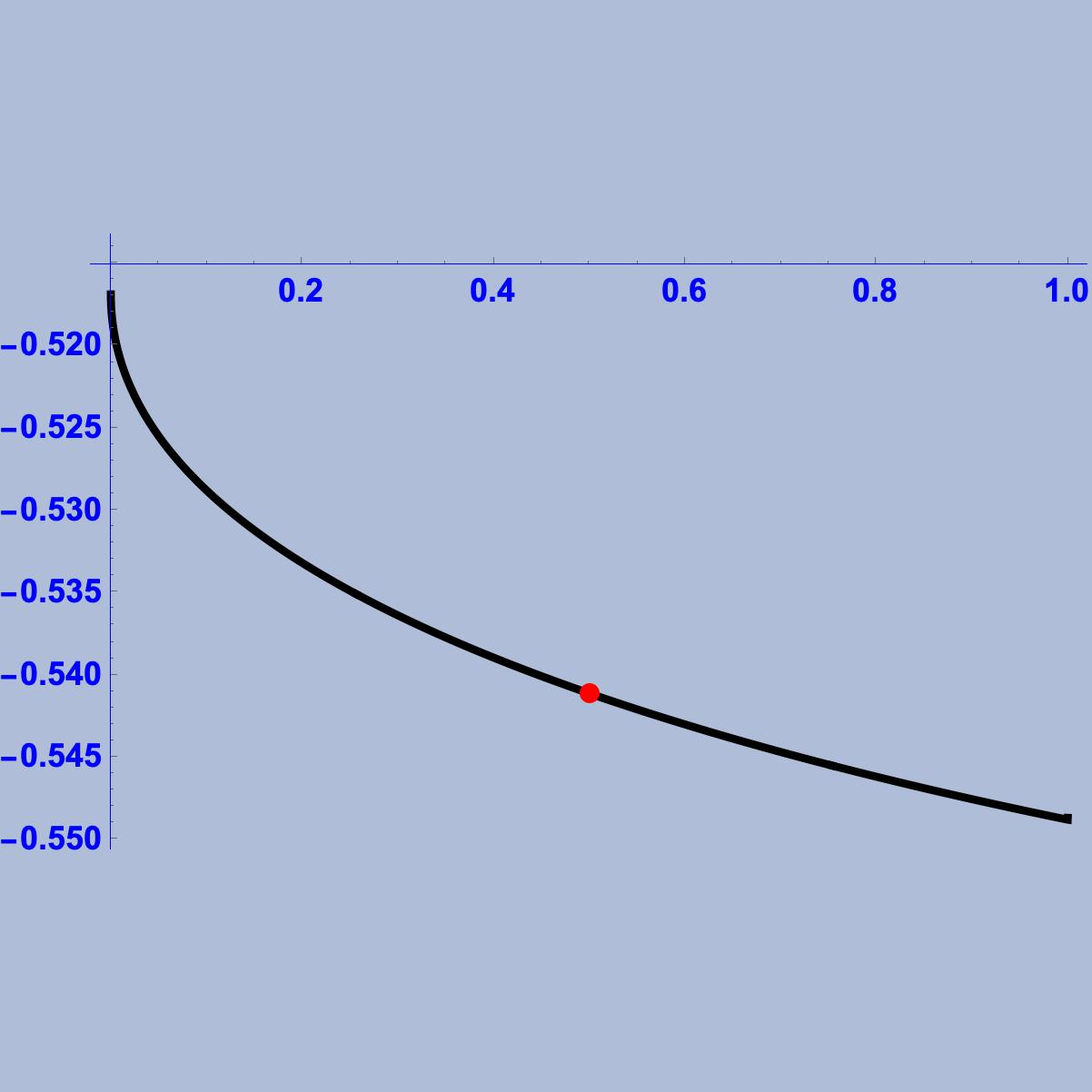}
\caption{On the left: the graph of the function $f_{31}$. On the right: the graph of the function $f_{32}$.}\label{FIG9}
\end{figure}

The function $f_{11}$ is strictly increasing, while the other three functions are strictly decreasing.
This implies that ${\mathfrak P}$ has maximal rank at ${\bf p}_*$.
Thus ${\mathcal P}_-={\mathfrak{P}}(K_-)$ is a set with non empty interior.
In particular ${\mathcal P}_-^r:={\mathcal P}_-\cap {\mathbb Q}$ is an infinite countable set
and, for every ${\bf q}=(q_1,q_2)\in {\mathcal P}_-^r$, there exists a closed critical curve of
type $\mathcal{B'}_{1}^{-}$ with quantum numbers $q_1$ and $q_2$.
 Figure~\ref{FIG10} reproduces the plot of the map ${\mathfrak P}$, an open convex set.
 The mesh supports a stronger conclusion: the map ${\mathfrak P}$ is 1-1. Therefore,
 one can assume
 that, for every rational point $(q_1,q_3)\in {\mathcal P}_-$, there exists a
 unique congruence class of closed critical curves with quantum numbers $q_1$ and $q_3$.
 The construction of a standard configuration of a critical curve associated to a rational point
 ${\bf q}\in {\mathcal P}_-$ can be done in three steps.

\begin{figure}[t]\centering
\includegraphics[height=6cm,width=9cm]{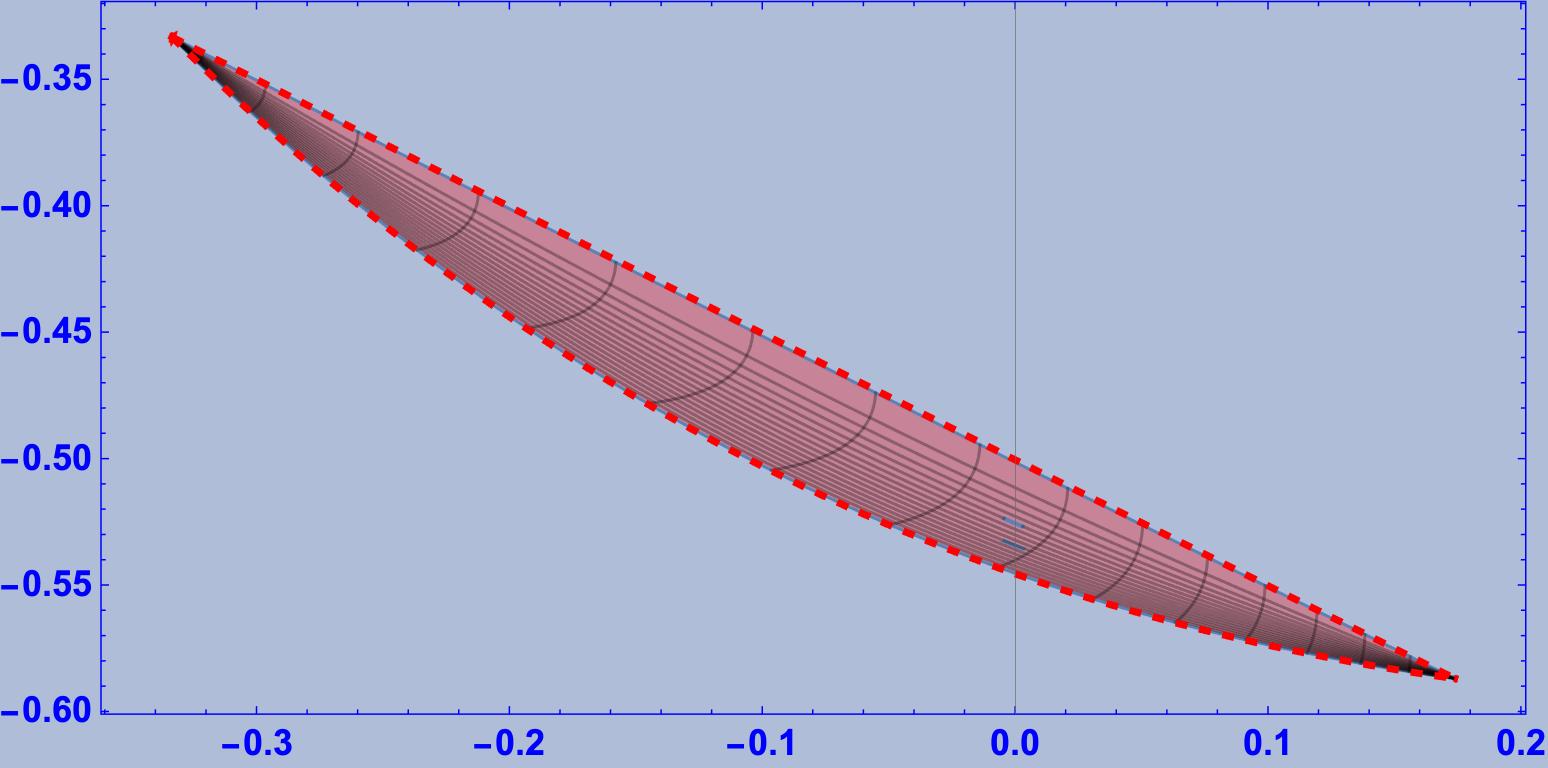}
\caption{The plot of the map ${\mathfrak P}$.}\label{FIG10}
\end{figure}

{\bf Step 1.} Choose a rational point ${\bf q}=(q_1,q_3)=(m_1/n_1,m_3/n_3)\in {\mathcal P}_-$.
To find the parameter ${\bf p}\in K_-$, such that
${\mathfrak P}({\bf p})={\bf q}$, we may proceed as follows: plot the level curves ${\mathtt X}_{q_1}={\mathfrak P}_1^{-1}(q_1)$
and ${\mathtt Y}_{q_3}={\mathfrak P}_3^{-1}(q_3)$ and choose a small rectangle ${\rm R}\subset K_-$
containing ${\mathtt X}_{q_1}\cap {\mathtt Y}_{q_2}$ (see Figure~\ref{FIG11}). Then we minimize numerically the function
\[
 \delta_{{\bf q}}\colon \ {\rm R}\ni {\bf p}\longmapsto \sqrt{({\mathfrak P}_1({\bf p})-q_1)^2+({\mathfrak P}_3({\bf p})-q_3)^2}.
\]
 We use the stochastic minimization method named ``differential evolution''~\cite{SP} implemented in \textsc{Mathematica}.

\begin{figure}[t]\centering
\includegraphics[height=6cm,width=6cm]{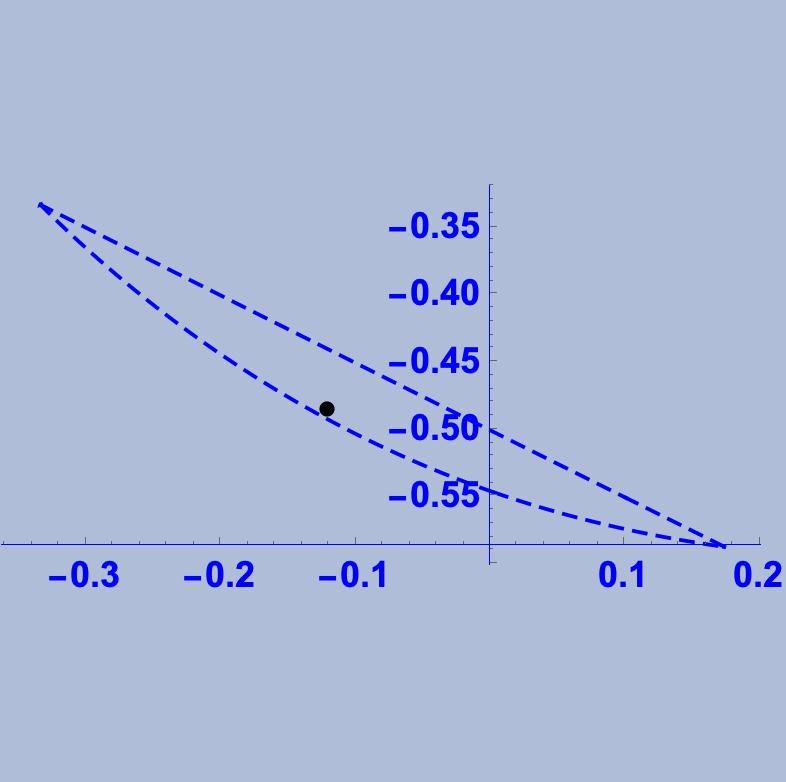}\qquad
\includegraphics[height=6cm,width=6cm]{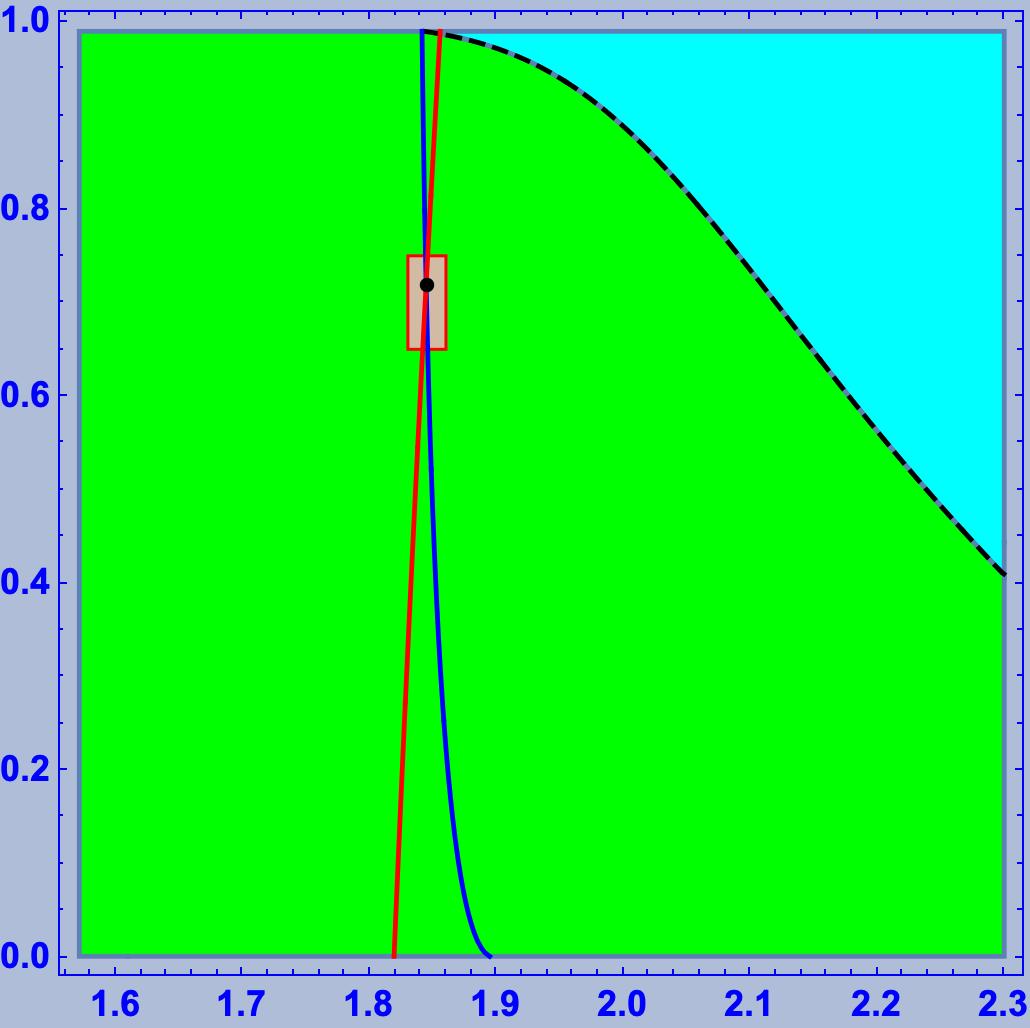}
\caption{On the left: the point ${\bf q}=(-2/15,-10/21)\in {\mathcal P}_-$. On the right: the level curves
${\mathtt X}_{-2/15}$ and ${\mathtt Y}_{-10/21}$. The dotted curve is the exceptional locus.
The green and the cyan domains are the two connected components of $K_-^*$. The brow rectangle is the one
chosen for the numerical minimization of the function $\delta_{{\bf q}}$.}\label{FIG11}
\end{figure}

\begin{Example}\label{eex1}
Let us revisit Example~\ref{ex1p1}.
Choose ${\bf q}=(-2/15,-10/21)\in {\mathcal P}_-$. The plot of the level curves ${\mathtt X}_{q_1}$ and ${\mathtt Y}_{q_3}$
is depicted in Figure~\ref{FIG11}. Minimizing $\delta_{{\bf q}}$ on the rectangle ${\rm R}=[1.83,1.86]\times [0.65,0.75]$
(depicted on the right picture in Figure~\ref{FIG11}) we obtain ${\bf p} = (1.84438,0.719473)$
and $\delta_{{\bf q}}({\bf p})=3.26867\cdot 10^{-9}$.
So, up to negligible numerical errors, we may assume
${\bf p}={\mathfrak P}^{-1}({\bf q})$. Computing $\psi_-({\bf p})$, we find the modulus ${\bf c}=(c_1,c_2)$
of the curve, where $c_1=-0.828424$ and $c_2= -8.349418$. With the modulus at hand, we compute
 the lowest real roots of the quintic polynomial, $e_1\approx -0.931924<e_2\approx -0.678034$, and the
 roots of the momentum, namely
 $\lambda_1\approx -2.40462<0<\lambda_2\approx 0.40614<\lambda_3\approx 1.99848$.
 \end{Example}

{\bf Step 2.}
We evaluate numerically the integral~\eqref{halfpB1} and we get the half-period $\omega$ of the twist of the critical curve.
In our example $\omega \approx 0.732307$. The next step is to evaluate the twist~$\tau$.
This can be done by solving numerically the Cauchy problem~\eqref{twistnB1'} on the interval $[0,2n\omega]$,
$n=\mathrm{lcm}(n_1,n_2)$.
The bending is given by $\kappa=c_1/\tau^2$. Next, we solve the Frenet type linear system~\eqref{MCE-W-frame},
with initial condition ${\mathcal F}(0)=I_3$. Then, $\widetilde{\gamma}\colon [0,2n\omega]\ni s\longmapsto [F_1(s)]\in \S$ is a critical
curve with quantum numbers $q_1$ and $q_3$ and ${\mathcal F}$ is a
Wilczynski frame
field along
{$\widetilde{\gamma}$}. However, {$\widetilde{\gamma}$}~is not in a standard configuration.

{\bf Step 3.}
The last step consists in building the standard configuration. The momentum ${\mathfrak M}$
of~$\widetilde{\gamma}$ is $L(0)$, where $L$ is as in~\eqref{obser-L}. Taking into account that $\tau(0)=e_2$,
$\tau'(0)=0$, and that $\kappa(0)=c_1/e_2^2$,
we get
\[
{\mathfrak M}=\begin{pmatrix}
 0 & 3\big(1 - \frac{c_1}{e_2}\big) & 2 {\rm i} e_2\\
 e_2 & 0 & 3{\rm i}\big(1 - \frac{c_1}{e_2}\big)\\
 3{\rm i} & -{\rm i} e_2& 0\\
 \end{pmatrix}.
\]
The eigenspace of the highest eigenvalue is timelike (i.e., these critical curves are
{negatively}
polarized).
We compute the eigenvectors and we build a unimodular pseudo-unitary basis ${A}=({A}_1,{A}_2, {A}_3)$,
 such that ${A}_1$ is an eigenvector of $\lambda_3$, ${A}_2$ is an eigenvector of $\lambda_2$, and ${A}_3$
 is an eigenvector of $\lambda_1$.
Let ${\rm B}$ be as in~\eqref{standardpsh}. Consider ${\rm M}={\rm B}{A}^{-1}\in {G}$.
Then, {$\gamma={\rm M}\widetilde{\gamma}$} is a standard configuration of a critical curve
with quantum numbers $q_1$ and $q_3$.

\begin{Remark}
The curve $\gamma$ does not pass through the pole of the Heisenberg projection $\pi_{H}$. So,
$\widehat{\gamma}=\pi_{H}\circ \gamma$ is a closed transversal curve of $\R^3$ which does not intersect the $Oz$-axis
and ${\rm tr}_*(\gamma)=
{\rm Lk}(\widehat{\gamma},Oz)$.
\end{Remark}

\begin{figure}[t]\centering
\includegraphics[height=6cm,width=6cm]{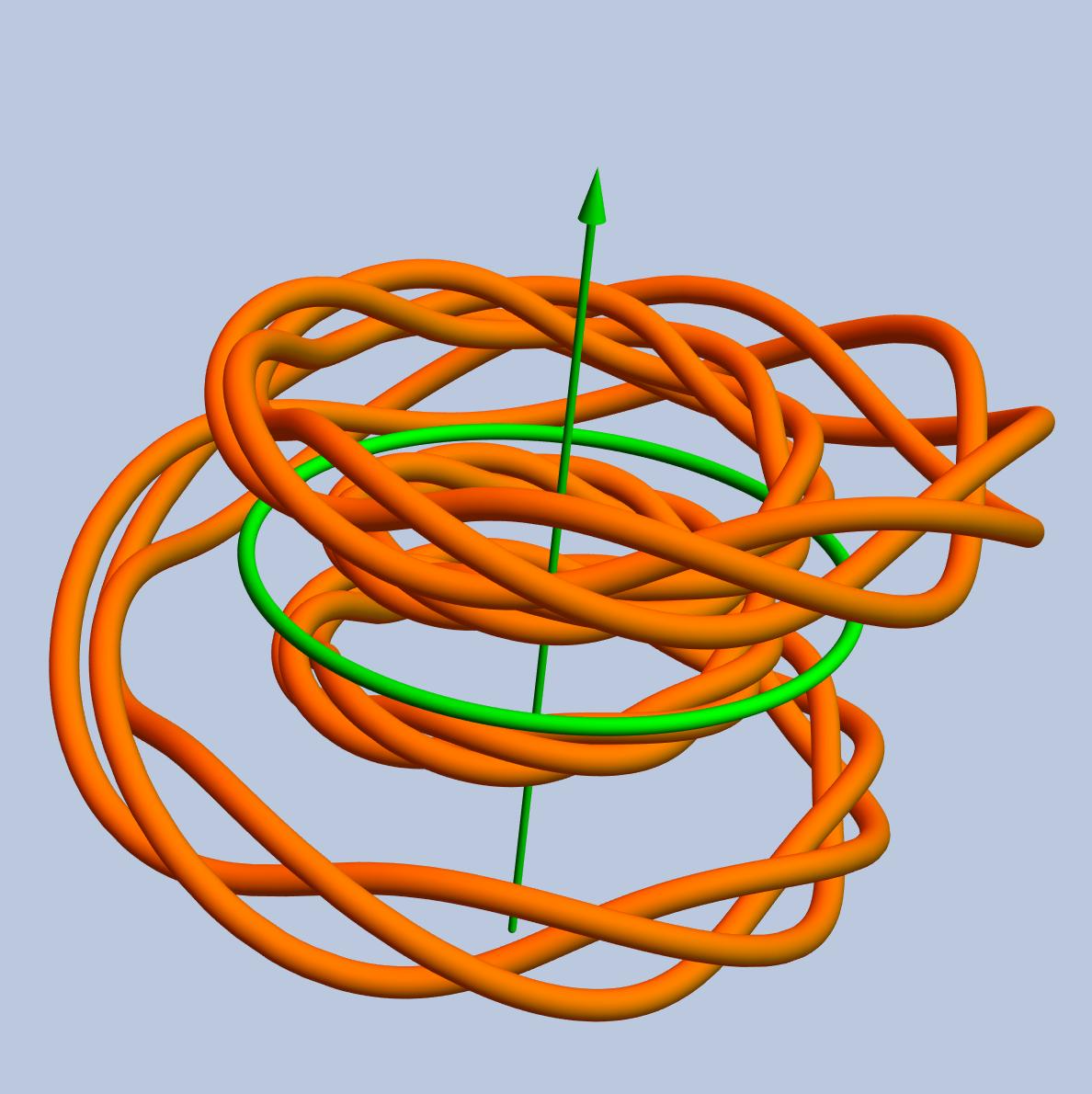}\qquad
\includegraphics[height=6cm,width=6cm]{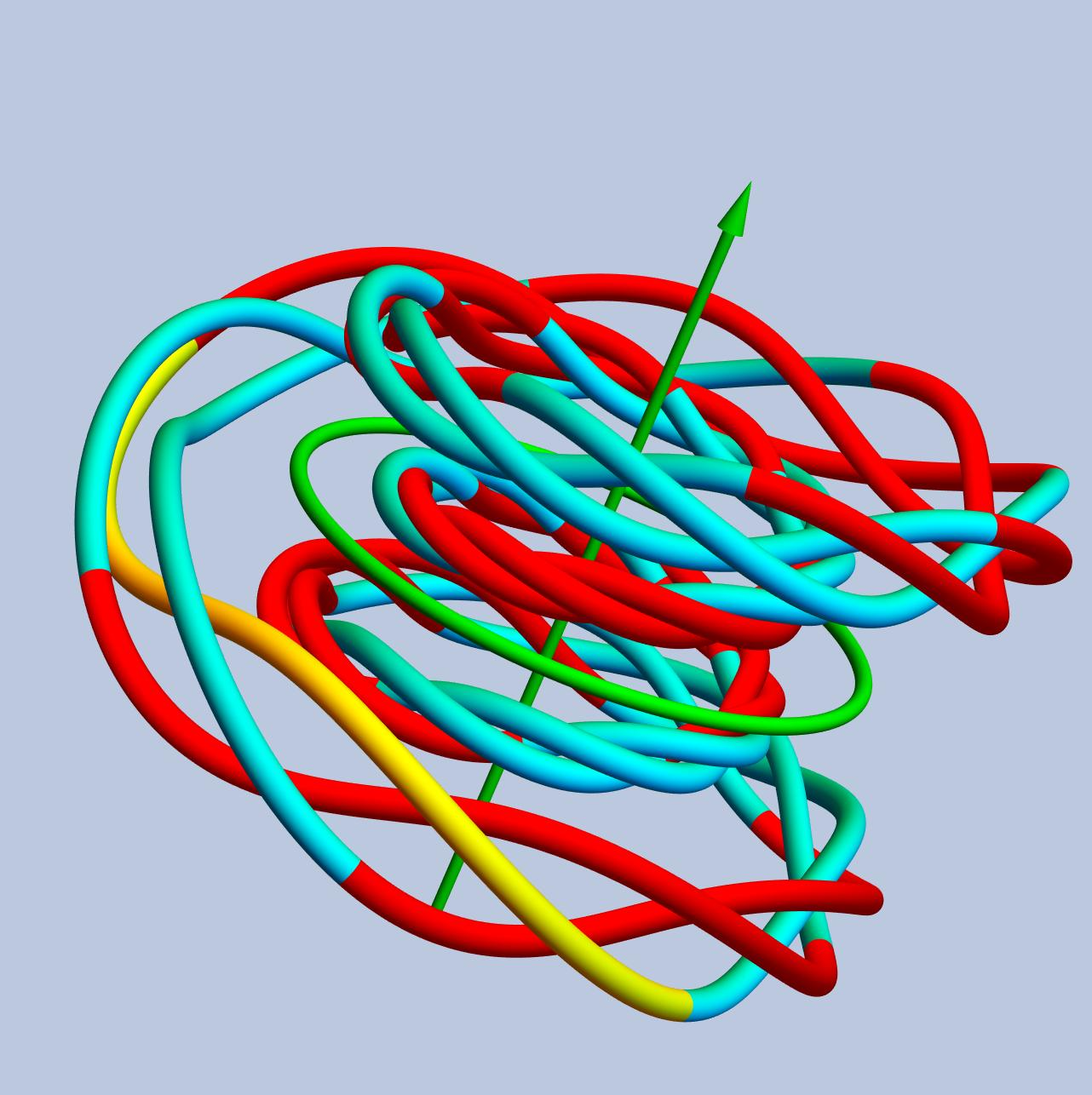}
\caption{The Heisenberg projection of a standard configuration of a critical curve with quantum numbers $q_1=-2/15$ and $q_2=-10/21$. The figure on the left reproduces the fundamental arc $\gamma([0,2\omega))$ (coloured in yellow). The curve can be constructed by acting with the monodromy on the fundamental~arc.\looseness=-1}\label{FIG12}
\end{figure}

\begin{figure}[t!]\centering
\includegraphics[height=6cm,width=6cm]{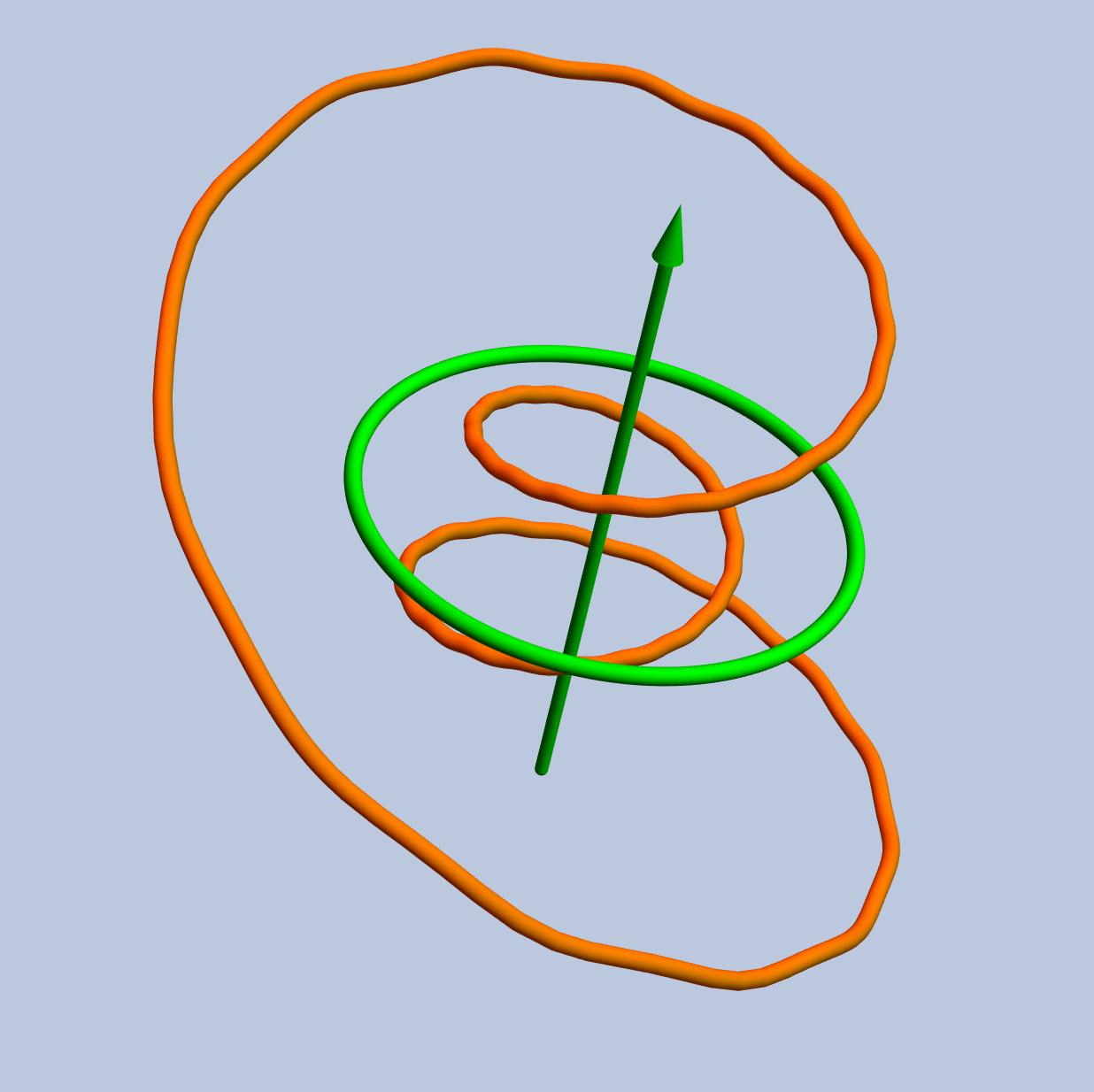}\qquad
\includegraphics[height=6cm,width=6cm]{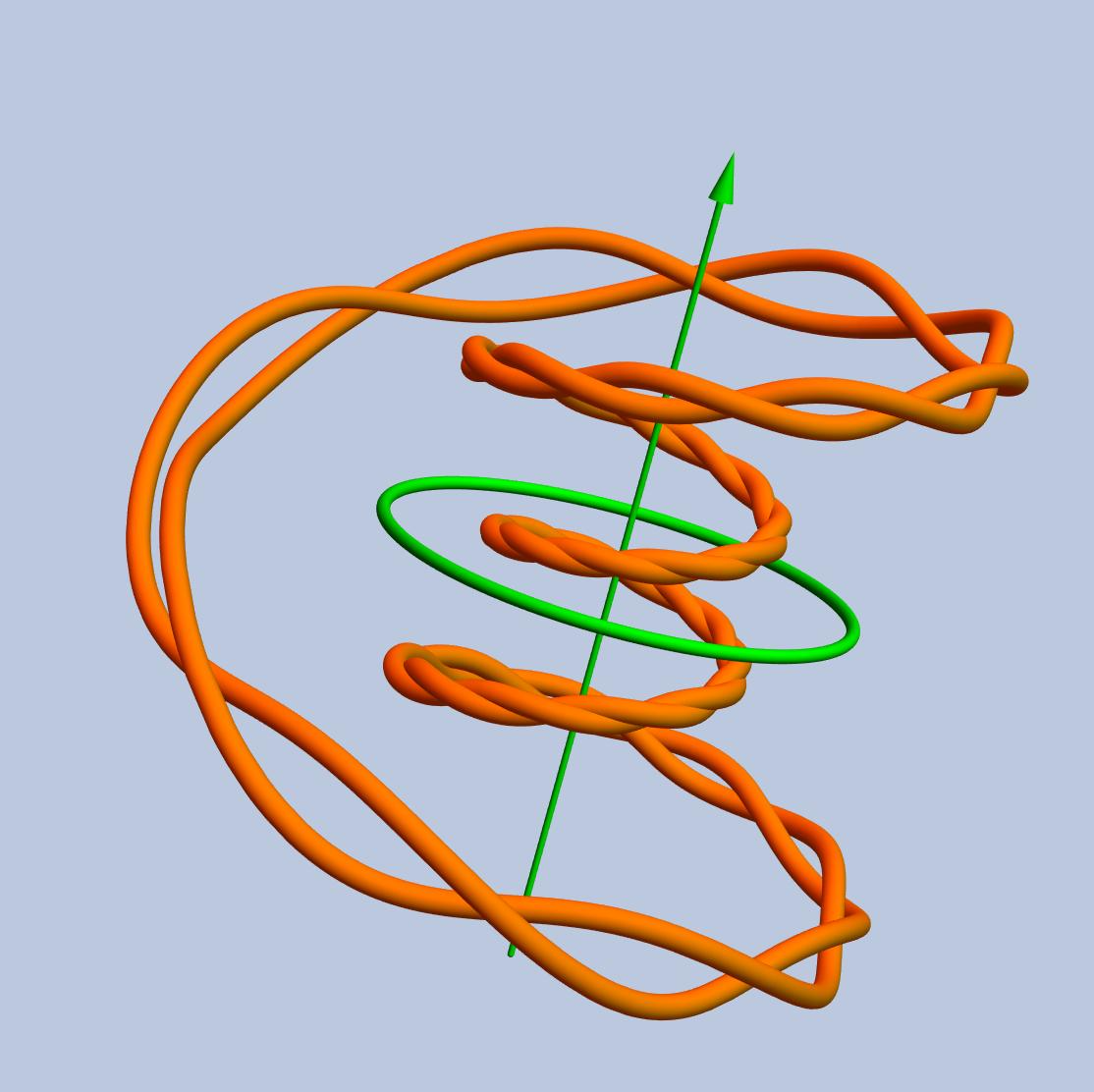}
\caption{On the left: the Heisenberg projection of a {standard} configuration of a critical
curve of type ${\mathcal B'}_{1}^{-}$ with quantum numbers $q_1=-3/10$, $q_2=-9/25$.
On the right: the Heisenberg projection of a~{standard} configuration of a
critical curve of type ${\mathcal B'}_{1}^{-}$ with quantum numbers $q_1=-1/5$, $q_2=-3/7$.}\label{FIG13}
\end{figure}

\begin{Example}
Applying Steps 2 and~3 to Example~\ref{eex1} and computing the Heisenberg projection,
we obtain the transversal curve depicted in Figure~\ref{FIG12}, a non-trivial transversal knot.
The quantum numbers are $q_1=-2/15$ and $q_2=-10/21$. Recalling what has been said about the {discrete}
invariants of a critical curve (cf.\ Section~\ref{ss:disinvcritical}), the spin is $1/3$, the wave number is $n=35$, the
CR turning number is $-50$, and the trace is $12$.
 \end{Example}

\begin{Example}
Figures~\ref{FIG13} and~\ref{FIG14} reproduce the Heisenberg projections of the
{standard} configurations of critical curves
of type ${\mathcal B'}_{1}^{-}$ with quantum numbers $(-3/10$, $-9/25)$, $(-1/5,-3/7)$, $(5/49, -4/7)$, and $(-7/36,-23/54)$,
 respectively. All of them have spin~1. The first is a trivial torus knot with wave number $n=50$, ${\rm tr}_*=3$, and
 CR turning number ${\it w}=-18$; the second example is a nontrivial transversal knot with $n=35$, ${\rm tr}_*=8$, and
 ${\it w}=-15$. The third example is a ``tangled" transversal curve with $n=49$, ${\rm tr}_*=33$, and ${\it w}=-28$.
 The last example is a~nontrivial transversal torus knot {with wave} number $108$, ${\rm tr}_*=25$, and ${\it w}=-46$.

\begin{figure}[t]\centering
\includegraphics[height=6cm,width=6cm]{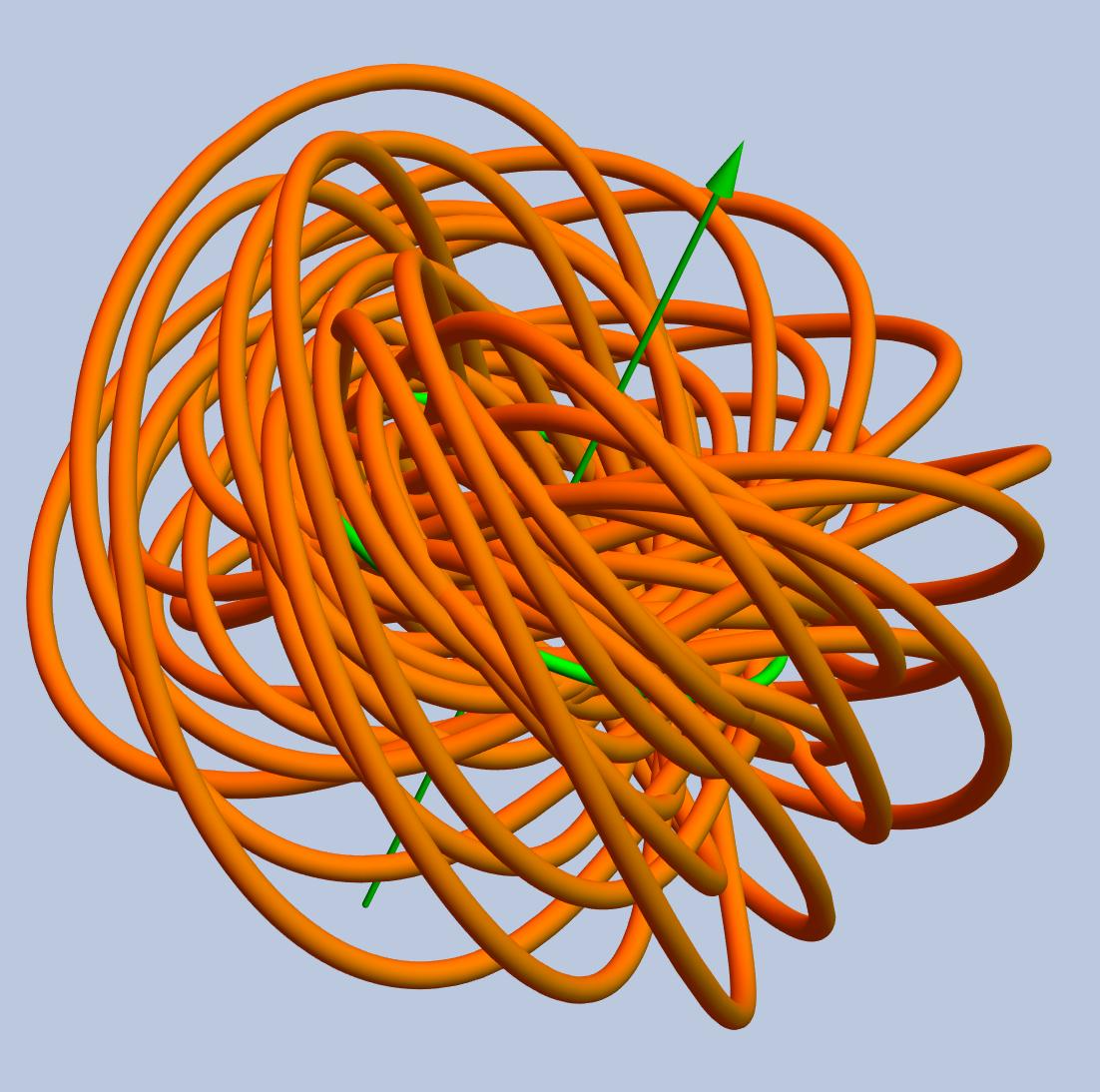}\qquad
\includegraphics[height=6cm,width=6cm]{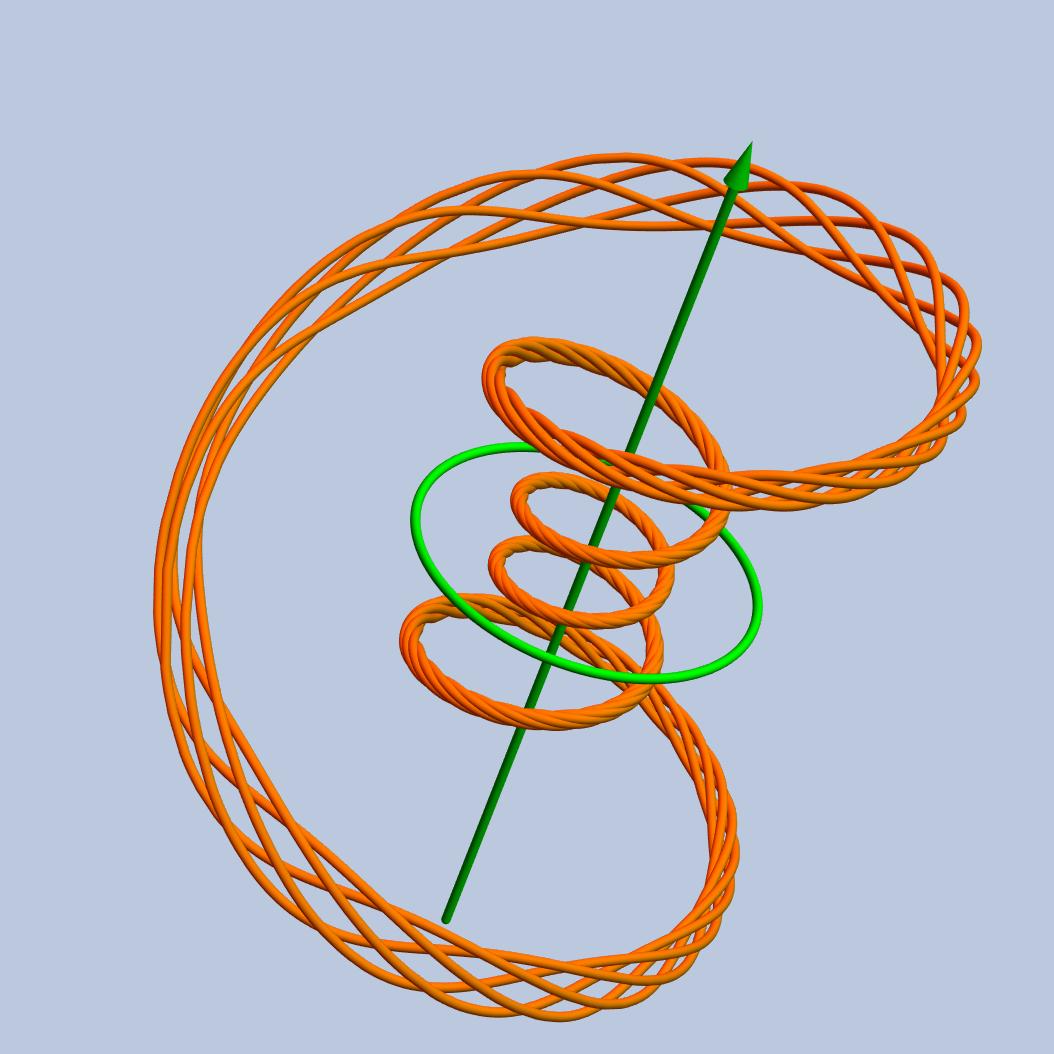}
\caption{On the left: the Heisenberg projection of a {standard}
configuration of a critical curve of type~${\mathcal B'}_{1}^{-}$,
with quantum numbers $q_1=5/49$, $q_2=-4/7$. On the right: the Heisenberg projection of a~standard
configuration of a critical
curve of type~${\mathcal B'}_{1}^{-}$, with quantum numbers $q_1=-7/36$, $q_2=-23/54$.}\label{FIG14}
\end{figure}
 \end{Example}

\begin{Remark}
It is clear that, being numerical approximations, the parametrizations obtained with this procedure are
only approximately periodic.
\end{Remark}

\subsection*{Acknowledgements}

The authors were partially supported by PRIN 2017 ``Real and Complex Manifolds: Topology, Geometry and holomorphic dynamics''
(protocollo 2017JZ2SW5-004); and by the GNSAGA of INdAM.
The present research was also partially supported by MIUR grant ``Dipartimenti di Eccellenza'' 2018-2022,
CUP: E11G18000350001, DISMA, Politecnico di Torino.
The authors gratefully acknowledge the referees for their helpful comments and suggestions.

\pdfbookmark[1]{References}{ref}
\LastPageEnding


\begin{thebibliography}{99}
\footnotesize\itemsep=0pt

\bibitem{Benn1983}
Bennequin D., Entrelacements et \'equations de {P}faff, \textit{Ast\'erisque}
 \textbf{107} (1983), 87--161.

\bibitem{Bryant1987}
Bryant R.L., On notions of equivalence of variational problems with one
 independent variable, in Differential {G}eometry: the {I}nterface {B}etween
 {P}ure and {A}pplied {M}athematics ({S}an {A}ntonio, {T}ex., 1986),
 \textit{Contemp. Math.}, Vol.~68, \href{https://doi.org/10.1090/conm/068/924805}{American Mathematical Society}, Providence,
 RI, 1987, 65--76.

\bibitem{COST}
Calabi E., Olver P.J., Shakiban C., Tannenbaum A., Haker S., Differential and
 numerically invariant signature curves applied to object recognition,
 \href{https://doi.org/10.1023/A:1007992709392}{\textit{Int.~J.~Comput. Vis.}} \textbf{26} (1998), 107--135.

\bibitem{CI}
Calini A., Ivey T., Integrable geometric flows for curves in pseudoconformal
 {$S^3$}, \href{https://doi.org/10.1016/j.geomphys.2021.104249}{\textit{J.~Geom. Phys.}} \textbf{166} (2021), 104249, 17~pages.

\bibitem{Cartan1932-2}
Cartan E., Sur la g\'eom\'etrie pseudo-conforme des hypersurfaces de l'espace
 de deux variables complexes~{II}, \textit{Ann. Scuola Norm. Super. Pisa Cl.
 Sci.~(2)} \textbf{1} (1932), 333--354.

\bibitem{Cartan1932}
Cartan E., Sur la g\'eom\'etrie pseudo-conforme des hypersurfaces de l'espace
 de deux variables complexes, \href{https://doi.org/10.1007/BF02417822}{\textit{Ann. Mat. Pura Appl.}} \textbf{11}
 (1933), 17--90.

\bibitem{ChMo1974}
Chern S.S., Moser J.K., Real hypersurfaces in complex manifolds, \href{https://doi.org/10.1007/BF02392146}{\textit{Acta
 Math.}} \textbf{133} (1974), 219--271.

\bibitem{Chiu-Ho2019}
Chiu H.-L., Ho P.T., Global differential geometry of curves in three-dimensional
 {H}eisenberg group and {CR} sphere, \href{https://doi.org/10.1007/s12220-018-00122-x}{\textit{J.~Geom. Anal.}} \textbf{29}
 (2019), 3438--3469.

\bibitem{DMN}
Dzhalilov A., Musso E., Nicolodi L., Conformal geometry of timelike curves in
 the {$(1+2)$}-{E}instein universe, \href{https://doi.org/10.1016/j.na.2016.05.011}{\textit{Nonlinear Anal.}} \textbf{143}
 (2016), 224--255, \href{https://arxiv.org/abs/1603.01035}{arXiv:1603.01035}.

\bibitem{Eliash1993}
Eliashberg Y., Legendrian and transversal knots in tight contact
 {$3$}-manifolds, in Topological {M}ethods in {M}odern {M}athematics ({S}tony
 {B}rook, {NY}, 1991), Publish or Perish, Houston, TX, 1993, 171--193.

\bibitem{EMN-JMAA}
Eshkobilov O., Musso E., Nicolodi L., The geometry of conformal timelike
 geodesics in the {E}instein universe, \href{https://doi.org/10.1016/j.jmaa.2020.124730}{\textit{J.~Math. Anal. Appl.}}
 \textbf{495} (2021), 124730, 32~pages.

\bibitem{Et3}
Etnyre J.B., Introductory lectures on contact geometry,
in Topology and {G}eometry of {M}anifolds ({A}thens, {GA}, 2001),
\textit{Proc. Sympos. Pure Math.}, Vol.~71, \href{https://doi.org/10.1090/pspum/071/2024631}{American Mathematical Society}, Providence, RI, 2003, 81--107, \href{https://arxiv.org/abs/math.SG/0111118}{arXiv:math.SG/0111118}.

\bibitem{Etn1999}
Etnyre J.B., Transversal torus knots, \href{https://doi.org/10.2140/gt.1999.3.253}{\textit{Geom. Topol.}} \textbf{3} (1999),
 253--268, \href{https://arxiv.org/abs/math.GT/9906195}{arXiv:math.GT/9906195}.

\bibitem{Et2}
Etnyre J.B., Legendrian and transversal knots, in Handbook of {K}not {T}heory,
 \href{https://doi.org/10.1016/B978-044451452-3/50004-6}{Elsevier}, Amsterdam, 2005, 105--185, \href{https://arxiv.org/abs/math.SG/0306256}{arXiv:math.SG/0306256}.

\bibitem{EtHo}
Etnyre J.B., Honda K., Knots and contact geometry~{I}: {T}orus knots and the
 figure eight knot, \href{https://doi.org/10.4310/JSG.2001.v1.n1.a3}{\textit{J.~Symplectic Geom.}} \textbf{1} (2001), 63--120,
 \href{https://arxiv.org/abs/math.GT/0006112}{arXiv:math.GT/0006112}.

\bibitem{FelsOlver1}
Fels M., Olver P.J., Moving coframes:~{I}. {A}~practical algorithm,
 \href{https://doi.org/10.1023/A:1005878210297}{\textit{Acta Appl. Math.}} \textbf{51} (1998), 161--213.

\bibitem{FelsOlver2}
Fels M., Olver P.J., Moving coframes:~{II}. {R}egularization and theoretical
 foundations, \href{https://doi.org/10.1023/A:1006195823000}{\textit{Acta Appl. Math.}} \textbf{55} (1999), 127--208.

\bibitem{FuTa1997}
Fuchs D., Tabachnikov S., Invariants of {L}egendrian and transverse knots in
 the standard contact space, \href{https://doi.org/10.1016/S0040-9383(96)00035-3}{\textit{Topology}} \textbf{36} (1997), 1025--1053.

\bibitem{Gr}
Griffiths P.A., Exterior differential systems and the calculus of variations,
 \textit{Progr. Math.}, Vol.~25, \href{https://doi.org/10.1007/978-1-4615-8166-6}{Birkh\"auser}, Boston, MA, 1983.

\bibitem{Ho}
Hoffman W.C., The visual cortex is a contact bundle, \href{https://doi.org/10.1016/0096-3003(89)90091-X}{\textit{Appl. Math.
 Comput.}} \textbf{32} (1989), 137--167.

\bibitem{Hsu}
Hsu L., Calculus of variations via the {G}riffiths formalism,
 \href{https://doi.org/10.4310/jdg/1214453181}{\textit{J.~Differential Geom.}} \textbf{36} (1992), 551--589.

\bibitem{Jacobo1985}
Jacobowitz H., Chains in {CR} geometry, \href{https://doi.org/10.4310/jdg/1214439561}{\textit{J.~Differential Geom.}}
 \textbf{21} (1985), 163--194.

\bibitem{K}
Klein F., Vorlesungen \"uber das {I}kosaeder und die {A}ufl\"osung der
 {G}leichungen vom f\"unften {G}rade, Birkh\"auser, Basel, 1993.

\bibitem{KRV}
Kogan I.A., Ruddy M., Vinzant C., Differential signatures of algebraic curves,
 \href{https://doi.org/10.1137/19M1242859}{\textit{SIAM~J. Appl. Algebra Geom.}} \textbf{4} (2020), 185--226,
 \href{https://arxiv.org/abs/1812.11388}{arXiv:1812.11388}.

\bibitem{M}
Musso E., Liouville integrability of a variational problem for {L}egendrian
 curves in the three-dimensional sphere, in Selected {T}opics in
 {C}auchy--{R}iemann {G}eometry, \textit{Quad. Mat.}, Vol.~9, Seconda
 Universit\`a di Napoli, Caserta, 2001, 281--306.

\bibitem{GM}
Musso E., Grant J.D.E., Coisotropic variational problems, \href{https://doi.org/10.1016/j.geomphys.2003.10.005}{\textit{J.~Geom.
 Phys.}} \textbf{50} (2004), 303--338, \href{https://arxiv.org/abs/math.DG/0307216}{arXiv:math.DG/0307216}.

\bibitem{MN-CQG}
Musso E., Nicolodi L., Closed trajectories of a particle model on null curves
 in anti-de {S}itter 3-space, \href{https://doi.org/10.1088/0264-9381/24/22/005}{\textit{Classical Quantum Gravity}} \textbf{24}
 (2007), 5401--5411, \href{https://arxiv.org/abs/0709.2017}{arXiv:0709.2017}.

\bibitem{MN-SIAM}
Musso E., Nicolodi L., Reduction for constrained variational problems on
 3-dimensional null curves, \href{https://doi.org/10.1137/070686470}{\textit{SIAM~J. Control Optim.}} \textbf{47}
 (2008), 1399--1414.

\bibitem{MNJMIV}
Musso E., Nicolodi L., Invariant signatures of closed planar curves,
 \href{https://doi.org/10.1007/s10851-009-0155-0}{\textit{J.~Math. Imaging Vision}} \textbf{35} (2009), 68--85.

\bibitem{MN-CAG}
Musso E., Nicolodi L., Quantization of the conformal arclength functional on
 space curves, \href{https://doi.org/10.4310/CAG.2017.v25.n1.a7}{\textit{Comm. Anal. Geom.}} \textbf{25} (2017), 209--242,
 \href{https://arxiv.org/abs/1501.04101}{arXiv:1501.04101}.

\bibitem{MNS-Kharkiv}
Musso E., Nicolodi L., Salis F., On the {C}auchy--{R}iemann geometry of
 transversal curves in the 3-sphere, \href{https://doi.org/10.15407/mag16.03.312}{\textit{J.~Math. Phys. Anal. Geom.}}
 \textbf{16} (2020), 312--363, \href{https://arxiv.org/abs/2004.11350}{arXiv:2004.11350}.

\bibitem{MS}
Musso E., Salis F., The {C}auchy--{R}iemann strain functional for {L}egendrian
 curves in the 3-sphere, \href{https://doi.org/10.1007/s10231-020-00974-7}{\textit{Ann. Mat. Pura Appl.}} \textbf{199} (2020),
 2395--2434, \href{https://arxiv.org/abs/2003.01713}{arXiv:2003.01713}.

\bibitem{Na}
Nash O., On {K}lein's icosahedral solution of the quintic, \href{https://doi.org/10.1016/j.exmath.2013.09.003}{\textit{Expo. Math.}}
 \textbf{32} (2014), 99--120, \href{https://arxiv.org/abs/1308.0955}{arXiv:1308.0955}.

\bibitem{Olver-book1}
Olver P.J., Applications of {L}ie groups to differential equations,
 \textit{Grad. Texts in Math.}, Vol. 107, \href{https://doi.org/10.1007/978-1-4612-4350-2}{Springer}, New York, 1993.

\bibitem{Olver-book2}
Olver P.J., Equivalence, invariants, and symmetry, \href{https://doi.org/10.1017/CBO9780511609565}{Cambridge University Press},
 Cambridge, 1995.

\bibitem{Pe}
Petitot J., Elements of neurogeometry, \textit{Lect. Notes Morphog.}, \href{https://doi.org/10.1007/978-3-319-65591-8}{Springer}, Cham,
 2017.

\bibitem{SP}
Storn R., Price K., Differential evolution~-- {A} simple and efficient
 heuristic for global optimization over continuous spaces, \href{https://doi.org/10.1023/A:1008202821328}{\textit{J.~Global
 Optim.}} \textbf{11} (1997), 341--359.

\bibitem{Tr}
Trott M., Adamchik V., Solving the quintic with \textsc{Mathematica}, available
 at \url{https://library.wolfram.com/infocenter/TechNotes/158/}.

\end{thebibliography}
\end{document}